\documentclass[12pt,reqno]{amsart}
\usepackage{hyperref}
\usepackage{geometry}
\usepackage{enumitem}
\usepackage{amsmath}
\usepackage{amssymb}
\usepackage{amsthm}
\usepackage{amsrefs}
\usepackage{amsfonts}
\usepackage[dvipsnames]{xcolor}
\usepackage{mathtools}
\usepackage{stackengine}
\usepackage{verbatim}
\usepackage{tikz-cd} 
\usepackage{float}

\makeatletter
\@namedef{subjclassname@2020}{%
  \textup{2020} MSC}
\makeatother

\hypersetup{
	colorlinks   = true, 
	urlcolor     = blue, 
	linkcolor    = purple, 
	citecolor   = blue 
}

\def\XXint#1#2#3{{\setbox0=\hbox{$#1{#2#3}{\int}$ }
\vcenter{\hbox{$#2#3$ }}\kern-.6\wd0}}

\makeatletter
\newcommand*{\rom}[1]{\expandafter\@slowromancap\romannumeral #1@}

\newcommand{\ind}{\protect\raisebox{2pt}{$\chi$}}

\newcommand{\disp}{\operatorname{Error}}
\newcommand{\prim}{\mathrm{prim}}

\newcommand{\SL}{\mathrm{SL}}
\newcommand{\sspan}{\operatorname{span}}

\newcommand{\GL}{\mathrm{GL}}
\newcommand{\M}{\mathrm{M}}

\newcommand{\X}{\mathcal{X}}
\newcommand{\XA}{\mathcal{X}^{\mathbb{A}}}
\newcommand{\Y}{\mathcal{Y}}
\newcommand{\E}{\mathcal{E}}
\newcommand{\R}{\mathbb{R}}
\newcommand{\Q}{\mathbb{Q}}
\newcommand{\T}{\mathbb{T}}
\newcommand{\B}{\mathcal{B}}
\newcommand{\e}{\varepsilon}
\newcommand{\Ortho}{\mathrm{O}}

\newcommand{\A}{\mathbb{A}}

\newcommand{\Z}{\mathbb{Z}}
\newcommand{\hZ}{\widehat{\mathbb{Z}}}
\newcommand{\hZp}{\widehat{\mathbb{Z}}_{\prim}}

\newcommand{\Sphere}{\mathbb{S}}

\newcommand{\N}{\mathbb{N}}
\newcommand{\Mat}{\M_{m \times n}(\R)}

\newcommand{\sed}{\mathcal{S}_{\epsilon}}
\newcommand{\ssed}{{\mathcal{S}^\sharp_{\epsilon}}}
\newcommand{\ued}{{\mathcal{U}_{\epsilon}}}
\newcommand{\ueda}{{\Tilde{\mathcal{U}}_{\epsilon}}}
\newcommand{\seda}{\Tilde{\mathcal{S}}_{\epsilon}}
\newcommand{\sseda}{{\Tilde{\mathcal{S}}^\sharp_{\epsilon}}}
\newcommand{\Le}{L_{\epsilon}}

\newcommand{\tmu}{\tilde{\mu}}
\newcommand{\tnu}{\tilde{\nu}}

\newcommand{\bfe}{\mathbf{e}}
\newcommand{\cN}{\mathcal{N}}
\newcommand{\el}{{\ell^\infty}}
\newcommand{\proj}{\operatorname{Proj}}

\newcommand{\BA}{\B^\A}
\newcommand{\Primes}{\mathbb{P}}

\newcommand{\ZZ}{\mathcal{Z}}
\renewcommand{\vector}{\mathfrak{v}}
\newcommand{\tvector}{\tilde{\mathfrak{v}}}
\newcommand{\BB}{\mathbb{B}}
\newcommand{\radius}{\mathfrak{r}}

\title{Generalized L\'{e}vy-Khintchine Theorems and a Conjecture of Y. Cheung}

\begin{document}
\theoremstyle{plain}
\newtheorem{thm}{Theorem}[section]
\newtheorem{lem}[thm]{Lemma}
\newtheorem{prop}[thm]{Proposition}
\newtheorem{cor}[thm]{Corollary}
\newtheorem{question}{Question}
\newtheorem{con}{Conjecture}
\theoremstyle{definition}
\newtheorem{defn}[thm]{Definition}
\newtheorem{exm}[thm]{Example}
\newtheorem{nexm}[thm]{Non Example}
\newtheorem{prob}[thm]{Problem}
\newtheorem{result}{Result}

\theoremstyle{remark}
\newtheorem{rem}[thm]{Remark}

\author{Gaurav Aggarwal}
\address{\textbf{Gaurav Aggarwal} \\
School of Mathematics,
Tata Institute of Fundamental Research, Mumbai, India 400005}
\email{gaurav@math.tifr.res.in}

\author{Anish Ghosh}
\address{\textbf{Anish Ghosh} \\
School of Mathematics,
Tata Institute of Fundamental Research, Mumbai, India 400005}
\email{ghosh@math.tifr.res.in}

\date{}

\thanks{ A.\ G.\ gratefully acknowledges support from a grant from the Infosys foundation to the Infosys Chandrasekharan Random Geometry Centre. G. \ A.\ and  A.\ G.\ gratefully acknowledge a grant from the Department of Atomic Energy, Government of India, under project $12-R\&D-TFR-5.01-0500$. }



\begin{abstract}  
The celebrated L\'evy--Khintchine theorem is a fundamental limiting law that describes the growth rate of the denominators of the convergents in the continued fraction expansion of a Lebesgue-typical real number. In a recent breakthrough, Cheung and Chevallier \textit{(Annales scientifiques de l’ENS, 2024)} extended this theorem to higher dimensions.

In this paper, we resolve a conjecture of Y.~Cheung and answer a question of Cheung and Chevallier concerning L\'evy--Khintchine type theorems for arbitrary norms. We also establish a higher-dimensional analogue of the Doeblin--Lenstra law.

While our results are new in higher dimensions, they also yield significant improvements in the classical one-dimensional setting. Specifically, we revisit the L\'evy--Khintchine theorem and the Doeblin--Lenstra law through the lens of Mahler’s influential proposal to study Diophantine approximation on fractals. In particular, we prove these results for almost every point on the middle-third Cantor set. More broadly, our framework applies to a wide class of measures, including those supported on curves and on self-similar fractals generated by iterated function systems (IFS), and it also allows constraints on the selection of best approximates.
\end{abstract}

\maketitle

\tableofcontents

\section{Introduction}
For a real number $\theta$, we denote its continued fraction expansion as usual:
\begin{align*}
    \theta = a_0+ \frac{1}{a_1 + \frac{1}{a_2 + \frac{1}{a_3 + \frac{1}{\ddots}}}}, 
\end{align*}
where
\begin{align*}
    \frac{p_l}{q_l}= a_0+ \frac{1}{a_1 + \frac{1}{a_2 + \frac{1}{ \ddots + \frac{1}{a_l}}}}
\end{align*}
\noindent denotes the $l$-th convergent. The famous L\'{e}vy-Khintchine Theorem tells us that the denominators $(q_l)_{l \geq 0}$ of the convergents of the continued fraction expansions of almost all real numbers $\theta$ satisfy
$$
    \lim_{l \rightarrow \infty} \frac{1}{l} \log (q_l) = \frac{\pi^2}{12 \log 2}.
$$

The existence of the limit was proved by Khintchine \cite{Khintchine36}, while its value was computed by L\'{e}vy \cite{Levy36}. Subsequently, there have been many significant developments; we refer the reader to \cite{CC19} for a nice account of these.\\

In \cite{Mahler}, K. Mahler initiated a major field of research by asking for typical Diophantine properties of measures supported on proper subsets of $\R$, for instance the middle-third Cantor set. Mahler's suggestion has led to an explosion of activity in a very challenging field. Recent highlights include the study of badly approximable numbers on fractals (\cites{SimmonsWeiss, AG24Random}), Khintchine's theorem on fractals (\cites{KhalilLuethi, dattajana24, benard24}) and singular vectors on fractals (\cites{aggarwal2025 ,Khalil, AG24singular}).  

In this vein, a natural question is whether the L\'{e}vy-Khintchine Theorem remains valid for almost every point with respect to a measure other than the Lebesgue measure, e.g. the middle-third Cantor measure. Specifically, 

\begin{question}
\label{Que: Cantor}
    Let $\mu$ denote the restriction of the $\log 2/ \log 3$-dimensional Hausdorff measure on the middle-third Cantor set. Then for $\mu$-almost every $\theta$, do the denominators $(q_n)_{n \geq 0}$ of the convergents of the continued fraction expansions of $\theta$ satisfy
$$
    \lim_{l \rightarrow \infty} \frac{1}{l} \log (q_l) = \frac{\pi^2}{12 \log 2}?
$$
\end{question}
We will answer this question affirmatively, and for a wide class of measures; see Corollary~\ref{thm:intro 1} and~\ref{cor: cantor}.
\vskip 0.4in

Another central problem in Diophantine approximation is to obtain higher dimensional analogues of classical Diophantine approximation results in the real line. One difficulty in doing so, is that there does not seem to be a canonical continued fraction expansion in higher dimensions. Since the study of L\'evy--Khintchine phenomena fundamentally relies on the properties of the convergents \((p_\ell, q_\ell)\) of a real number \(\theta\)---which are precisely its best approximations---it is natural to seek a suitable higher-dimensional analogue of ``best approximates.'' One such analogue was proposed by Y.~Cheung. To motivate this, let us first observe that the definition of the best approximation of $\theta \in \R$ can be interpreted geometrically as follows.

An integer vector \((p, q) \in \mathbb{Z} \times \mathbb{N}\) is a best approximation of \(\theta \in \mathbb{R}\) if and only if the only primitive vectors of the lattice
\[
\Lambda_\theta = \begin{pmatrix}
1 & \theta \\
0 & 1
\end{pmatrix} \mathbb{Z}^2
\]
that lie inside the rectangle
\[
[-p - \theta q,\, p + \theta q] \times [-q,\, q]
\]
are the endpoints \((p + \theta q,\, q)\) and \((-p - \theta q,\,-q)\); see Figure~\ref{fig:rectangle}.
\begin{figure}[H]
    \centering
     \includegraphics[width=0.3\textwidth]{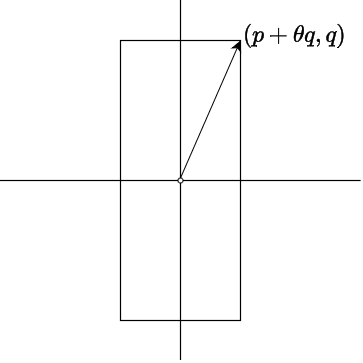}
    \caption{The rectangle \([-p - \theta q,\; p + \theta q] \times [-q,\; q]\) used to characterize best approximations in dimension 1.}
    \label{fig:rectangle}
\end{figure}

This rectangular region plays a fundamental geometric role in characterizing best approximations. In higher dimensions, a natural analogue of the rectangle is a \emph{hyper-cuboid} centered at the origin. Accordingly, a natural way to generalize the notion of a best approximation is as follows:

\begin{defn}
\label{def:best approx k=m, n=1}
Let \(\theta \in \R^m\). An integer vector \((p, q) \in \mathbb{Z}^m \times \mathbb{N}\) is said to be a \textbf{best approximation} of \(\theta\) if the only primitive vectors of the lattice
\[
\Lambda_\theta = \begin{pmatrix}
I_m & \theta \\
0 & 1
\end{pmatrix} \mathbb{Z}^{m+1}
\]
lying in the hyper-cuboid
\[
[-p_1 - \theta_1 q,\, p_1 + \theta_1 q] \times \cdots \times [-p_m - \theta_m q,\, p_m + \theta_m q] \times [-q,\, q]
\]
are the endpoints \((p + \theta q,\, q)\) and \((-p - \theta q,\,-q)\).
\end{defn}

\begin{figure}[H]
    \centering
    \includegraphics[width=0.5\textwidth]{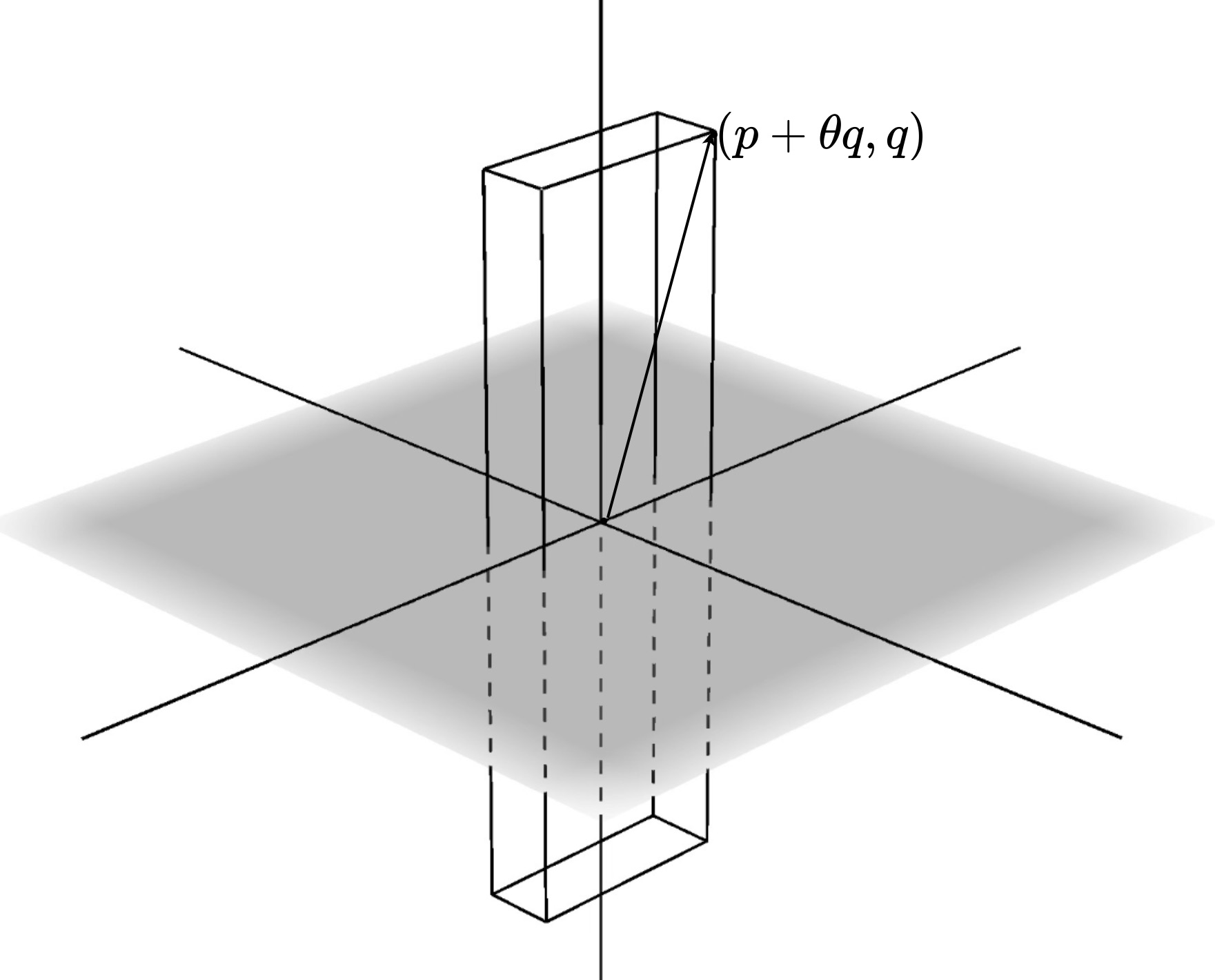}
    \caption{A schematic representation of the hyper-cuboid region in \(\mathbb{R}^{m+1}\) with \((p + \theta q, q)\) as one of its endpoints.}
    \label{fig:hypercuboid}
\end{figure}

With this definition, we have the following conjecture of Yitwah Cheung \footnote{Private Communication.}.

\begin{con} (Cheung's conjecture)
    \label{thm: Cheung Conjecture}
    There exists a constant $c \in (0, \infty)$ such that the following holds for Lebesgue almost every $\theta \in \R^m$. If $(p_l,q_l)$ denote the best approximates of $\theta$ (as in Definition \ref{def:best approx k=m, n=1}), then we have
    \begin{align}
        \label{eq:Cheung Conjecture}
        \lim_{l \rightarrow \infty} \frac{(\log q_l)^m}{l} = c. 
    \end{align}
\end{con}
In this paper, we will resolve this conjecture; see Remark \ref{rem: cheung conjecture}. \vskip 0.4in

Another natural generalization of the notion of best approximations in higher dimensions was introduced by Cheung and Chevallier \cite{CC19}, who defined the best Diophantine approximations for matrices as follows.
\begin{defn}
\label{def: best k=r=1}
    Let $m,n$ be positive integers. Fix norms $\|.\|_{\R^m}$ on $\R^m$ and $\|.\|_{\R^n}$ on $\R^n$. An integer vector $(p,q) \in \Z^m \times (\Z^n \setminus \{0\})$ is called a best approximation of $\theta \in \Mat$, if there is no integer solution $(p',q') \in \Z^m \times (\Z^n \setminus \{0\})$ other than $(p,q)$ and $(-p,-q)$ to the following inequalities
    \begin{align*}
        \|p'+ \theta q'\|_{\R^m} &\leq \|p+ \theta q\|_{\R^m}, \\
        \|q'\|_{\R^n} & \leq \|q\|_{\R^n}.
    \end{align*}
\end{defn}
\begin{rem}
\label{rem: difference}
    Note that above definition for $n=1$ is equivalent to following condition geometrically: An integer vector \((p, q) \in \mathbb{Z}^m \times \mathbb{N}\) is said to be a best approximation of \(\theta\) if the only primitive vectors of the lattice
\[
\Lambda_\theta = \begin{pmatrix}
I_m & \theta \\
0 & 1
\end{pmatrix} \mathbb{Z}^{m+1}
\]
lying in the hyper-cylinder
\[
\{x \in \R^m: \|x\|_{\R^m} \leq \|p+\theta q\|_{\R^m} \} \times [-q,\, q]
\]
are the endpoints \((p + \theta q,\, q)\) and \((-p - \theta q,\,-q)\).

\begin{figure}[H]
    \centering
    \includegraphics[width=0.5\textwidth]{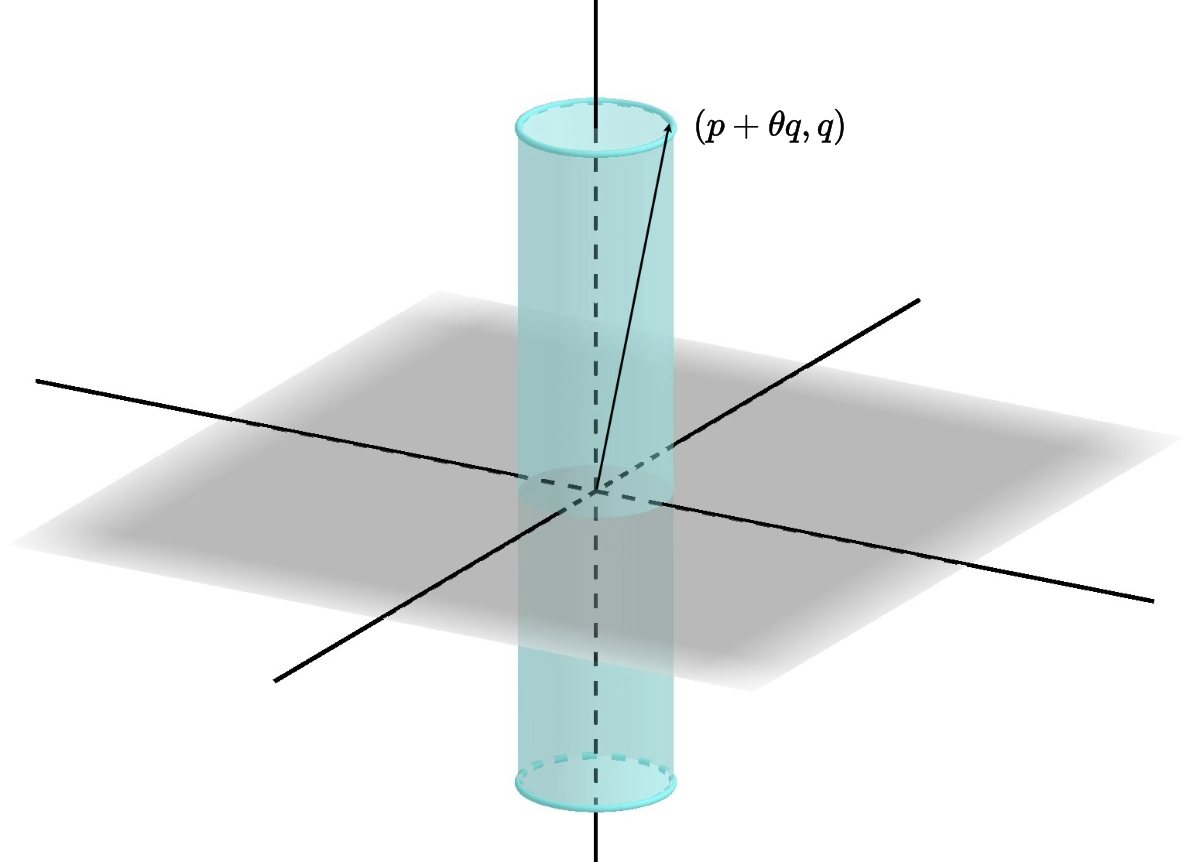}
    \caption{A schematic representation of the hyper-cylinder region in \(\mathbb{R}^{m+1}\) with \((p + \theta q, q)\) as one of its endpoints.}
    \label{fig:cylinder}
\end{figure}

Note that for the standard Euclidean norm \(\|\cdot\|_{\mathbb{R}^m}\) on \(\mathbb{R}^m\), the hypercuboid is always contained within the corresponding hypercylinder (having the same axis and endpoint). Consequently, the conditions in Definition~\ref{def:best approx k=m, n=1} are strictly weaker than those in Definition~\ref{def: best k=r=1}; see Appendix~\ref{appendix:best-approximation-remarks} for further discussion.

To visually compare the two, we include a top-down view of the respective regions below.
\begin{figure}[H]
    \centering
    \begin{minipage}{0.45\textwidth}
        \centering
        \begin{tikzpicture}
            \draw[thick] (0,0) circle (2);
            \filldraw[black] (0,0) circle (1pt);
            \draw[->,thick] (0,0) -- ({sqrt(3)},1);
            \node[right] at ({sqrt(3)},1) {$p + \theta q$};
        \end{tikzpicture} \\
        \textit{Top view of hyper-cylinder: isotropic in all directions}
    \end{minipage}
    \hfill
    \begin{minipage}{0.45\textwidth}
        \centering
        \begin{tikzpicture}
            \draw[thick] ({-sqrt(3)},1) rectangle ({sqrt(3)},-1);
            \draw[dashed] (0,0) circle (2);
            \filldraw[black] (0,0) circle (1pt);
            \draw[->,thick] (0,0) -- ({sqrt(3)},1);
            \node[right] at ({sqrt(3)},1) {$p + \theta q$};
        \end{tikzpicture} \\
        \textit{Top view of hyper-cuboid: bounded along coordinate axes}
    \end{minipage}
    \caption{Visual comparison between the cylinder (Cheung--Chevallier) and cuboid (Cheung's conjecture) conditions}
    \label{fig:geometric_interpretation_2}
\end{figure}

\end{rem}

In this generalization, Cheung and Chevallier \cite{CC19} proved an analogue of the L\'{e}vy–Khintchine theorem, which states the following.

\begin{thm}[{\cite[Theorems~1 and 2]{CC19}}]
\label{thm: CC19}
    Let $m,n$ be two positive integers. Suppose $\R^m$ and $\R^n$ are equipped with standard Euclidean norms $\|.\|_{\R^m}$ and $\|.\|_{\R^n}$ respectively. Then there exists a constant $L_{m,n}$ and a measure $\nu_{m,n}$ on $\R$ such that for Lebesgue almost every $\theta \in \Mat$, the sequence of best approximates $(p_l,q_l)$ of $\theta$ associated with the norms $\|.\|_{\R^m}$ and $\|.\|_{\R^n}$, satisfy
    \begin{align}
        \lim_{l \rightarrow \infty} \frac{1}{l} \log \|q_l\|_{\R^n} &= L_{m,n} \label{eq: CC19 1}\\
        \lim_{l \rightarrow \infty} \frac{1}{l} \log \|p_l + \theta q_l\|_{\R^m} &= -\frac{n}{m} L_{m,n}, \label{eq: CC19 2}\\
        \lim_{l \rightarrow \infty} \frac{1}{l} \sum_{i=1}^l \delta_{\|q_{i+1}\|_{\R^n}^n \|p_i + \theta q_i\|_{\R^m}^m} &= \nu_{m,n},\label{eq: CC19 3}
    \end{align}
    where the convergence of measure in last equality is given with respect to weak topology. Furthermore, the support of $\nu_{m,n}$ is contained in a bounded interval and contains $0$ provided that $m+n \geq 3$.  
\end{thm}

See \cite{cheung2021valuedimensionallevysconstant} for a numerical approximation of the constant in dimension~$2$ by the same authors. \\

The above very general theorem naturally gives rise to numerous open questions. A particularly natural one---posed as Question~1 in Section~10 of \cite{CC19} by Cheung and Chevallier---asks whether the result remains valid when the standard Euclidean norms on $\R^m$ and $\R^n$ are replaced with arbitrary norms. Another line of inquiry involves extending the result to cases where points in $\Mat$ are sampled from measures that are singular with respect to Lebesgue measure, such as Bernoulli measures on self-similar fractals or natural measures supported on embedded curves. A third direction concerns analogues of \eqref{eq: CC19 1}, \eqref{eq: CC19 2}, \eqref{eq: CC19 3} for best approximations $(p, q)$ that satisfy additional constraints---such as \emph{congruence} conditions or \emph{directional} restrictions.

In this paper, we answer all of these questions affirmatively; see Corollary~\ref{thm:intro 1}. We introduce a unified, higher-dimensional definition of best approximation (Definition~\ref{def: best general case}) that encompasses both Definitions~\ref{def:best approx k=m, n=1} and~\ref{def: best k=r=1}, and we establish a generalized version of the L\'evy--Khintchine theorem in this broader setting; see Corollary~\ref{thm: cor 3 to main thm}. Our main theorems are quite general: they provide precise asymptotics for the number of best approximations to a given matrix \(\theta\), whose denominators lie below a given threshold and satisfy additional geometric and arithmetic constraints; see Corollary~\ref{cor: Counting Best Approximates with Constraints}. As an application, Corollary~\ref{cor2 to main thm best} recovers and extends the main result of Shapira and Weiss~\cite[Thm.~1.1 \&~2.1]{SW22}, which concerned the joint equidistribution of best approximations, defined as in Definition~\ref{def: best k=r=1} in the case \(n = 1\).

\begin{rem}
Some readers have pointed out the differences between Definitions~\ref{def:best approx k=m, n=1} and~\ref{def: best k=r=1}—especially in the context of the $\ell^\infty$-norm—as a source of confusion. A related issue arises in the heuristic behind Cheung’s conjecture: namely, why the expression \((\log q_\ell)^m/\ell\) appears in one setting~\eqref{eq:Cheung Conjecture}, but only \(\log q_\ell/\ell\) in another~\eqref{eq: CC19 1}. These matters are clarified further in Appendix~\ref{appendix:best-approximation-remarks}.
\end{rem}

\begin{rem}
To keep the introduction concise, we have postponed the discussion of the Doeblin--Lenstra law to Section~\ref{subsec: Doeblin--Lenstra Law}, which also contains a proof of its generalisation to higher dimensions. For the generalisation of the Doeblin--Lenstra law to fractals and curves, see Corollaries~\ref{thm:intro 1} and~\ref{cor: cantor}.
\end{rem}

\begin{rem}
The authors wish to emphasize that Definition~\ref{def:best approx k=m, n=1} and Definition~\ref{def: best k=r=1} describe fundamentally different notions of best approximates. In fact, the former is significantly more difficult to study. The key reason lies in the dynamical systems naturally associated to each: the case of Definition~\ref{def: best k=r=1} reduces to the study of one-parameter flows on homogeneous spaces, whereas Definition~\ref{def:best approx k=m, n=1} requires dealing with multi-parameter flows.

A natural question for the reader is why the number of parameters in the flow should affect the difficulty. The reason is that the theory of cross-sections—an essential tool in homogeneous dynamics—is well developed in the one-parameter setting, owing to foundational work of Kakutani and Ambrose. In contrast, no such analogous theory existed for multi-parameter flows, primarily due to the absence of natural counterparts to first return maps and return time functions, which play a crucial role in the one-parameter setting. This fundamental obstruction hindered progress for a long time. Only recently have these challenges been systematically addressed in the authors' work~\cite{aggarwalghosh2024joint}, where a new framework for studying cross-sections in the multi-parameter setting was developed.

This disparity in dynamical complexity is also reflected in the literature. The relatively tractable case of Definition~\ref{def: best k=r=1} has been the focus of several recent works. For instance, Cheung and Chevallier~\cite{CC19} established a L\'evy--Khintchine-type limit law for best approximates in this setting (see Theorem~\ref{thm: CC19}). Shapira and Weiss~\cite{SW22} studied joint equidistribution phenomena (see Section~\ref{subsec: Joint Equidistribution of Best Approximates} for further discussion). The authors have also contributed two papers—\cite{AG25ELK}, which proves an effective L\'evy--Khintchine theorem, and~\cite{AG25DL}, which establishes an effective version of the Doeblin--Lenstra law—both in the setting of Definition~\ref{def: best k=r=1}.

By contrast, best approximates as in Definition~\ref{def:best approx k=m, n=1}, and more generally in Definition~\ref{def: best general case}, are systematically studied for the first time in this paper. Extending effective equidistribution results—and in particular, proving an effective L\'evy--Khintchine theorem—in this broader setting remains an intriguing direction for future research.
\end{rem}

We now outline the structure of the paper and indicate how the main results are developed.

\subsection*{Structure of the Paper}

We begin in Section~\ref{sec: main result} by introducing a unified, higher-dimensional notion of best approximation and stating the main result of the paper. Section~\ref{sec: Corollaries: General Case} discusses several consequences of this result, including a proof of Cheung’s Conjecture.

In Section~\ref{sec: Further refinements}, we consider the special case where best approximation is given by Definition~\ref{def: best k=r=1}, and refine the main result in this setting. Section~\ref{sec: Corollaries: Special Case $k=r=1$} derives additional corollaries for this special case, including a resolution of Question~\ref{Que: Cantor} and other questions posed after Theorem~\ref{thm: CC19}.

Section~\ref{sec: Notation} fixes the notational conventions used throughout the paper. Sections~\ref{sec: Results borrowed} and~\ref{sec: Properties of the Cross-section} are devoted to defining a suitable cross-section for a multi-parameter flow on the space of lattices, and establishing its key properties.

In Section~\ref{sec: Cross-section Correspondence}, we connect the dynamics on the constructed cross-section with Diophantine approximation. Specifically, we show that visits to the cross-section correspond to best approximates, and that equidistribution of best approximates is equivalent to genericity with respect to the cross-section. Section~\ref{sec: Cross-section Genericity} is dedicated to proving the equivalence of Birkhoff genericity and cross-section genericity.

Finally, Section~\ref{sec: Final proofs} concludes the paper by assembling the preceding results to prove Theorems~\ref{main thm best} and~\ref{Main thm time visits}, as well as Corollaries~\ref{thm:intro 1} and~\ref{cor: determinant}.

\section{Main Result}
\label{sec: main result}

We begin this section by introducing the aforementioned unified, higher-dimensional definition of best approximation, which simultaneously encompasses both Definitions~\ref{def:best approx k=m, n=1} and~\ref{def: best k=r=1}.

\subsection{Definitions}

For the remainder of the paper, we fix positive integers \( m \) and \( n \), and set \( d := m + n \). We fix a decomposition
\[
m = m_1 + \cdots + m_k, \qquad n = n_1 + \cdots + n_r,
\]
where \( m_i, n_j \in \mathbb{N} \) for all \( 1 \leq i \leq k \) and \( 1 \leq j \leq r \). We fix norms on each of the spaces \( \mathbb{R}^{m_1}, \ldots, \mathbb{R}^{m_k} \) and \( \mathbb{R}^{n_1}, \ldots, \mathbb{R}^{n_r} \), and—with slight abuse of notation—we denote each of these norms by \( \|\cdot\| \).

Using these, we define norms on \( \mathbb{R}^m \), \( \mathbb{R}^n \), and \( \mathbb{R}^d \) as follows:
\begin{align}
    \|(x_1, \ldots, x_k)\| &:= \max_i \|x_i\| \quad \text{for } (x_1, \ldots, x_k) \in \mathbb{R}^{m_1} \times \cdots \times \mathbb{R}^{m_k} = \mathbb{R}^m, \\
    \|y\| &:= \max_j \|\rho_j'(y)\| \quad \text{for } y \in \mathbb{R}^n, \\
    \|(x,y)\| &:= \max\{ \|x\|, \|y\| \} \quad \text{for } (x,y) \in \mathbb{R}^m \times \mathbb{R}^n = \mathbb{R}^d.
\end{align}

For \( \delta > 0 \), let \( B_\delta^l \) denote the closed ball of radius \( \delta \) centered at the origin in \( \mathbb{R}^l \), with respect to the chosen norm on \( \mathbb{R}^l \).

We define the projection maps \( \varrho_1 \) and \( \varrho_2 \) from \( \mathbb{R}^d = \mathbb{R}^m \times \mathbb{R}^n \) onto \( \mathbb{R}^m \) and \( \mathbb{R}^n \), respectively. Likewise, for each \( 1 \leq i \leq k \), let \( \rho_i \) denote the projection from \( \mathbb{R}^m = \mathbb{R}^{m_1} \times \cdots \times \mathbb{R}^{m_k} \) onto the \( i \)-th component \( \mathbb{R}^{m_i} \). Similarly, for each \( 1 \leq j \leq r \), let \( \rho_j' : \mathbb{R}^n = \mathbb{R}^{n_1} \times \cdots \times \mathbb{R}^{n_r} \to \mathbb{R}^{n_j} \) denote the projection onto the \( j \)-th component.
\begin{defn}
    \label{def: best general case}
    An integer vector \( (p, q) \in \mathbb{Z}^m \times (\mathbb{Z}^n \setminus \{0\}) \) is called a \emph{best approximation} of \( \theta \in \Mat \) if there are no other primitive vectors in the lattice
    \[
    \Lambda_\theta := \begin{pmatrix}
         I_m & \theta \\ 0 & I_n
     \end{pmatrix} \mathbb{Z}^{m+n}
    \]
    lying in the region
    \[
    B_{\|\rho_1(p+\theta q)\|}^{m_1} \times \cdots \times B_{\|\rho_k(p+\theta q)\|}^{m_k} \times B_{\|\rho_1'(q)\|}^{n_1} \times \cdots \times B_{\|\rho_r'(q)\|}^{n_r}
    \]
    other than \( \pm(p + \theta q, q) \).
\end{defn}

\begin{rem}
    The above definition of best approximation depends on the chosen decomposition \( m = m_1 + \cdots + m_k \), \( n = n_1 + \cdots + n_r \), as well as on the norms fixed on each of the corresponding subspaces. Throughout the paper, unless stated otherwise, we refer to best approximations with respect to this fixed decomposition and choice of norms.
\end{rem}

\begin{rem}
    If we take \( m = 1 + \cdots + 1 \) and \( n = 1 \) (i.e., \( k = m \), \( r = n = 1 \)), then the above definition coincides with Definition~\ref{def:best approx k=m, n=1}. On the other hand, if we use the trivial decomposition \( m = m \), \( n = n \) (i.e., \( k = r = 1 \)), then it reduces to Definition~\ref{def: best k=r=1}. Thus, the general definition presented here unifies both of these as special cases.
\end{rem} \vspace{0.15in}

\subsection{Natural Objects to Consider}

To each best approximation $(p,q)$ of $\theta \in \Mat$, we associate the following natural objects, motivated by our earlier work \cite{aggarwalghosh2024joint} on $\e$-approximations.

\subsubsection{The Error Term}

The first such object is the volume of the hyper-cuboid
\[
B_{\|\rho_1(p+\theta q)\|}^{m_1} \times \cdots \times B_{\|\rho_k(p+\theta q)\|}^{m_k} \times B_{\|\rho_1'(q)\|}^{n_1} \times \cdots \times B_{\|\rho_r'(q)\|}^{n_r}.
\]
Up to a constant factor, this volume equals the quantity $\disp(\theta, p,q)$, defined as
\begin{align}
\label{defdisp}
\disp(\theta, p,q) = \left( \prod_{i=1}^k \|\rho_i(p+\theta q)\|^{m_i} \right) \cdot \left( \prod_{j=1}^r \|\rho_j'(q)\|^{n_j} \right),
\end{align}
for every $\theta \in \Mat$ and $(p,q) \in \mathbb{Z}^m \times \mathbb{Z}^n$.

\medskip

By Minkowski’s second theorem and the fact that the lattice $\Lambda_\theta$ has no nonzero points in the interior of the above hyper-cuboid, it follows that there exists a constant $\e > 0$ such that
\[
\disp(\theta, p,q) < \e
\]
for all $\theta$ and corresponding best approximations $(p,q)$. We fix such an $\e$ for the remainder of the paper.

\subsubsection{Projections}

The second natural object associated with a Diophantine approximation $(p,q)$ is the direction of the vector $(p + \theta q, q)$. This corresponds to the projection of the vector onto the product of unit spheres; see, for instance,~\cite{AGT}. Explicitly, this projection is given by
\begin{align}
\label{defproj}
\left( 
\frac{\rho_1(p + \theta q)}{\|\rho_1(p + \theta q)\|}, \ldots, 
\frac{\rho_k(p + \theta q)}{\|\rho_k(p + \theta q)\|}, 
\frac{\rho_1'(q)}{\|\rho_1'(q)\|}, \ldots, 
\frac{\rho_r'(q)}{\|\rho_r'(q)\|} 
\right).
\end{align}

For ease of reference in later sections, we define the map
\[
\proj: \Mat \times \Z^m \times \Z^n \setminus \disp^{-1}(0) \longrightarrow \Sphere^{m_1} \times \cdots \times \Sphere^{m_k} \times \Sphere^{n_1} \times \cdots \times \Sphere^{n_r}
\]
to be the function that assigns to each triple $(\theta, p, q)$ the corresponding element in~\eqref{defproj}. Here, $\Sphere^\ell$ denotes the unit sphere in $\R^\ell$, with respect to the norm chosen on $\R^\ell$.

\subsubsection{The Congruence Condition}

In Diophantine approximation, it is often natural to restrict attention to a subclass of approximates that satisfy certain congruence conditions. This leads to the study of the distribution of approximates $(p,q)$ in $(\Z/N\Z)^d$ for all $N \in \N$, which can equivalently be phrased as the distribution of $(p,q)$ in $\widehat{\Z}^d$, the $d$-fold Cartesian product of $\widehat{\Z}$ (the profinite completion of $\Z$). Accordingly, the third natural object to consider is the element $(p,q) \in \widehat{\Z}_\prim^d$, where $\widehat{\Z}_\prim^d$ denotes the closure of the image of $\Z_\prim^d$ under the canonical diagonal embedding $\Z^d \hookrightarrow \widehat{\Z}^d$.

We refer the reader to~\cite{AlamGhoshYu} for Diophantine approximation with congruence conditions, as well as \cite{SW22}; and \cite{AG23} and \cite{AlamGhoshHan} for counting problems with congruence conditions.

\subsubsection{The Relative Size}
\label{sec: The Relative Size}

The three objects discussed above capture much of the structure associated with best approximates. However, an important feature still missing is the \emph{relative size} of the quantities
\[
\|\rho_1(p+\theta q)\|, \ldots, \|\rho_k(p+\theta q)\|, \|\rho_1'(q)\|, \ldots, \|\rho_r'(q)\|.
\]
We encode this information into a suitable unimodular lattice. More concretely, given $\theta \in \Mat$, define the unimodular lattice
\begin{align}
    \label{eq:def lambda theta}
    \Lambda_\theta = \begin{pmatrix}
        I_m & \theta \\ & I_n
    \end{pmatrix} \Z^d.
\end{align}
Clearly, $\Lambda_\theta$ contains the vector $(p+\theta q, q)$ as a primitive vector. Now consider the lattice
\[
\Lambda_\theta (p,q) = \begin{pmatrix}
    \|\rho_1(p+ \theta q)\|^{-1} I_{m_1} \\ & \ddots \\ && \|\rho_r'(q)\|^{-1}I_{n_r}
\end{pmatrix} \Lambda_\theta.
\]
This lattice is obtained by scaling $\Lambda_\theta$ so that the vector $(p+\theta q, q)$ maps to $\proj(\theta, p, q)$. Note that $\Lambda_\theta(p,q)$ is not unimodular—it has co-volume $\disp(\theta,p,q)^{-1}$—but it still encodes the magnitudes of the projected components. Specifically, for each $1 \leq i \leq k$, the quantity $\|\rho_i(p+ \theta q)\|^{-m_i}$ equals the co-volume of the $m_i$-dimensional lattice
\[
\Lambda_\theta(p,q) \cap \{0\}^{m_1 + \cdots + m_{i-1}} \times \R^{m_i} \times \{0\}^{m_{i+1} + \cdots + m_k + n}
\]
in its ambient vector space. 

Importantly, this lattice also contains the information of previously defined objects. Concretely, $\Lambda_\theta(p,q)$ has co-volume $\disp(\theta, p, q)^{-1}$ and contains a point in $\Sphere^{m_1} \times \cdots \times \Sphere^{n_r}$, namely $\proj(\theta, p, q)$. The fourth object we associate with an best approximate is a derived lattice that retains the relative size information while removing redundancy.

To formalize this, denote by $\bfe_1, \ldots, \bfe_d$ the standard basis vectors of $\R^d$. For $1 \leq j \leq d$, let $\E_d^j$ denote the space of unimodular lattices in $\R^d$ that contain $\bfe_j$ as a primitive vector.

We define a map
\[
\lambda_j: \left(\Mat  \times (\Z^m \times \Z^n)_\prim\right) \setminus \left(\disp^{-1}(0) \cup \left\{(\theta,p,q): (p+\theta q, q)_j=0\right\} \right) \rightarrow \E_d^j
\]
by setting
\[
\lambda_j(\theta, p, q) := A_j(\proj(\theta,p,q), \disp(\theta,p,q))\, \Lambda_\theta(p,q),
\]
where $A_j(x, \gamma)$ is the unique matrix in $\GL_d(\R)$ with determinant $\gamma$ that maps $x$ to $\pm \bfe_j$ (with sign matching the $j$-th coordinate of $x$) and acts on the orthogonal complement $\sspan(\bfe_1, \ldots, \bfe_{j-1}, \bfe_{j+1}, \ldots, \bfe_d)$ by positive scalar multiplication. 

Thus, the fourth object we associate is $\lambda_j(\theta, p, q) \in \E_d^j$, for any fixed $1 \leq j \leq d$.

\begin{rem}
    Note that different values of $j$ become more natural for various values of $m,n$. In particular, for $n=1$, the choice $j=d$ is most natural and was considered in \cite{SW22}.
\end{rem} \vspace{0.15in}

\subsection{Natural Measures}
\label{subsec:Natural measures}

Fix an integer $1 \leq j \leq d$. For notational convenience, we define the space
\begin{align}
    \label{eq: def z j}
    \ZZ_j := \E_d^j \times \Sphere^{m_1} \times \cdots \times \Sphere^{m_k} \times \Sphere^{n_1} \times \cdots \times \Sphere^{n_r} \times [0,\e] \times \hZp^d.
\end{align}
We also define the map
\[
\Theta_j: \left(\Mat \times (\Z^m \times \Z^n)_\prim\right) \setminus \left(\disp^{-1}(0) \cup \left\{(\theta, p, q) : (p+\theta q, q)_j = 0\right\} \right) \longrightarrow \ZZ_j
\]
by
\begin{align}
    \label{eq: def Xi}
    \Theta_j(\theta, p, q) := \left( \lambda_j(\theta, p, q),\, \proj(\theta, p, q),\, \disp(\theta, p, q),\, (p, q) \right).
\end{align}

The space $\ZZ_j$ is a product of several geometric spaces, each of which carries a natural measure. We now describe these component measures, and hence the product measure on $\ZZ_j$.

\begin{enumerate}
    \item[\textbf{(i)}] The space $\E_d^j$ can be identified with a homogeneous space of the group
    \[
    H^j := \{h \in \SL_d(\R) : h \cdot \bfe_j = \bfe_j\}.
    \]
    It carries a unique $H^j$-invariant probability measure, denoted by $m_{\E_d^j}$.

    \item[\textbf{(ii)}] For each $l \in \{m_1, \ldots, m_k, n_1, \ldots, n_r\}$, we define a measure $\mu^{(\Sphere^l)}$ on $\Sphere^l$ as the pushforward of the Lebesgue measure $m_{\R^l}$ restricted to the unit ball $\{x \in \R^l : \|x\| \leq 1\}$ under the projection map $x \mapsto x/\|x\|$. Note that $\mu^{(\Sphere^l)}$ is not necessarily a probability measure.

    \item[\textbf{(iii)}] On the interval $(0, \e)$, we consider the Lebesgue measure restricted to this interval, denoted $m_{\R}|_{[0,\e]}$.

    \item[\textbf{(iv)}] The set $\hZp^d$ is an orbit under the natural action of $\SL_d(\hZ)$ on $\hZ^d$. We endow it with the unique $\SL_d(\hZ)$-invariant probability measure, denoted $m_{\hZp^d}$.
\end{enumerate}

The product of these component measures defines a natural measure on the space $\ZZ_j$, denoted
\begin{align}
    \label{eq:def tmu j}
    \tmu^j := m_{\E_d^j} \times \mu^{(\Sphere^{m_1})} \times \cdots \times \mu^{(\Sphere^{m_k})} \times \mu^{(\Sphere^{n_1})} \times \cdots \times \mu^{(\Sphere^{n_r})} \times m_{\R}|_{[0,\e]} \times m_{\hZp^d}.
\end{align}

Finally, we define $\mu^j$ to be the probability measure obtained by normalizing $\tmu^j$. \vspace{0.15in}

\subsection*{Main Theorem} 

The main theorem of the paper is the following:

\begin{thm}
    \label{main thm best}
    Fix an integer $1 \leq j \leq d$. Then, there exists a measure $\tnu^j$ on $\ZZ_j$ (defined in~\eqref{eq: def z j}) such that for Lebesgue almost every $\theta \in \Mat$, the following holds. Let $(p_l, q_l) \in \Z^m \times \Z^n$ be the sequence of best approximations to $\theta$ (defined in~\ref{def: best general case}), ordered according to increasing $\|q_l\|$. Then, for all bounded continuous functions $f$ on the space $\ZZ_j$, we have:
    \begin{equation}
    \label{eq: main thm general 2}
        \lim_{T \rightarrow \infty} \frac{1}{T^{k+r-1}} \sum_{ \|q_l\| \leq e^T} f\left(\Theta_j(\theta,p_l,q_l)\right) = \frac{ c_{k+r-1}(n_1, \ldots, n_r)}{ (k+r-1)! \zeta(d)}\, \tnu^j(f),
    \end{equation}
    where
    \begin{align}
        \label{eq: def c k r 1}
        c_{k+r-1}(n_1, \ldots, n_r)= \sum_{(x_1, \ldots, x_r) \in \{0,1\}^r} (-1)^{r - (x_1 + \cdots+ x_r)} \left(n_1x_1 + \cdots + n_rx_r\right)^{k+r-1}.
    \end{align}
    
    The measure $\tnu^j$ satisfies the following properties:
    
    \begin{enumerate}
        \item The measure $\tnu^j$ is absolutely continuous with respect to $\tmu^j$, and the Radon--Nikodym derivative of $\tnu^j$ with respect to $\tmu^j$ is a characteristic function (explicitly described in the proof).
        
        \item The measure $\tnu^j$ decomposes as
        $$
        \tnu^j = \tnu^{j*} \otimes m_{\hZp^d},
        $$
        for some finite measure $\tnu^{j*}$ on 
        $\E_d^j \times \Sphere^{m_1} \times \cdots \times \Sphere^{m_k} \times \Sphere^{n_1} \times \cdots \times \Sphere^{n_r} \times \R$.
        
        \item Let $\tnu^{**}$ denote the pushforward of $\tnu^j$ under the natural projection to 
        $$
        \Sphere^{m_1} \times \cdots \times \Sphere^{m_k} \times \Sphere^{n_1} \times \cdots \times \Sphere^{n_r} \times \R.
        $$
        Then $\tnu^{**}$ is invariant under the action of the group
        $$
        \Ortho(m_1) \times \cdots \times \Ortho(m_k) \times \Ortho(n_1) \times \cdots \times \Ortho(n_r),
        $$
        where $\Ortho(l) \subset \GL_{l}(\R)$ denotes the group of linear transformations preserving the chosen norm on $\R^l$.
    \end{enumerate}
\end{thm}

\begin{rem}
    Note that for a given $\theta$, the quantities $\Theta_j(\theta,p,q)$ might not be defined for certain values of best approximations $(p,q)$. The equation~\eqref{eq: main thm general 2} should be interpreted as asserting that the number of such exceptional approximates is of order $o(T^{k+r-1})$ as $T \to \infty$, and hence their contribution becomes negligible in the limit.

    Furthermore, the action of 
    $$
    \Ortho(m_1) \times \cdots \times \Ortho(m_k) \times \Ortho(n_1) \times \cdots \times \Ortho(n_r)
    $$
    on
    $$
    \Sphere^{m_1} \times \cdots \times \Sphere^{m_k} \times \Sphere^{n_1} \times \cdots \times \Sphere^{n_r} \times \R
    $$
    is given by the natural diagonal action: each $\Ortho(m_i)$ acts on the corresponding factor $\Sphere^{m_i}$ by rotation, and each $\Ortho(n_j)$ acts on the corresponding $\Sphere^{n_j}$ similarly, while the $\R$ coordinate is left invariant.
\end{rem} \vspace{0.15in}

\section{Corollaries: General Case}
\label{sec: Corollaries: General Case}
In this section, we briefly discuss some immediate consequences of Theorem~\ref{main thm best}.

\subsection{Counting Best Approximates with Constraints}

The problem of counting approximates to a matrix—especially under additional geometric or algebraic constraints—has received significant attention historically. We refer the reader to the classical work of Wolfgang Schmidt~\cite{S60} (see also \cite{WYYK}) and more recent developments such as \cites{Gallagher, AlamGhoshYu, NRS}. 

In light of this, it is natural to ask whether one can count the number of \emph{best approximates} that satisfy such constraints. Theorem~\ref{main thm best} provides an affirmative answer to this question. Specifically, by applying Theorem~\ref{muJM} to Theorem~\ref{main thm best}, and choosing the test function $f = \ind_A$ for a suitable $\tmu^j$-JM subset \( A \subset \ZZ_j \) (as defined in Definition~\ref{def JM}), we obtain the following result.

\begin{cor}
\label{cor: Counting Best Approximates with Constraints}
    Fix \( 1 \leq j \leq d \), and let \( A \subset \ZZ_j \) be a $\tmu^j$-JM subset. Then for Lebesgue-almost every \( \theta \in \Mat \), the following holds. Let \( \cN(\theta, A, T) \) denote the number of best approximates \( (p, q) \) to \( \theta \) satisfying
    \[
    \Theta_j(\theta, p, q) \in A, \quad \|q\| \leq e^T.
    \]
    Then
    \begin{equation}
    \label{eq: cor: Counting Best Approximates with Constraints}
        \lim_{T \to \infty} \frac{\cN(\theta, A, T)}{T^{k+r-1}} = \frac{ c_{k+r-1}(n_1, \ldots, n_r)}{(k+r-1)! \, \zeta(d)} \cdot \tnu^j(A).
    \end{equation}
\end{cor} 

\begin{rem}
\label{rem: importance  of A}
    We emphasize that by choosing the set \( A \) appropriately in Corollary~\ref{cor: Counting Best Approximates with Constraints}, one can impose various types of constraints on the best approximates. These include:
    \begin{itemize}
        \item \emph{Congruence conditions} via the \( \hZ^d \) component,
        \item \emph{Directional constraints} via the \( \Sphere^{m_1} \times \cdots \times \Sphere^{n_r} \) component,
        \item \emph{Quality of approximation} via the \( \R \) component, and
        \item \emph{Relative size constraints} via the \( \E_d^j \) component.
    \end{itemize}
\end{rem}
\vspace{0.2in}

\subsection{Generalized L\'{e}vy-Khintchine Theorem}
\label{subsec: Generalized Levy-Khintchine Theorem}

We begin with an observation that allows us to deduce L\'{e}vy-Khintchine-type theorems from Theorem~\ref{main thm best}. Fix \( \theta \in \Mat \), and let \( (p_l, q_l) \in \Z^m \times \Z^n \) denote the sequence of best approximations to \( \theta \) (as defined in~\ref{def: best general case}), ordered by increasing \( \|q_l\| \). Then we have:
\[
\cN(\theta, \ZZ_j, \log \|q_l\|) = l,
\]
where \( \cN(\cdot) \) is as defined in Corollary~\ref{cor: Counting Best Approximates with Constraints}. It follows that
\[
\frac{(\log \|q_l\|)^{k+r-1}}{l} = \frac{(\log \|q_l\|)^{k+r-1}}{\cN(\theta, \ZZ_j, \log \|q_l\|)}.
\]
This identity shows that the L\'{e}vy-Khintchine theorem can be understood as an asymptotic for the growth rate of the number of best approximates as the denominator \( q \) grows. Thus, applying Corollary~\ref{cor: Counting Best Approximates with Constraints} immediately yields the following result:

\begin{cor}
    \label{thm: cor 3 to main thm}
    For Lebesgue almost every \( \theta \in \Mat \), the sequence \( (p_l, q_l) \in \Z^m \times \Z^n \) of best approximations to \( \theta \), ordered by increasing \( \|q_l\| \), satisfies
    \begin{equation}
        \label{eq: cor 3 to main thm}
        \lim_{l \to \infty} \frac{(\log \|q_l\|)^{k+r-1}}{l}
        = \left( \frac{c_{k+r-1}(n_1, \ldots, n_r)}{(k+r-1)! \, \zeta(d)} \cdot \tnu^j(\ZZ_j) \right)^{-1}.
    \end{equation}
\end{cor}

\begin{rem}
\label{rem: cheung conjecture}
    Cheung's conjecture follows from Corollary~\ref{thm: cor 3 to main thm} in the special case \( n = r = 1 \), \( k = m \). Moreover, Equation~(1) in the Cheung–Chevallier theorem \cite{CC19}, which treats general norms, follows from Corollary~\ref{thm: cor 3 to main thm} by taking \( k = r = 1 \).
\end{rem}

\begin{rem}
    As noted above, the classical L\'{e}vy-Khintchine theorem concerns the asymptotic count of best approximates as the denominator \( q \) tends to infinity. Corollary~\ref{cor: Counting Best Approximates with Constraints} therefore yields a \emph{L\'{e}vy-Khintchine-type theorem with constraints}, where one can impose congruence, directional, or size restrictions on the approximates. 
\end{rem} \vspace{0.2in}

\subsection{Joint Equidistribution of Best Approximates}
\label{subsec: Joint Equidistribution of Best Approximates}
The third corollary of Theorem~\ref{main thm best} generalises the joint equidistribution of best approximations studied by Shapira and Weiss in~\cite{SW22}. In fact, in the special case $n = r = k = 1$ and $j = d$, this corollary recovers the main result of~\cite{SW22}. Moreover, our result yields several further Diophantine corollaries. These can be derived in a manner analogous to that in~\cite{aggarwalghosh2024joint}, where similar corollaries were obtained from an equidistribution result concerning $\varepsilon$-approximates.

\begin{cor}
\label{cor2 to main thm best}
    Fix an integer $1 \leq j \leq d$. Suppose $\nu^j$ denotes the probability measure on $\ZZ_j$ obtained by normalising $\tnu^j$. Then for Lebesgue almost every $\theta \in \Mat$, the following holds. Let $(p_l, q_l) \in \Z^m \times \Z^n$ be the sequence of best approximations of $\theta$, ordered according to increasing $\|q_l\|$. Then for every bounded continuous function $f$ on $\ZZ_j$, we have 
    \begin{equation*}
       \frac{1}{l} \sum_{i=1}^l f\left( \Theta_j(\theta, p, q) \right) = \nu^j(f). 
    \end{equation*}
\end{cor}
\vspace{0.2in}

\subsection{Doeblin--Lenstra Law}
\label{subsec: Doeblin--Lenstra Law}
An especially intriguing application of Corollary~\ref{cor2 to main thm best} is to the classical Doeblin--Lenstra conjecture, which was resolved by Bosma, Jager, and Wiedijk~\cite{BJW}. The theorem states that for Lebesgue-almost every $\theta \in (0,1)$,
\[
\lim_{l \to \infty} \frac{1}{l} \sum_{j=1}^l f\bigl( q_j |\theta q_j - p_j| \bigr) = \int_0^1 f(z)\, d\nu(z),
\]
for every bounded continuous function $f : [0,1] \to \mathbb{R}$, where $\nu$ is the probability measure on $[0,1]$ with density
\begin{align} \label{eq: def measure Doeblin Lenstra}
    \frac{d\nu}{dz} = 
\begin{cases}
\displaystyle \frac{1}{\ln 2}, & 0 \le z \le 1/2, \\[4pt]
\displaystyle \frac{1 - 1/z}{\ln 2}, & 1/2 < z \le 1.
\end{cases}
\end{align}

Geometrically, the quantity $q_k |\theta q_k - p_k|$ represents the area of a rectangle, as depicted in Figure~\ref{fig:rectangle}. This naturally motivates higher-dimensional analogues, where the role of $q_k |\theta q_k - p_k|$ is played by the volume of suitable geometric objects:

\begin{itemize}
    \item In the setting of Definition~\ref{def:best approx k=m, n=1}, the analogue is the volume of the hypercuboid in Figure~\ref{fig:hypercuboid}, given by
    \[
    q \cdot |p_1 + \theta_1 q| \cdots |p_m + \theta_m q|.
    \]
    
    \item For Definition~\ref{def: best k=r=1}, it corresponds to the volume of the cylinder in Figure~\ref{fig:cylinder}, namely
    \[
    \|q\|^n \cdot \|p + \theta q\|^m.
    \]
\end{itemize}

More generally, for best approximations as defined in Definition~\ref{def: best general case}, the natural generalization is the volume of a product of balls:
\[
B_{\|\rho_1(p+\theta q)\|}^{m_1} \times \cdots \times B_{\|\rho_k(p+\theta q)\|}^{m_k} \times B_{\|\rho_1'(q)\|}^{n_1} \times \cdots \times B_{\|\rho_r'(q)\|}^{n_r}.
\]
Up to a multiplicative constant, this volume coincides with the quantity $\disp(\theta, p,q)$ defined by
\begin{align*}
\disp(\theta, p,q) = \left( \prod_{i=1}^k \|\rho_i(p+\theta q)\|^{m_i} \right) \cdot \left( \prod_{j=1}^r \|\rho_j'(q)\|^{n_j} \right),
\end{align*}
for every $\theta \in \Mat$ and $(p,q) \in \mathbb{Z}^m \times \mathbb{Z}^n$.

\medskip

We therefore obtain a higher-dimensional analogue of the Doeblin--Lenstra law using Corollary~\ref{cor2 to main thm best}:

\begin{cor}
    \label{cor 3 Doeblin Lenstra}
    There exists a probability measure $\nu_{\mathbb{R}}$ on $\mathbb{R}$ such that the following holds for Lebesgue-almost every $\theta \in \Mat$. Let $(p_l, q_l) \in \mathbb{Z}^m \times \mathbb{Z}^n$ denote the sequence of best approximations to $\theta$, ordered by increasing $\|q_l\|$. Then for every bounded continuous function $f : \mathbb{R} \to \mathbb{R}$,
    \[
    \lim_{l \to \infty} \frac{1}{l} \sum_{i=1}^l f\left( \disp(\theta, p_i, q_i) \right) = \nu_{\mathbb{R}}(f).
    \]
\end{cor}

\begin{rem}
For historical background and additional references on the Doeblin-Lenstra law, see Remark~\ref{rem: Doeblin Lenstra}. See also Corollaries~\ref{thm:intro 1} and~\ref{cor: cantor} for a generalisation of the above result in the special case $k = r = 1$ to measures singular with respect to Lebesgue measure, including those supported on curves and fractals.
\end{rem}

\section{Further Refinement}
\label{sec: Further refinements}

In special cases $k=r=1$, the statement of Theorem \ref{main thm best} can be further strengthened as follows.

    \begin{thm}
    \label{Main thm time visits}
    Assume that $k=r=1$. Fix $\theta \in \Mat$ of generic type (defined as in Definition \ref{def: generic type}). Fix an integer $1 \leq j \leq d$. Fix a $\tnu^j$-JM subset $A \subset \ZZ_j$ (defined as in Definition~\ref{def JM}) of positive $\tnu^j$-measure, such that
    \begin{align}
    \label{eq: condition Main thm time visits}
        A \subset \{(\Lambda, v, \gamma, w) \in \E_d^j \times (\Sphere^m \times \Sphere^n) \times [0,\e] \times \hZ^d: v_j \geq 0\}.
    \end{align}
    Then there exists a measure $\mu^{A}$ (independent of $\theta$) on $(\ZZ_j \times \R)^\N$ such that the following holds. 
    
    Let $(p_l, q_l) \in \Z^m \times \Z^n$ be the sequence of best approximates of $\theta$ satisfying 
    $$
     \Theta_j(\theta,p,q) \in A,
     $$
     and ordered ordered according to increasing $\|q_l\|$. 
     
     Then for any $s \geq 1$, and any continuous bounded function $f$ on $(\ZZ_j \times \R)^\N$ which depends only first $s$-coordinates of input argument, we have
    \begin{equation}
    \label{eq: main thm time visits}
        \lim_{T \rightarrow \infty} \frac{1}{T} \sum_{\|q_l\| \leq e^T} f \left( \left(\Theta_j(\theta, p_{l+i}, q_{l+i}), \ \log \frac{\| q_{l+i}\|}{\| q_{l+i-1}\|}   \right)_{i \in \Z_{\geq 0}} \right) = \mu^{A}(f).
    \end{equation}

    Moreover, if $\mu^{A*}$ denote the measure on $\ZZ_j \times \R$ obtained by pushforward of  $\mu^A$ under natural projection map onto the first coordinate, then $\mu^{A*}$ satisfies the following properties
    \begin{enumerate}
    \item The pushforward of $\mu^{A*}$ under the natural projection map $\ZZ_j \times \R \rightarrow \ZZ_j$ equals the restriction of $\tnu^j$ to the set $A$.
    \item  The pushforward of $\mu^{A*}$ under the natural projection map $\ZZ_j \times \R \rightarrow \R$ has expectation equals to $1$.
    \end{enumerate}
\end{thm}
\begin{rem}
\label{rem: Generic has full measure}
   A matrix \(\theta \in \Mat\) is said to be of \emph{generic type} if the trajectory
\[
\left( 
\begin{pmatrix}
    e^{\frac{t}{m}} I_m & \\
    & e^{-\frac{t}{n}} I_n
\end{pmatrix}
\begin{pmatrix}
    I_m & \theta \\
    & I_n
\end{pmatrix}
\SL_d(\Z) 
\right)_{t \geq 0}
\]
equidistributes with respect to the unique \(\SL_d(\R)\)-invariant probability measure on \(\SL_d(\R)/\SL_d(\Z)\). The condition that \(\theta \in \Mat\) is of generic type is well studied, and it is well known that Lebesgue-almost every \(\theta\) satisfies this condition.
 In fact, there are a wide class of examples of measures other than the Lebesgue measure which gives full measure to above set. The example includes the natural measures defined on limit sets of IFS defined as in \cite{SimmonsWeiss}, e.g., Cantor set or Koch snowflake, or the natural measure on the non-degenerate curves in $\Mat$ defined as in \cite{solanwieser}. Thus, we find that the above theorem holds for almost every point in $\Mat$ with respect to a wide class of measures, other than the Lebesgue measure.
\end{rem}
\begin{rem}
\label{rem: condition explain}
    Note that the condition \eqref{eq: condition Main thm time visits} is needed to select only once between the best approximations $(p,q)$ and their negatives $(-p,-q)$. If both approximations $(p,q)$ and $(-p,-q)$ are considered, ordering according to increasing $\| q\|$ would not be unique.
\end{rem}

\section{Corollaries: Special Case}
\label{sec: Corollaries: Special Case $k=r=1$}
 
In this section, we state further corollaries for the special case $k=r=1$, including a resolution of Question~\ref{Que: Cantor} and other questions posed after Theorem~\ref{thm: CC19}.

\subsection{Generalisation of Cheung-Chevallier Theorem}
Using Theorem~\ref{Main thm time visits}, we can generalise the main theorem of Cheung and Chevallier~\cite{CC19} as follows. 
\begin{cor}
\label{thm:intro 1}

Fix a probability measure $\mu$ on $\Mat$ chosen among the following:
\begin{itemize}
    \item $\mu$ is absolutely continuous with respect to Lebesgue measure;
    \item $\mu$ is the Bernoulli measure supported on the limit sets of IFSs defined as in~\cite{SimmonsWeiss}, e.g., the middle third Cantor set or Koch snowflake;
    \item $\mu$ is the natural measure supported on non-degenerate curves in $\Mat$ as defined in~\cite{solanwieser}, e.g., the Veronese curve.
\end{itemize}

Fix an integer $1 \leq j \leq d$. Let $A \subset \ZZ_j$ be a $\tnu^j$-JM subset (as in Definition~\ref{def JM}) of positive $\tnu^j$-measure, such that
\begin{align}
\label{eq: condition Main thm time visits 1}
    A \subset \left\{(\Lambda, v, \gamma, w) \in \E_d^j \times (\Sphere^m \times \Sphere^n) \times [0,\e] \times \hZ^d : v_j \geq 0 \right\}.
\end{align}

Then, for $\mu$-almost every $\theta$, the following holds.Let \((p_l, q_l) \in \mathbb{Z}^m \times \mathbb{Z}^n\) denote the sequence of best approximates to \(\theta\), defined as in Definition~\ref{def: best k=r=1}, satisfying
\[
\Theta_j(\theta, p_l, q_l) \in A,
\]
and ordered by increasing \(\|q_l\|\). Then, for any $s \geq 0$, the following limits exist and are independent of the choice of $\theta$ and $\mu$:
\begin{align}
    \lim_{l \rightarrow \infty} \frac{1}{l} \log \|q_{l}\|, \label{eq: thm:intro 1 1} \\
    \lim_{l \rightarrow \infty} \frac{1}{l} \log \|p_{l} + \theta q_{l}\|, \label{eq: thm:intro 1 2} \\
    \lim_{l \rightarrow \infty} \frac{1}{l} \sum_{i=1}^l \delta_{\|q_{i+s}\|^n \|p_i + \theta q_i\|^m}. \label{eq: thm:intro 1 3}
\end{align}
Moreover, we still have:
\begin{align}
\label{eq: thm:intro 1 4}
    \lim_{l \rightarrow \infty} \frac{1}{l} \log \|q_{l}\|_{\R^n} 
    = -\frac{m}{n} \lim_{l \rightarrow \infty} \frac{1}{l} \log \|p_{l} + \theta q_{l}\|_{\R^m}.
\end{align}
\end{cor} \vspace{0.2in}
\begin{rem}
As explained in Remark~\ref{rem: condition explain}, the condition~\eqref{eq: condition Main thm time visits 1} is required to select only one representative between each pair of best approximations \((p, q)\) and \(({-p}, {-q})\). If both are considered, then ordering by increasing \(\|q\|\) would not yield a unique sequence.
\end{rem}

\begin{rem}
As discussed in Remark~\ref{rem: importance of A}, by choosing the set \(A\) appropriately in Corollary~\ref{thm:intro 1}, one can impose additional constraints on the best approximates:
\begin{itemize}
    \item \emph{Congruence conditions} via the \(\hZ^d\) component,
    \item \emph{Directional constraints} via the \(\Sphere^{m_1} \times \cdots \times \Sphere^{n_r}\) component,
    \item \emph{Quality of approximation} via the \(\R\) component, and
    \item \emph{Relative size constraints} via the \(\E_d^j\) component.
\end{itemize}
Thus, Corollary \ref{thm:intro 1} answers the questions posed after Theorem~\ref{thm: CC19}.
\end{rem}

\begin{rem}
\label{rem: Doeblin Lenstra}
    The authors wish to emphasize that the case $m = n = 1$ of equation~\eqref{eq: thm:intro 1 3} is commonly known as the Jager version of the Doeblin--Lenstra law (see Section~\ref{subsec: Doeblin--Lenstra Law}). For further background, we refer the reader to the excellent surveys~\cites{Iosifescu, IosifescuKraaikamp}.

    We also note that Cheung and Chevallier~\cite{CC19} established equation~\eqref{eq: thm:intro 1 3} only in the case $s = 1$, and only with respect to Lebesgue measure. For $n = 1$ and $s = 0$, the result was previously obtained through the work of Shapira and Weiss~\cite{SW22}. The remaining values of $s$, including the case $s = 0$ for $n \neq 1$ (corresponding to the classical Doeblin--Lenstra setting), are new to this paper. 
\end{rem}
\vspace{0.2in}

\subsection{Application to Fractals}

Corollary~\ref{thm:intro 1} immediately implies the following result, which answers Question~\ref{Que: Cantor} and extends the Doeblin-Lenstra law to the middle-third Cantor set.

\begin{cor}
\label{cor: cantor}
Let $\mu$ denote the restriction of the $\log 2 / \log 3$-dimensional Hausdorff measure to the middle-third Cantor set. Then, for $\mu$-almost every $\theta$, the following holds. Let $(p_l, q_l)_{l \in \mathbb{N}}$ denote the numerators and denominators of the convergents in the continued fraction expansion of $\theta$. Then
\[
    \lim_{l \rightarrow \infty} \frac{1}{l} \log q_l = \frac{\pi^2}{12 \log 2}.
\]
Moreover, for every bounded continuous function $f : [0,1] \to \mathbb{R}$, we have
\[
\lim_{l \to \infty} \frac{1}{l} \sum_{j=1}^l f\bigl( q_j |\theta q_j - p_j| \bigr) = \int_0^1 f(z)\, d\nu(z),
\]
where $\nu$ is the probability measure on $[0,1]$ as defined in~\eqref{eq: def measure Doeblin Lenstra}.
\end{cor}

\vspace{0.2in}

\subsection{Application to Polynomial Approximation}

Another application of Corollary~\ref{thm:intro 1} arises by choosing $m=1$ and taking the measure $\mu$ to be the pushforward of Lebesgue measure under the map $x \mapsto (x, \ldots, x^n) \in \M_{1 \times n}(\R)$. This leads to a L\'evy–Khintchine-type theorem for polynomial approximation, a connection suggested to us by Yann Bugeaud, whom we gratefully acknowledge.

We briefly explain the setup. For an integer polynomial \( P(X) \), the height \( H(P) \) is defined as the maximum of the moduli of the coefficients of the non-zero powers of \( X \), i.e.,
\[
H(c_0 + c_1X + \cdots + c_sX^s) = \max\{ |c_1|, \ldots, |c_s| \}.
\]
Let \( \theta \in \R \) and \( s \in \N \). A polynomial \( P(X) \) of degree at most \( s \) is said to be a \emph{best approximating integer polynomial} for \( \theta \) if it satisfies:
\begin{align*}
    |P(\theta)| &< |Q(\theta)| \quad \text{for all integer polynomials \( Q \) of degree \( \leq s \) with \( H(Q) < H(P) \)}, \\
    |P(\theta)| &\leq |Q(\theta)| \quad \text{for all integer polynomials \( Q \) of degree \( \leq s \) with \( H(Q) = H(P) \)}.
\end{align*}

In this case, Corollary~\ref{thm:intro 1} implies the following:

\begin{cor}
\label{cor:poly-approx}
For all \( s \in \N \) and for Lebesgue-almost every \( \theta \in \R \), the following holds. Let \( \{P_l(X)\}_{l \geq 1} \) denote the sequence of best approximating integer polynomials of degree at most \( s \), ordered according to increasing height \( H(P_l) \). Then the following limits exist and are independent of \( \theta \):
\[
\lim_{l \rightarrow \infty} \frac{\log H(P_l)}{l}, \qquad
\lim_{l \rightarrow \infty} \frac{\log |P_l(\theta)|}{l}.
\]
\end{cor}\vspace{0.2in}

\subsection{Distribution of determinants}

The study of the determinants of matrices whose entries are best approximates is classical, see for instance the important work of Lagarias \cites{Lagarias, Lagarias2}, and also the more recent works of Moshchevitin \cite{Moshchevitin2007}.  Theorem \ref{Main thm time visits} gives the following general result about the limiting distribution of determinants. 

\begin{cor}
\label{cor: determinant}
    Assume that $k=r=1$. Fix $\theta \in \Mat$ of generic type (defined as in Definition \ref{def: generic type}). Fix an integer $1 \leq j \leq d$. Fix a $\tnu^j$-JM subset $A \subset \ZZ_j$ (defined as in Definition~\ref{def JM}) of positive $\tnu^j$-measure, such that
    \begin{align}
    \label{eq: condition Main thm time visits 2}
        A \subset \{(\Lambda, v, \gamma, w) \in \E_d^j \times (\Sphere^m \times \Sphere^n) \times [0,\e] \times \hZ^d: v_j \geq 0\}.
    \end{align}
    Then there exists a measure $\mu^{A, \Z}$ (independent of $\theta$) on $\Z$ such that the following holds for all bounded functions on $\Z$.
    \begin{align}
    \label{eq: cor: determinant}
      \lim_{l \rightarrow \infty} \frac{1}{l} \sum_{i=1}^l f\left( \det \begin{pmatrix}
            p_{i} & \cdots & p_{i+d-1}\\ q_i & \cdots & q_{i+d-1}
        \end{pmatrix} \right) = \mu^{A, \Z}(f).
    \end{align}
\end{cor} \vspace{0.2in}

\begin{rem}
    
During a recent talk of Uri Shapira, we learned that Shapira and Weiss have independently proved the result for $n = 1$ and $A = \ZZ_j$. This is to appear in a forthcoming paper of theirs.

\end{rem}

\subsection{Congruence condition}
Diophantine approximation with congruence conditions in one dimension was studied systematically in the work of Sz\"{u}sz and his co-authors (\cites{SzuszActa, HartmanSzusz, Szusz, Szusz58}. More recently, Borda \cite{borda2023} has established very general limiting theorems, including functional central limit theorems in one dimensional Diophantine approximation with congruence constraints. The higher dimensional problem has seen some recent activity using methods from dynamics and the geometry of numbers, cf. \cites{AlamGhoshYu, AlamGhoshHan, SW22}.

Theorem~\ref{Main thm time visits} contributes further to this line of inquiry by analyzing the distribution of
$$
\left( \begin{pmatrix}
    p_i & \cdots & p_{i+l} \\
    q_i & \cdots & q_{i+l}
\end{pmatrix}  \mod N \right)_{i \in \N},
$$
where \((p_i, q_i)\) denotes the sequence of best approximations to \(\theta\). It shows that for every \(\theta\) of generic type—in particular, for Lebesgue almost every \(\theta\)—the sequence equidistributes with respect to a common limiting measure. For further results in the special case \(m=n=1\) and \(l=d\), see \cite{borda2023}.

In a forthcoming work \cite{aggarwalghosh2025c}, we establish central limit theorems for multi-dimensional Diophantine approximation with congruence constraints.

\section{Notations}
\label{sec: Notation}

\subsection{Norms}
\label{sec: Norms}
By re-scaling the norms, without loss of generality, we may assume that the norms \( \|.\| \) on each of \( \R^{m_i} \) satisfy
\[
\{x \in \R^{m_i} : \|x\|_\el < 1\} \subset \{x \in \R^{m_i} : \|x\| < 1\}
\]
for all \( i \). Similarly, we may assume that the norms \( \|.\| \) on each of \( \R^{n_j} \) satisfy
\[
\min_{x \in \Z^{n_j} \setminus \{0\}} \|x\| > 1.
\]

Let us fix \( \e > 0 \) large enough so that
\[
m_{\R^d}(\BB_{\e}) > 2^d,
\]
where for any \( \delta > 0 \), we define
\begin{align}
    \label{eq:def C_r}
    \BB_\delta := B_{\delta^{1/m_1}}^{m_1} \times B_1^{m_2} \times \cdots \times B_1^{n_r},
\end{align}
where \( B_\delta^l \) denotes the ball of radius \( \delta \) centered at the origin in \( \R^l \), with respect to the chosen norm on \( \R^l \).

\subsection{Convergence of Measures on lcsc spaces}
\label{Tight Convergence}
This short section is taken from our paper \cite{aggarwalghosh2024joint}. We recall some definitions and facts from measure theory that will be used in this paper. Let $X$ be a locally compact, second countable, Hausdorff topological space. Let $\mathcal{B}_X$ denote the Borel $\sigma$-algebra for the underlying topology and let $\mathcal{M}(X)$ denote the collection of finite regular Borel measures on $X$. For $\nu \in \mathcal{M}(X)$ and $f \in C_b(X)$, $\nu(f)$ will denote the integral $\int_X f \, d\nu$. We will use the \textit{tight topology} on $\mathcal{M}(X)$ for which convergence $\nu_l \rightarrow \nu$ is defined by either of the following equivalent requirements (see \cite[$\S 5$, Prop.~$9$]{Bou04b} for the equivalence):
\begin{itemize}
    \item[(i)] for all $f \in C_b(X)$, $\nu_l(f) \rightarrow \nu(f)$,
    \item[(ii)] for any compactly supported continuous function $f: X\rightarrow \R$, $\nu_l(f) \rightarrow \nu(f)$ and $\nu_l(X) \rightarrow \nu(X)$.
\end{itemize}
While this convergence is not equivalent to weak$-^*$ convergence, the two notions coincide when all measures involved are probability measures.

\begin{defn}
\label{def JM}
    Let $X$ be a locally compact second countable space and $\nu \in \mathcal{M}(X)$. We say that $E \in \mathcal{B}_X$ is \textit{Jordan measurable with respect to $\nu$} (abbreviated $\nu-$JM) if $\nu(\partial_X(E))=0$.
\end{defn}

Note that the following lemma is immediate.

\begin{lem}
\label{JMintersect}
    If $E,F \subset X$ are $\nu$-JM, then $E \cap F$ is also a $\nu$-JM subset of $X$.
\end{lem}

The importance of $\nu$-JM sets is that it captures tight convergence to $\nu$.

\begin{thm}[See {\cite[Thms.~2.1 $\&$ 2.7]{Bil68}} or {\cite[Chap.~4]{Bou04a}}]
\label{muJM}
     If $\nu_l, \nu \in \mathcal{M}(X)$, then $\nu_l \rightarrow \nu$ tightly if and only if for any $\nu$-JM set $E$ one has $\nu_l(E) \rightarrow \nu(E)$.

    Moreover, if $Y$ is also a locally compact second countable space and $\psi: X \rightarrow Y$ is a measurable function, then the push-forward map $\psi_*: \mathcal{M}(X)\rightarrow \mathcal{M}(Y)$ is continuous at a measure $\mu \in \mathcal{M}(X)$ (with respect to the tight topology) provided that $\psi$ is continuous $\nu$-almost everywhere.
\end{thm}

\subsection{Real Setup}
Let $G= \SL_d(\R)$, $\Gamma= \SL_d(\Z)$ and $\X \backsimeq G/\Gamma$. We will identify $\X$ with the space of all unimodular lattices in $\R^d$, via the isomorphism $g\Gamma \mapsto g\Z^d$. We denote by $\mu_{\X}$ the unique $G$-invariant Haar-Siegel probability measure on $G$. 

For \( t = (t_1, \ldots, t_{k+r-1}) \in \R^{k+r-1} \), we define the diagonal matrix \( a_t \in G \) by
\begin{equation}
    \label{defatt}
    a_t = \begin{pmatrix}
        e^{t_1} I_{m_1} & & & & & & \\
        & \ddots & & & & & \\
        & & e^{t_k} I_{m_k} & & & & \\
        & & & e^{-t_{k+1}} I_{n_1} & & & \\
        & & & & \ddots & & \\
        & & & & &  e^{-t_{k+r}} I_{n_r}
    \end{pmatrix},
\end{equation}
where
\[
    t_{k+r} = \frac{m_1 t_1 + \cdots + m_k t_k - n_1 t_{k+1} - \cdots - n_{r-1} t_{k+r-1}}{n_r}.
\]\\

Now we borrow some notation from Section 8 of \cite{SW22}. Given a unimodular lattice $\Lambda$, we define the set of all primitive vectors in $\Lambda$ by $\Lambda_\prim$. Given a subset $W \subset \R^d$ and $l \geq 1$, we define 
\begin{equation}
    \label{defX(W,l)}
    \X(W,l) := \{\Lambda: \# (\Lambda_{\prim} \cap W) \geq l\}.
\end{equation}
For $l = 1$, we omit $l$ and denote 
\begin{equation}
    \label{defX(W)}
    \X(W):= \X(W,1).
\end{equation}
We also define 
\begin{equation}
    \label{defXsharp}
    \X^\sharp(W):= \X(W) \setminus \X(W,2).
\end{equation}
There is a natural map 
\begin{equation}
    \label{defv-unique}
    \vector: \X^\sharp(W) \rightarrow W \text{,  defined by  } \{\vector(\Lambda)\} =\Lambda_\prim \cap W.
\end{equation}
With this notation, we are now ready to state the following results from \cite{SW22}.
\begin{lem}[{\cite[Lem.~8.1]{SW22}}]
\label{SW22 Lem 8.1}
    Let $W \subset  \R^d$ be a compact set, $V \subset W$ a relatively open subset and $l \geq 1$ an integer. 
    \begin{itemize}
        \item The set $\X(W,l)$ is closed in $\X$.
        \item The set $\X^{\sharp}(W) \cap \X(V)$ is open in $\X(W)$.
        \item The map $\vector : \X^\sharp(W) \rightarrow W$ is continuous.
    \end{itemize}
\end{lem}
\begin{lem}[{\cite[Lem.~8.2]{SW22}}]
\label{Sw22 Lem 8.2}
    Let $W \subset \R^d$ be an open set. For any $l \geq 1$, $\X(W,l)$ is open in $\X$.
\end{lem}

\subsection{Adelic Setup}
In this brief section, we introduce an adelic machinery that lies at the heart of incorporating congruence conditions into Diophantine analysis. Our presentation here is heavily influenced by \cite{SW22}. Let $\Primes$ be the set of rational primes. Let $\A= \R \times \A_f= \R \times \prod_{p \in \Primes}'\Q_p$ be the ring of adeles. Here $\prod'$ stands for the restricted product, i.e, a sequence $\underbar{$\beta$}= (\beta_\infty, \beta_f)= (\beta_\infty,\beta_2, \beta_3 \ldots, \beta_p, \ldots )$ belongs to $\A$ if and only if $\beta_p \in \Z_p$ for all but finitely many $p$. As suggested by the notation, we denote the real coordinate of a sequence $\underbar{$\beta$} \in \A$ by $ \beta_\infty$ and the sequence of $p$-adic coordinates by $\beta_f= (\beta_p)_{p \in \Primes}.$ The rational numbers $\Q$ are embedded in $\A$ diagonally, that is, $q \in \Q$ is identified with the constant sequence $(q,q, \ldots)$. We let $\SL_d(\A)= \SL_d(\R) \times \SL_d(\A_f)= \SL_d(\R) \times \prod_{p \in \Primes}'\SL_d(\Q_p)$ and use similar notation $(g_\infty, g_f)= (g_\infty, (g_p)_{p \in \Primes })$ to denote elements of $\SL_d(\A).$ By a theorem of Borel, the diagonal embedding of $\SL_d(\Q)$ in $\SL_d(\A)$ is a lattice in $\SL_d(\A)$. Let $$K_f:= \prod_{p \in \Primes } \SL_d(\Z_p)$$ and $$\pi_f: \SL_d(\A) \rightarrow \SL_d(\A_f), \ \pi_f(g_\infty, g_f):=g_f.$$ Then $K_f$ is a compact open subgroup of $\SL_d(\A_f).$ Via the embedding $\SL_n(\A_f) \simeq \{e\} \times \SL_d(\A_f)$ we also think of $K_f$ as a subgroup of $\SL_d(\A)$. We shall use the following two basic facts (see \cite[Chap.~7]{PR94}):
\begin{itemize}
    \item[(i)] The intersection $K_f \cap \pi_f(\SL_d(\Q)) $ is equal to $\pi_f(\SL_d(\Z))$.
    \item[(ii)] The projection $\pi_f(\SL_d(\Q))$ is dense in $\SL_d(\A_f)$.
\end{itemize}
Let $$\XA= \SL_d(\A)/\SL_d(\Q), $$ and let $\mu_{\XA}$ denote the $\SL_d(\A)$-invariant probability measure on $\XA$. There is a natural projection $\pi: \XA \rightarrow  \X$, which can be described as follows: Given $\tilde x = (g_\infty, g_f)\SL_d(\Q) \in \XA$, using $(ii)$ and the fact that $K_f$ is open, we may replace the representative $(g_\infty, g_f)$ by another representative $(g_\infty \gamma, g_f\gamma)$, where $\gamma \in \SL_d(\Q)$ is such that $g_f\gamma \in K_f$. We then define $\pi(\tilde x)= g_\infty \gamma \SL_d(\Z)$. This is well defined since, if $g_f\gamma_1, g_f \gamma_2 \in K_f$, then by $(i)$, $\pi_f(\gamma_1^{-1}\gamma_2) \in K_f \cap \pi_f(\SL_d(\Q))= \pi_f(\SL_d(\Z))$, and so $g_\infty\gamma_1 \SL_d(\Z)= g_\infty\gamma_2 \SL_d(\Z)$. 

It is clear that $\pi$ intertwines the actions of $G_\infty:= \SL_d(\R)$ on $\XA$ and $\X.$ Since there is a unique $G_\infty$-invariant probability measure on $\X$, we have $(\pi)_*\mu_{\XA}= \mu_{\X}$. In particular, the  group $\{a_{t}\} \subset G_\infty$ defined as in \eqref{defatt} acts on both of these spaces and $\pi$ is a factor map for these actions.

\subsection{Birkhoff Genericity}
\label{sec: Birkhoff Genericity}
Given $\theta \in \Mat$, we define $\tilde \Lambda_\theta \in \XA$ as 
\begin{align}
\label{eq: def tilde Lambda}
    \tilde \Lambda_\theta = \left( \begin{pmatrix}
    I_m & \theta \\ & I_n
\end{pmatrix}, e_f \right) \SL_d(\Q) \in \XA,
\end{align}
and also define $\Lambda_\theta= \pi(\tilde \Lambda_\theta) \in \X$.

\begin{defn}
\label{def: generic flow}
     We will say that $x \in \X$ (resp. $\XA$) is $(a_t, \mu_{\X})$-generic  (resp. $(a_t, \mu_{\XA})$-generic) if
     \begin{align}
     \label{eq: def generic flow}
         \lim_{T \rightarrow \infty} \frac{1}{m_{\R^{k+r-1}}(J^T)} \int_{J^T} \delta_{a_{t_1, \ldots, t_r}x}\, dt_1 \cdots dt_r= \mu_{\X} \text{ (resp. } \mu_{\XA}),
     \end{align}
     with respect to the tight topology, where $J^\tau$ is defined as 
    \begin{align}
    \label{eq: def J^T}
    J^T = \left\{
        (t_1, \ldots, t_{k}, s_1, \ldots, s_{r-1}) \in (\R_{\geq 0})^{k + r - 1} \;\middle|\;
        \begin{aligned}
            &s_i \leq T \text{ for all } i = 1, \ldots, r-1, \\
            &0 \leq m_1 t_1 + \cdots + m_k t_k - n_1 s_1 - \cdots - n_{r-1} s_{r-1} \leq n_r T
        \end{aligned}
    \right\}.
\end{align}
    Note that
    \begin{align}
        \label{eq: measure of J T}
        m_{\R^{k+r-1}}(J^T)= \frac{T^{k+r-1}c_{k+r-1}(n_1, \ldots, n_r)}{m_1 \cdots n_{r-1}. (k+r-1)!},
    \end{align}
where $c_{k+r-1}(n_1, \ldots, n_r)$ is defined as in \eqref{eq: def c k r 1}.
\end{defn}
\begin{defn}
    \label{def: generic type}
    Given $\theta \in \Mat$ will be called of generic type with respect to decomposition $m= m_1 + \cdots +m_k$ , $n= n_1 + \cdots +n_r$ if $\Lambda_\theta$ is $(a_t, \mu_{\X})$-generic. If $k=r=1$, i.e., if the decomposition is trivial, then we simply say that $\theta$ is of generic type.
\end{defn}

With this notation, we now recall the following results from \cite{aggarwalghosh2024joint}.
\begin{lem}[{\cite[Cor.~10.6]{aggarwalghosh2024joint}}]
\label{lem: Birkhoff genericity}
    For Lebesgue almost every $\theta \in \Mat$, $\theta$ is of generic type with respect to decomposition $m= m_1 + \cdots +m_k$ , $n= n_1 + \cdots +n_r$.
\end{lem}

\begin{lem}[{\cite[Cor.~10.8]{aggarwalghosh2024joint}}]
\label{lem: Generic iff adele generic }
    An element $\theta \in \Mat$ is of generic type with respect to decomposition $m= m_1 + \cdots +m_k$ , $n= n_1 + \cdots +n_r$ if and only if $\Tilde{\Lambda}_\theta$ is $(a_t, \mu_{\XA})$-generic.
\end{lem}


\section{Construction of the Cross–Section}
\label{sec: Results borrowed}

The main strategy of this paper is to construct and exploit a suitable \emph{cross–section} in the space of lattices.  While the theory of cross–sections for one–parameter flows was pioneered by Kakutani and Ambrose and is succinctly surveyed in \cite{SW22}, its extension to multi–parameter actions was developed in our earlier work \cite{aggarwalghosh2024joint}.  In particular, although cross–sections have played an important role in homogeneous dynamics (see, e.g., \cites{AthreyaCheung, Marklof2010, MarStro, CC19}), those results pertain exclusively to one–parameter flows.

In \cite{aggarwalghosh2024joint} we introduced an analogous framework for multi–parameter flows and applied it to counting and joint equidistribution for $\e$–approximation.  The present paper relies heavily on those constructions; re-deriving every step here would unduly lengthen the exposition.  Accordingly, we state below only the key definitions and results, whose proofs follow exactly the same lines as in \cite{aggarwalghosh2024joint}.

\medskip
\noindent Before proceeding, for $\e>0$ chosen as in Section \ref{sec: Norms}, define the following subsets of $\R^{m_1}\times \left(\Sphere^{m_2}\times\cdots\times\Sphere^{m_k}\times\Sphere^{n_1}\times\cdots\times\Sphere^{n_r}\right)$:
\begin{align}
  \Le &:= \bigl\{(x,y)\colon \|x\|^{m_1}\le\e\bigr\},
  \label{eq:def-Le}\\
  \Le^+ &:= \{\,z\in\Le : z_1\ge0\}.
  \label{eq:def-Le-plus}
\end{align}
We then set
\[
  \sed \;:=\;\X(\Le) \;=\;\X(\Le^+),
  \qquad
 \ssed \;:=\;\X^\sharp(\Le^+),
\]
and
\[
   \seda \;:=\;\pi^{-1}\bigl(\sed\bigr),
  \qquad \qquad \qquad
  \sseda \;:=\;\pi^{-1}\bigl(\ssed\bigr).
\]

Next, define the maps \(\tilde \psi, \tilde{\psi}^-: \sseda \rightarrow \ZZ_1\) by
    \begin{align}
        \label{defpsia}
        \tilde \psi &= (\psi_{\E_d^1} \circ \pi, \psi' \circ \pi, \psi_f), \\
        \tilde \psi^- &= (\psi_{\E_d^1} \circ \pi, \vartheta \circ \psi' \circ \pi, -\psi_f),
    \end{align}
    where the component maps are defined as follows:
    \begin{itemize}
        \item \textbf{Finite part map \(\psi_f\):} 
        \(\psi_f: \sseda \to \hat{\Z}_\prim^d\). For \(\tilde \Lambda = (g_\infty, g_f)\SL_d(\Q) \in \sseda\) with \(g_f \in K_f\), let \(\Lambda = \pi(\tilde \Lambda) = g_\infty \Z^d\). Since \(\Lambda \in \ssed\), the vector \(\vector(\Lambda)\) defined in \eqref{defv-unique} is the unique primitive vector in \(\Le^+\). By choosing a representative in the coset \(g_\infty \SL_d(\Z)\), we may assume that the \(d\)-th column of \(g_\infty\) is \(\vector(\Lambda)\).  The uniqueness of $v(\Lambda)$ implies that if $\gamma \in \SL_d(\Z)$ satisfies that $g_\infty\gamma$ is another representative of $\Lambda$ having this property, then $$g_\infty \bfe_d = v(\Lambda)= g_\infty \gamma \bfe_d,$$ and hence $\gamma \bfe_d= \bfe_d $. Define $$\psi_f(\Lambda):= g_f \bfe_d;$$ namely, if we write $g_f=(g_p)_{p \in \mathbf{P}}$, then $\psi_f(\tilde \Lambda)$ is the element of $\hat{\Z}^d$ whose $p$ coordinate is the $d$-th column of $g_p$. The above discussion implies that $\psi_f$ is well defined.
        
        \item \textbf{Directional part \(\psi'\):}
        \(\psi': \ssed \to \Sphere^{m_1} \times \cdots \times \Sphere^{n_r} \times [0,\e]\). Let \(\vector(\Lambda) = (x,y)\) be the unique primitive vector in \(\Le^+\) contained in \(\Lambda\), where \(x \in \R^{m_1}\), \(y \in \R^{d - m_1}\). Then define
        \[
        \psi'(\Lambda) := \left(\frac{x}{\|x\|}, y, \|x\|^{m_1} \right).
        \]
        
        \item \textbf{Reflection map \(\vartheta\):} 
        \(\vartheta: \Sphere^{m_1} \times \cdots \times \Sphere^{n_r} \times [0,\e] \to \Sphere^{m_1} \times \cdots \times \Sphere^{n_r} \times [0,\e]\) is given by
        \[
        \vartheta(x,y) := (-x, y).
        \]
        
        \item \textbf{Shape map \(\psi_{\E_d^1}\):}
        \(\psi_{\E_d^1}: \ssed \to \E_d^1\) is defined via the renormalization
        \[
        \psi_{\E_d^1}(\Lambda) := w(\psi'(\Lambda))^{-1} \Lambda,
        \]
        where for \((x, y, \gamma) \in (0, \infty) \times \R^{d-1} \times (0, \infty)\), the matrix \(w(x,y,\gamma)\) is given by
        \[
        w(x,y,\gamma) = 
        \begin{pmatrix}
            \gamma I_{m_1} & \\
            & I_{d - m_1}
        \end{pmatrix}
        \begin{pmatrix}
            x \\ y & (\gamma^{m_1} \|x\|)^{-1/(d-1)} I_{d-1}
        \end{pmatrix}.
        \]
    \end{itemize}
    These maps will be used to give an explicit description of the cross-sectional measure (see Result~\ref{res 7}). 

    Further, define 
\[
  \radius \colon \sed \;\longrightarrow\;\R,
  \qquad
  \radius(\Lambda)\;=\;\min\bigl\{\|\rho_1(\varrho_1(x))\|^{m_1} : x\in\Lambda_{\mathrm{prim}}\cap\Le\bigr\}.
\]
By Lemma 8.1(iii) of \cite{SW22}, the restriction of $\radius$ to the subset $\ssed$ is continuous. 
    
    Finally define
\[
  \B \;:=\;\bigl\{\Lambda\in\seda : \#\bigl(\BB_{\radius(\Lambda)}\cap\Lambda_{\mathrm{prim}}\bigr)=2\bigr\},
  \qquad
  \BA \;:=\;\pi^{-1}(\B).
\]
The set $\BA\subset\seda$ is our desired cross–section for the multi–parameter flow $a_t$. \vspace{0.2in}

\noindent Since $\BA\subset\seda$, it is convenient to first study the properties of $\sed$ and $\seda$. In this section we state these properties, and defer the discussion of the corresponding properties of the cross–section $\BA$ to Section~\ref{sec: Properties of the Cross-section}, which follows.  Since all arguments mirror those in \cite{aggarwalghosh2024joint}, we omit the proofs.

\vspace{0.2in}
\begin{result}\label{res 1:discreteness}
The set \( \seda \) is a cross-section for the multi-parameter flow \( a_t \) on \( \XA \), in the sense that for \( \mu_{\XA} \)-almost every \( x \in \XA \) (resp. \( x \in \X \)), the set
\[
\{t \in \R^{k+r-1} : a_t x \in \seda\}
\]
is a discrete subset of \( \R^{k+r-1} \). The same holds for the set \( \sed \) in \( \X \).
\end{result}

\medskip
\begin{result}\label{res 2:measure}
The set \( \seda \) carries a natural \( K_f \)-invariant measure \( \mu_{\seda} \), defined by
\begin{equation}
    \label{defmu_Sa}
    \mu_{\seda}(E) := \lim_{\tau \to 0} \frac{1}{\tau^{k+r-1}} \mu_{\XA}(E^{I_\tau}) \quad \text{for every Borel measurable set } E \subset \seda,
\end{equation}
where
\begin{align}
    I_\tau &:= [0,\tau]^{k+r-1}, \label{eq: def: I tau} \\
    E^I &:= \{ a_{t} x : t \in I,\ x \in E \}. \label{def: E^I adeles}
\end{align}

Similarly, the set \( \sed \) is naturally equipped with the measure \( \mu_{\sed} := \pi_*(\mu_{\seda}) \), which satisfies an analogous formula to \eqref{defmu_Sa}.
\end{result}
\medskip

\begin{result}
\label{res: 3}
For any \( \delta > 0 \), consider the sets
\begin{align*}
    (\seda)_{\geq \delta} &:= \left\{ x \in \seda : \min\left\{ \|t\|_{\el} : t \in \tilde{T}_{\e}(x) \setminus \{(0,0)\} \right\} \geq 2\delta \right\} = \pi^{-1}\left( (\sed)_{\geq \delta} \right), \\
    (\seda)_{< \delta} &:= \seda \setminus (\seda)_{\geq \delta}.
\end{align*}
These sets satisfy the following important properties:
\begin{itemize}
    \item Both \( (\seda)_{\geq \delta} \) and \( (\seda)_{< \delta} \) are \( \mu_{\seda} \)-JM subsets of \( \seda \).
    \item For any Borel set \( E \subset (\seda)_{\geq \delta} \) and any Borel subset \( I \subset [0,\delta]^{k+r-1} \), we have
    \begin{equation}
        \label{eq: measure on subsets of sed geq delta}
        \mu_{\XA}(E^I) = \mu_{\seda}(E) \cdot m_{\R^{k+r-1}}(I).
    \end{equation}
\end{itemize}
\end{result}
\medskip

\begin{result}
    \label{res 4}
    The set \( \sseda := \pi^{-1}(\X^\sharp(\Le^+)) \) (resp. \( \sed := \X^\sharp(\Le^+) \)) is a co-null open subset of \( \seda \) (resp. \( \sed \)).
\end{result}
\medskip

\begin{result}
    \label{res 5}
    Define the set \( \ueda \) as
    \begin{align}
        \label{eq: def ued}
        \ueda := \pi^{-1}(\ssed \cap \X(L_\e^\circ)),
    \end{align}
    where \( L_\e^\circ \) is defined by
    \[
        L_\e^\circ := \left\{ (x,y) \in \R^{d-n_r} \times \R^{n_r} : (x,y) \in \Le,\ \|y\|^{n_r} < \e \right\}.
    \]
    Then \( \ueda \) is an open subset of \( \seda \) and satisfies the following properties:
    \begin{enumerate}
        \item \( \mu_{\XA}\left( \left( \operatorname{cl}_{\XA}(\seda) \setminus \ueda \right)^{(0,1)^{k+r-1}} \right) = 0 \).
        \item The map \( (0,1)^{k+r-1} \times \ueda \to \XA \), defined by \( (t, x) \mapsto a_t x \), is open.
    \end{enumerate}
    The same statements hold for \( \ued := \pi(\ueda) \), with \( \X \) in place of \( \XA \).
\end{result}
\medskip

\begin{result}
    \label{res 6}
    For any \(\mu_{\seda}\)-JM subset \(E \subset \seda\) and any \(\tau > 0\), the set
    \[
        E^{J^\tau} = \{ a_t x : t \in J^\tau, \, x \in E \}
    \]
    is \(\mu_{\XA}\)-JM, where \(J^\tau\) is defined in \eqref{eq: def J^T} and \(E^I\) as in \eqref{def: E^I adeles}.
\end{result}
\medskip

\begin{result}
    \label{res 7}
    The maps \(\tilde \psi\) and \(\tilde \psi^-\) are \(K_f\)-equivariant and continuous, where the action of \(K_f\) on $\ZZ_1$ is given by the trivial action on 
    \(\E_d^1 \times \Sphere^{m_1} \times \cdots \times \Sphere^{n_r} \times [0,\e]\) 
    and the natural action on \(\hat{\Z}_\prim^d\). Moreover, their pushforwards of $\mu_{\seda}$ satisfy:
    \begin{align}
        \label{eq: pushforward general}
        (\tilde \psi)_* \mu_{\seda} + (\tilde \psi^-)_* \mu_{\seda} = 
        \frac{m_1 \cdots n_{r-1}}{\zeta(d)} 
        \tmu_1.
    \end{align}
\end{result}

\medskip

\begin{result}
    \label{res 9}
    Assume that \(k = r = 1\), so that the flow \((a_t)\) becomes a one-parameter flow. In this case, given a subset \(A \subset \seda\), we define the following two functions on \(A\):
    \begin{itemize}
        \item The \emph{return time function} \(\tau_A: A \rightarrow \R_{>0}\), defined by
        \[
            \tau_A(x) = \min \{ t \in \R_{>0} : a_t x \in A \}.
        \]
        The minimum exists for \(\mu_{\seda}|_A\)-almost every \(x\), since the set \(\{ t \in \R : a_t x \in A \}\) is discrete for \(\mu_{\seda}\)-almost every \(x\). Hence, \(\tau_A\) is defined \(\mu_{\seda}|_A\)-almost everywhere.
        
        \item The \emph{first return map} \(T_A: A \rightarrow A\), defined by \(T_A(x) = a_{\tau_A(x)} x\). Clearly, this map is also defined \(\mu_{\seda}|_A\)-almost everywhere.
    \end{itemize}
    If \(A \subset \seda\) is a \(\mu_{\seda}\)-JM set with \(\mu_{\seda}(A) > 0\), then the maps \(\tau_A\) and \(T_A\) are continuous almost everywhere (with respect to \(\mu_{\seda}|_A\)). Moreover, they satisfy the following identity, known as the \emph{Kac formula}:
    \begin{align}
        \label{Kac formula}
        \int_A \tau_A \, d\mu_{\seda}|_A = 1.
    \end{align}
\end{result}

\section{Properties of the Cross-section}
\label{sec: Properties of the Cross-section}
 As mentioned in Section \ref{sec: Results borrowed}, this section aims to derive the properties of $\BA$ from the results of the previous section. To begin with, note that since $\BA \subset \seda$, it follows immediately from Result~\ref{res 1:discreteness} that for $\mu_{\XA}$-almost every \( x \in \XA \), the set  
\[
 \{t \in \R^{k+r-1} : a_t x \in \BA\}
\]
is discrete. This implies that $\BA$ is also a cross-section for the $(a_t)$-flow on $\XA$. 

Secondly, using Result~\ref{res 2:measure}, we obtain a natural $K_f$-invariant measure $\mu_{\BA}$ on $\BA$, given by the restriction of $\mu_{\seda}$ to $\BA$, and explicitly described by~\eqref{defmu_Sa}. The same holds for the set $\B \subset \X$, with the corresponding measure \(\mu_{\B} := \pi_* (\mu_{\BA})\). 

To proceed further, we define $\tvector: \B \rightarrow \Le^+$ as 
\begin{align}
    \label{eq: def tvector}
    \tvector(\Lambda)\;=\;\max_{\mathrm{lex}}\{\,w,\,-w\},
\end{align}
where $\{\pm w\}\subset \Lambda$ are the unique pair of nonzero vectors satisfying
\[
\radius(\Lambda)
=\bigl\|\varrho_1(\rho_1(w))\bigr\|^{m_1},
\]
and the \emph{lexicographic order} on $\R^d$ is given by
\[
(x_1,\dots,x_d)\;>\;(y_1,\dots,y_d)
\quad\Longleftrightarrow\quad
x_j>y_j\text{ for the first index }j\text{ with }x_j\neq y_j.
\]

 We now derive further properties of $\BA$ in the following lemmas. 
\medskip

\begin{lem}
\label{lem:Tempered}
    There exists a constant $M$ such that for any $x \in \BA$, we have $$\#\{t \in \Delta: a_tx \in \BA \} \leq M,$$ where $\Delta \subset \R^{k+r-1}$ is defined as
\begin{align}
\label{eq: def L}
\Delta = \left\{
    (t_1, \ldots, t_{k+r-1}) \in \R^{k+r-1} \;\middle|\;
    \begin{aligned}
        &t_1 > 0,\quad t_j \in [-1, 0] \text{ for } 2 \leq j \leq k, \\
        &t_j \in [0, 1] \text{ for } k+1 \leq j \leq k + r - 1, \\
        &0 \leq m_1 t_1 + \cdots + m_k t_k - n_1 t_{k+1} - \cdots - n_{r-1} t_{k + r - 1} \leq n_r
    \end{aligned}
\right\}.
\end{align}
\end{lem}

\begin{proof}
    The proof is motivated by \cite[Prop.~9.8]{SW22}. Fix \( M \) large enough so that any subset of \( \BB_1 \) of cardinality \( M + 1 \) contains distinct points \( x, y \) with \( \|x - y\| < e^{-1} \). Such an \( M \) exists by boundedness of \( \BB_1 \).

    Now apply a linear transformation of \( \R^d \) that dilates \( \R^{m_1} \) and \( \R^{d - m_1} \). Then for any \( l > 0 \), any subset of
    \[
        B^{m_1}_{l^{1/m_1}} \times B^{m_2}_e \times \cdots \times B^{n_r}_e
    \]
    of cardinality \( M + 1 \) contains distinct points \( x, y \) such that \( x - y \in \BB_l^\circ \).

    We claim that this constant \( M \) works. Note that since $\BA= \pi^{-1}(\B)$, it suffices to show that for any \( x \in \B \), we have
    \[
        \#\{t \in \Delta : a_t x \in \B \} \leq M.
    \]
    
    To prove this, assume by contradiction that there exists \( \Lambda \in \B \), and distinct \( t_1, \ldots, t_M \in \Delta \setminus \{(0, \ldots, 0)\} \) such that \( a_{t_i} \Lambda \in \B \) for all \( i \). Let \( v_0 = \tvector(\Lambda) \), and define \( v_i = \tvector(a_{t_i}\Lambda) \) for \( 1 \leq i \leq M \), where \( \tvector: \B \to \Le^+ \) is as in \eqref{eq: def tvector}.

    Since \( t_i \in \Delta \) and \( a_{t_i} v_i \in \Le \), we have
    \[
        1 \leq \|\rho_j(\varrho_1(v_i))\| \leq e \quad \text{for } 2 \leq j \leq k,
    \]
    and
    \[
        1 \leq \|\rho_l'(\varrho_2(v_i))\| \leq e \quad \text{for } 1 \leq l \leq r.
    \]

    Furthermore, for all \( 1 \leq i \leq M \), we have
    \begin{align}
        a_{t_i} v_0 \notin \BB_{\radius(a_{t_i} \Lambda)}. \label{eq:aa-bb-cc-3}
    \end{align}

    Since \( t_i \in \Delta \), it follows that
    \begin{align}
        \|\rho_j(\varrho_1(a_{t_i} v_0))\| \leq 1 &\quad \text{for } 2 \leq j \leq k, \label{eq:aa-bb-cc-1} \\
        \|\rho_l'(\varrho_2(a_{t_i} v_0))\| \leq 1 &\quad \text{for } 1 \leq l \leq r. \label{eq:aa-bb-cc-2}
    \end{align}

    Combining \eqref{eq:aa-bb-cc-3}, \eqref{eq:aa-bb-cc-1}, and \eqref{eq:aa-bb-cc-2}, we obtain
    \[
        \|\rho_1(\varrho_1(v_i))\|^{m_1} < \|\rho_1(\varrho_1(v_0))\|^{m_1} = \radius(\Lambda).
    \]

    Hence, the vectors \( v_0, v_1, \ldots, v_M \) are distinct and lie in
    \[
        B^{m_1}_{\radius(\Lambda)^{1/m_1}} \times B^{m_2}_e \times \cdots \times B^{n_r}_e.
    \]

    By the choice of \( M \), there exist distinct vectors \( w_1, w_2 \) among \( v_0, v_1, \ldots, v_M \) such that \( w_1 - w_2 \in \BB_{\radius(\Lambda)}^\circ \). Since \( w_1, w_2 \in \Lambda \), this contradicts the assumption that \( \#(\BB_{\radius(\Lambda)} \cap \Lambda_{\prim}) = 2 \). This contradiction completes the proof.
\end{proof}

\medskip
\begin{lem}
\label{lem: BA is JM}
    The set $\BA$ is $\mu_{\seda}$-JM subset of $\seda$ with positive $\mu_{\seda}$-measure.
\end{lem}
\begin{proof}
Since $\BA= \pi^{-1}(\B)$, it suffices to show that $\B$ is a $\mu_{\sed}$–JM subset of $\sed$ and has positive $\mu_{\sed}$–measure. We now prove the two claims in order:

\medskip

\noindent\textbf{(i) Positive measure.}  
We will show that $\B \cap \ssed$ is open. Since $\mu_{\sed}$ has full support on $\sed$ (Result~\ref{res 7}), this will imply that $\B$ has positive $\mu_{\sed}$–measure.

\medskip

To see this, note that
\begin{align}
     \sed\setminus (\B\cap \ssed)
  \;=\;
  (\sed\setminus\ssed)\;\cup\;K, \label{eq: aa bb cc dd 1}
\end{align}
where
\[
  K \;=\;\bigl\{\Lambda\in\ssed : \#\bigl(\BB_{\radius(\Lambda)}\cap\Lambda_{\mathrm{prim}}\bigr)\ge3\bigr\}.
\]

To prove that the right-hand side of \eqref{eq: aa bb cc dd 1} is closed, let \((\Lambda_i)_i\) be a sequence in the right-hand side of \eqref{eq: aa bb cc dd 1} converging to some \(\Lambda\). We aim to show that \(\Lambda\) also belongs to the right-hand side of \eqref{eq: aa bb cc dd 1}.

Since \(\sed \setminus \ssed\) is closed by Result~\ref{res 4}, we may assume without loss of generality that \(\Lambda_i \in K \cap \ssed\) for all \(i\). If the limit \(\Lambda \in \sed \setminus \ssed\), then we are done. Thus, we may further assume that \(\Lambda \notin \sed \setminus \ssed\).

Now, by Lemma~\ref{SW22 Lem 8.1}, the limit \(\Lambda\) lies in \(\sed\), and hence \(\Lambda \in \ssed\). Since each \(\Lambda_i \in K\), the set \(\BB_{\radius(\Lambda_i)}\) contains at least three primitive vectors. Because the radius function \(\radius\) is continuous on \(\ssed\), we have \(\radius(\Lambda_i) \to \radius(\Lambda)\).

Fix any \(r' > \radius(\Lambda)\). For all sufficiently large \(i\), we then have \(\radius(\Lambda_i) < r'\), and so \(\Lambda_i \in \X(\BB_{r'}, 3)\). By another application of Lemma~\ref{SW22 Lem 8.1}, it follows that \(\Lambda \in \X(\BB_{r'}, 3)\). Since this holds for all \(r' > \radius(\Lambda)\), we conclude that \(\Lambda \in K\).

Therefore, \(\Lambda \in (\sed \setminus \ssed) \cup K\), showing that the right-hand side of \eqref{eq: aa bb cc dd 1} is closed. It follows that \(\B\cap \ssed\) is open in \(\sed\), completing the proof.\\

\medskip

\noindent\textbf{(ii) The JM condition.}
Define
\[
  U_0
  \;=\;
  \bigl\{\Lambda \in \ssed : \Lambda \cap \BB_{\radius(\Lambda)}^\circ \neq \emptyset \bigr\}.
\]
We first show that \(U_0\) is open in \(\sed\). To that end, let \((\Lambda_i)_i\) be a sequence in \(\sed\) converging to some \(\Lambda \in U_0\). We aim to show that \(\Lambda_i \in U_0\) for all sufficiently large \(i\).

Since \(\Lambda \in \X(\BB_{\radius(\Lambda)}^\circ)\), we can write
\[
\Lambda \in \X(\BB_{\radius(\Lambda)}^\circ) = \bigcup_{r < \radius(\Lambda)} \X(\BB_r),
\]
so there exists \(s < \radius(\Lambda)\) such that \(\Lambda \in \X(\BB_s)\). By Lemma~\ref{Sw22 Lem 8.2}, the set \(\X(\BB_s)\) is open in \(\sed\), and therefore \(\Lambda_i \in \X(\BB_s)\) for all sufficiently large \(i\).

Moreover, since \(U_0 \subset \ssed\) and \(\ssed\) is open in \(\sed\) by Result~\ref{res 4}, we have \(\Lambda_i \in \ssed\) for all large enough \(i\). Hence by Lemma \ref{SW22 Lem 8.1}, we have \(\radius(\Lambda_i) \to \radius(\Lambda)\). In particular, \(s < \radius(\Lambda_i)\) for all large \(i\), and thus
\[
\Lambda_i \in \X(\BB_s) \subset \X(\BB_{\radius(\Lambda_i)}^\circ).
\]
This means that \((\Lambda_i)_\prim \cap \BB_{\radius(\Lambda_i)}^\circ \neq \emptyset\), i.e., \(\Lambda_i \in U_0\) for all sufficiently large \(i\). We conclude that \(U_0\) is open in \(\sed\), as claimed. \\

\medskip
To prove that \(\B\) is \(\mu_{\sed}\)-JM, observe that
\[
\partial_{\sed}(\B) \subset \left((\sed \setminus \ssed) \cup K\right) \setminus U_0.
\]
By Result~\ref{res 4}, we have \(\mu_{\sed}(\sed \setminus \ssed) = 0\). Therefore, it suffices to show that \(K \setminus U_0\) is \(\mu_{\sed}\)-null. \\

By Result~\ref{res 2:measure}, we have
\[
\mu_{\sed}(K \setminus U_0) = 0 \quad \Longleftrightarrow \quad \mu_{\X}((K \setminus U_0)^{\R^{k+r-1}}) = 0.
\]
Also we have
\[
(K \setminus U_0)^{\R^{k+r-1}} \subset \{g \cdot \Z^d : g \in X\},
\]
where
\begin{align*}
X 
=\, & \bigcup_{i=1}^k \left\{ M \in \SL_d(\R) : \|\rho_i(\varrho_1(M_{,1}))\| = \|\rho_i(\varrho_1(M_{,d}))\| \right\} \\
& \cup 
\bigcup_{j=1}^r \left\{ M \in \SL_d(\R) : \|\rho_j'(\varrho_2(M_{,1}))\| = \|\rho_j'(\varrho_2(M_{,d}))\| \right\},
\end{align*}
where for a matrix \(M \in \SL_d(\R)\), we denote by \(M_{,i}\) its \(i\)th column. \\

Since each condition above defines a proper submanifold of codimension at least one, the set \(X\) has Haar measure zero in \(\SL_d(\R)\). It follows that
\[
\mu_{\X}((K \setminus U_0)^{\R^{k+r-1}}) = 0,
\]
and hence \(\mu_{\sed}(K \setminus U_0) = 0\). This completes the proof that \(\B\) is \(\mu_{\sed}\)-JM.
\end{proof}

\medskip

\begin{lem}
    A Borel subset \( E \subset \BA \) is $\mu_{\BA}$-JM if and only if it is $\mu_{\seda}$-JM when considered as a subset of $\seda$.
\end{lem}

\begin{proof}
    Note that the boundary of $E$ relative to $\BA$ satisfies
    \[
    \partial_{\BA}(E) \subset \partial_{\seda}(E) \subset \partial_{\BA}(E) \cup \partial_{\seda}(\BA).
    \]
    Since $\BA$ is a $\mu_{\seda}$-JM subset of $\seda$ (by Lemma~\ref{lem: BA is JM}), the $\mu_{\seda}$-measure of $\partial_{\seda}(\BA)$ is zero. Therefore, $\mu_{\seda}(\partial_{\seda}(E)) = 0$ if and only if $\mu_{\seda}(\partial_{\BA}(E)) = 0$. The result follows.
\end{proof}

\section{Cross-section Correspondence}
\label{sec: Cross-section Correspondence}
This section is devoted to connecting dynamics on the cross-section we have constructed in the previous sections, with Diophantine approximation. More concretely, we relate successive time tuples $t_i$ for which $a_{t_i}\tilde \Lambda_\theta \in \seda,$ to the best approximates of $\theta$.

\begin{lem}
\label{lem: Cross-section correspondence}
Fix \( \theta \in \Mat \). Let \( (p_l, q_l) \) denote the best approximations to \( \theta \), ordered by increasing \( \|q_l\| \), and satisfying the following conditions:
\begin{align}
    \|\rho_i(p_l + \theta q_l)\| &\neq 0 && \text{for all } 2 \leq i \leq k, \label{eq: con 1} \\
    \|\rho_j'(q_l)\| &\neq 0 && \text{for all } 1 \leq j \leq r, \label{eq: con 2} \\
    \disp(\theta, p_l, q_l) &\geq \|\rho_1(p_l + \theta q_l)\|^{m_1}. \label{eq: con 3}
\end{align}

Define \( t_l = (t_{1,l}, \ldots, t_{k+r-1,l}) \in \R^{k+r-1} \) by
\begin{align}
\label{eq: def t l 1}
    t_{i,l} = 
    \begin{cases}
        -\log \|\rho_i(p_l + \theta q_l)\| & \text{for } 2 \leq i \leq k, \\
        \log \|\rho_{i-k}'(q_l)\| & \text{for } k+1 \leq i \leq k+r-1,
    \end{cases}
\end{align}
and require that
\begin{align}
\label{eq: def t l 2}
    m_1 t_{1,l} + \cdots + m_k t_{k,l} 
    - n_1 t_{k+1,l} - \cdots - n_r t_{k+r-1,l}
    = n_r \log \|\rho_r'(q_l)\|.
\end{align}

Then the map \( (p_l, q_l) \mapsto t_l \) is two-to-one. In fact, if \( (p_l, q_l) \mapsto t_l \), then so does \( (-p_l, -q_l) \). Moreover, for each \( l \), if \( a_{t_l} \tilde \Lambda_\theta \in \sseda \), and $(p_l+\theta q_l, q_l)= \tvector(\pi(a_{t_l} \tilde \Lambda_\theta))$ then
\begin{align}
    \tilde \psi(a_{t_l} \tilde \Lambda_\theta) &= \Theta_1(\theta, p_l, q_l), \label{eq: Image under psi} \\
    \tilde \psi^-(a_{t_l} \tilde \Lambda_\theta) &= \Theta_1(\theta, -p_l, -q_l). \label{eq: Image under psi -}
\end{align}

Define
\[
\Y_T(\theta) := \{ t_l : \|q_l\| \leq e^T \}.
\]
Then we have the equality
\[
N(\tilde \Lambda_\theta, T, \BA) = \Y_T(\theta),
\]
where for a Borel subset \( E \subset \seda \), we define
\begin{align}
\label{def:N(x,T,E)}
    N(x, T, E) := \{ t \in J^T : a_t x \in E \}.
\end{align}
Here, the set \( J^T \) is defined as in~\eqref{eq: def J^T}.

\end{lem}

\begin{proof}
Suppose $(p, q)$ and $(p', q')$ are best approximates to $\theta$ satisfying conditions~\eqref{eq: con 1}--\eqref{eq: con 3}, and both map to the same vector $t \in \R^{k+r-1}$. Without loss of generality, assume
\[
\|\rho_1(p' + \theta q')\| \leq \|\rho_1(p + \theta q)\|.
\]
Then, using~\eqref{eq: def t l 1} and~\eqref{eq: def t l 2}, we have
\begin{align*}
\|\rho_i(p + \theta q)\| &= \|\rho_i(p' + \theta q')\| = e^{-t_i} \quad \text{for } 2 \leq i \leq k, \\
\|\rho_j'(q)\| &= \|\rho_j'(q')\| = e^{t_{k+j}} \quad \text{for } 1 \leq j \leq r, \\
\|\rho_r'(q)\| &= \|\rho_r'(q')\| = e^{\frac{m_1 t_1 + \cdots + m_k t_k - n_1 t_{k+1} - \cdots - n_r t_{k+r-1}}{n_r}}.
\end{align*}
Therefore, the point $(p' + \theta q', q')$ lies in the box
\[
B_{\|\rho_1(p + \theta q)\|}^{m_1} \times \cdots \times B_{\|\rho_k(p + \theta q)\|}^{m_k} \times B_{\|\rho_1'(q)\|}^{n_1} \times \cdots \times B_{\|\rho_r'(q)\|}^{n_r}.
\]
Since $(p, q)$ is a best approximation and $(p',q') \neq (0,0)$, this implies $(p', q') \in \{(p, q), (-p, -q)\}$. This proves the first part of the lemma.

The identities in~\eqref{eq: Image under psi} follow directly from the definitions of $\tilde \psi$, $\tilde \psi^-$, and $\Theta_1$.

We now show that $N(\tilde \Lambda_\theta, T, \BA) \subset \Y_T(\theta)$. Let $t = (t_1, \ldots, t_{k+r-1}) \in N(\tilde \Lambda_\theta, T, \BA)$. By definition of $\BA$, there exists $(p, q) \in \Z^m \times \Z^n$ such that
\[
a_t(p + \theta q, q) = \tvector(\pi(a_t \tilde \Lambda_\theta)) \in \Le^+.
\]
Then, by the definition of $\Le^+$, the vector $t$ and the pair $(p, q)$ satisfy~\eqref{eq: def t l 1} and~\eqref{eq: def t l 2}, and conditions~\eqref{eq: con 1} and~\eqref{eq: con 2} are clearly satisfied.

To verify~\eqref{eq: con 3}, note that $t \in (\R_{\geq 0})^k \times [0, T]^{r-1}$ implies $t_1 \geq 0$, hence
\[
\disp(\theta, p, q) \geq \|\rho_1(p + \theta q)\|^{m_1},
\]
so condition~\eqref{eq: con 3} holds.

Now observe that
\[
\radius(\pi(\tilde \Lambda_\theta)) = \|\rho_1(\varrho_1(a_t(p + \theta q)))\|^{m_1},
\]
and since $a_t \tilde \Lambda_\theta \in \BA$, we have
\[
\BB_{\radius(\pi(\tilde \Lambda_\theta))} \cap \Lambda_\prim = \{\pm \tvector(\pi(\tilde \Lambda_\theta))\}.
\]
Applying $a_{-t}$ to both sides implies that the box
\[
B_{\|\rho_1(p + \theta q)\|}^{m_1} \times \cdots \times B_{\|\rho_k(p + \theta q)\|}^{m_k} 
\times B_{\|\rho_1'(q)\|}^{n_1} \times \cdots \times B_{\|\rho_r'(q)\|}^{n_r}
\]
contains no nonzero point of $\Lambda_\theta = \pi(\tilde \Lambda_\theta)$. Thus $(p, q)$ is a best approximation, and hence $t \in \Y_T(\theta)$. This proves $N(\tilde \Lambda_\theta, T, \BA) \subset \Y_T(\theta)$. The reverse inclusion $ \Y_T(\theta) \subset N(\tilde \Lambda_\theta, T, \BA) $ follows by retracing the above steps, along with the assumption on the choice of norms fixed in Section~\ref{sec: Norms}.

This completes the proof.
\end{proof}


\section{Cross-section Genericity}
\label{sec: Cross-section Genericity}

In the previous section, we explored the connection between the distribution of best approximates and the distribution of visits of the trajectory $\{a_t\tilde{\Lambda}_\theta : t \in \R^{k+r-1}\}$ to the set $\BA$. By Theorem~\ref{muJM}, the latter requires studying the quantity
\begin{align}
    \label{eq: abcdef 1}
    \frac{1}{m_{\R^{k+r-1}}(J^T)} \# N(x, T, E),
\end{align}
as $T \to \infty$, where $E$ ranges over $\mu_{\BA}$-JM subsets of $\BA$.  
This section aims to show that a point $x \in \XA$ is $(a_t, \mu_{\XA})$-generic if and only if the limit in \eqref{eq: abcdef 1} exists and equals $\mu_{\seda}(E)$—a property we refer to as \emph{cross-section genericity}. \\
\medskip

We start with \textbf{only if} part.

\begin{lem}[Birkhoff Genericity implies Cross-section Genericity]
\label{lem:BG implies CG}
Suppose $x \in \BA$ is $(a_t, \mu_{\XA})$-generic. Then for every $\mu_{\BA}$-JM subset $E \subset \BA$, we have
\begin{align}
    \lim_{T \to \infty} \frac{1}{m_{\R^{k+r-1}}(J^T)} \# N(x, T, E) = \mu_{\BA}(E).
\end{align}
\end{lem}

We first prove an auxiliary estimate:

\begin{lem}
\label{lem: tempered implies bound}
Suppose $x \in \BA$ is $(a_t, \mu_{\XA})$-generic. Then for every $\mu_{\BA}$-JM subset $E \subset \BA$, we have
\begin{align}
    \limsup_{T \to \infty} \frac{1}{m_{\R^{k+r-1}}(J^T)} \# N(x, T, E) \leq \frac{M}{m_{\R^{k+r-1}}(\Delta)} \mu_{\XA}(E^\Delta),
\end{align}
where $\Delta$ is as defined in~\eqref{eq: def L}.
\end{lem}

\begin{proof}
Let \(E\subset\BA\) be a \(\mu_{\BA}\)\nobreakdash–JM subset. By definition of \(N(x,T,E)\) we have
\begin{align*}
    \#N(x,T,E)
&= \frac{1}{m_{\R^{k+r-1}}(L)}
  \sum_{t\in N(x,T,E)} m_{\R^{k+r-1}}\bigl(\{\,t + \ell : \ell\in\Delta\}\bigr) \\
  &\leq \frac{M}{m_{\R^{k+r-1}}(L)}\;
    m_{\R^{k+r-1}}\Bigl(\bigcup_{t\in N(x,T,E)}\{t+\ell:\ell\in\Delta\}\Bigr),
\end{align*}
where last inequality follows from Lemma~\ref{lem:Tempered}, and \(\Delta\) is defined as in \eqref{eq: def L}. Observe that
\[
\bigcup_{t\in N(x,T,E)}\{t+\ell:\ell\in\Delta\}
= \bigl(J^T\cap\bigcup_{t\in N(x,T,E)}\{t+\ell:\ell\in\Delta\}\bigr)
  \;\cup\;R_T,
\]
where the remainder \(R_T\) has measure \(o(T^{k+r-1})\) as \(T\to\infty\).  But
\[
J^T\cap\bigcup_{t\in N(x,T,E)}\{t+\ell:\ell\in\Delta\}
= \{\,t\in J^T: a_t x\in E^\Delta\},
\]
so altogether
\begin{align}
\label{eq: aa 1}
    \#N(x,T,E)
\le \frac{M}{m_{\R^{k+r-1}}(L)}\;
    \Bigl(\;
      m_{\R^{k+r-1}}\{\,t\in J^T: a_t x\in E^\Delta\}
      +o(T^{k+r-1})
    \Bigr).
\end{align}
Since \(E\subset\BA\) is \(\mu_{\seda}\)\nobreakdash–JM, using steps similar to \cite[Lemma~9.3]{aggarwalghosh2024joint}, one can show that \(E^\Delta\subset\XA\) is \(\mu_{\XA}\)\nobreakdash–JM.  Hence, by the \((a_t,\mu_{\XA})\)\nobreakdash–genericity of \(x\) and Theorem~\ref{muJM}, the equation \ref{eq: aa 1} implies that
\[
\limsup_{T\to\infty}
\frac{\#N(x,T,E)}{m_{\R^{k+r-1}}(J^T)}
\;\le\;
\frac{M}{m_{\R^{k+r-1}}(L)}
\;\mu_{\XA}(E^\Delta),
\]
as claimed.
\end{proof}

\begin{proof}[Proof of Lemma~\ref{lem:BG implies CG}]
Fix a $\mu_{\seda}$-JM subset $E \subset \seda$. By Result \ref{res: 3} and Lemma~\ref{JMintersect}, for any small enough $\alpha > 0$, the set $E \cap (\seda)_{\geq \alpha}$ is $\mu_{\seda}$-JM. By Result \ref{res 6}, the set $(E \cap (\seda)_{\geq \alpha})^{J^\alpha}$ is $\mu_{\XA}$-JM. Therefore by Theorem~\ref{muJM} and Birkhoff genericity of $x$, we have
\begin{align}
\label{eq: bb 1}
    \lim_{T \to \infty} \frac{1}{m_{\R^{k+r-1}}(J^T)} \int_{J^T} \ind_{(E \cap (\seda)_{\geq \alpha})^{J^\alpha}}(a_t x) \, dt
    = \mu_{\XA}\left((E \cap (\seda)_{\geq \alpha})^{J^\alpha}\right).
\end{align}
Note that
\begin{align}
    \label{eq: bb 2}
    \int_{J^T} \ind_{(E \cap (\seda)_{\geq \alpha})^{J^\alpha}}(a_t x) \, dt
= m_{\R^{k+r-1}}(J^\alpha) \cdot \# N(x, T, E \cap (\seda)_{\geq \alpha}) + O(T^{k+r-2}).
\end{align}
Therefore equations \eqref{eq: bb 1} and \eqref{eq: bb 2} together imply that
\begin{align}
\label{eq: waste w 1}
    \lim_{T \to \infty} \frac{1}{m_{\R^{k+r-1}}(J^T)} \# N(x, T, E \cap (\seda)_{\geq \alpha})
    = \mu_{\seda}(E \cap (\seda)_{\geq \alpha}).
\end{align}
Now consider:
\begin{align*}
    \mu_{\seda}(E \cap (\seda)_{\geq \alpha})
    &= \lim_{T \to \infty} \frac{1}{m_{\R^{k+r-1}}(J^T)} \# N(x, T, E \cap (\seda)_{\geq \alpha}) \\
    &\leq \liminf_{T \to \infty} \frac{1}{m_{\R^{k+r-1}}(J^T)} \# N(x, T, E) \\
    &\leq \limsup_{T \to \infty} \frac{1}{m_{\R^{k+r-1}}(J^T)} \# N(x, T, E) \\
    &\leq \mu_{\seda}(E \cap (\seda)_{\geq \alpha}) + \frac{M}{m_{\R^{k+r-1}}(L)} \mu_{\XA}((E \cap (\seda)_{< \alpha})^\Delta),
\end{align*}
where the last inequality uses \eqref{eq: waste w 1} and Lemma~\ref{lem: tempered implies bound}. Since \(\alpha>0\) was arbitrary and the error term 
\[
 \frac{M}{m_{\R^{k+r-1}}(L)} \mu_{\XA}((E \cap (\seda)_{< \alpha})^\Delta)
\]
tends to zero as \(\alpha\to 0\), we deduce
\[
\lim_{T \to \infty}
\frac{1}{m_{\R^{k+r-1}}(J^T)}\,\#N(x,T,E)
\;=\;\mu_{\seda}(E).
\]
Hence proved.
\end{proof}

\medskip

We now prove the \textbf{if} part. We note that this result is not required for the proofs of the main theorems, but we include it here for completeness.
\begin{lem}[Cross-section Genericity implies Birkhoff Genericity]
\label{lem:CG implies BG}
Suppose $x \in \BA$ is such that for every $\mu_{\BA}$-JM subset $E \subset \BA$, we have
\begin{align*}
    \lim_{T \to \infty} \frac{1}{m_{\R^{k+r-1}}(J^T)} \# N(x, T, E) = \mu_{\BA}(E).
\end{align*} 
Then $x$ is $(a_t, \mu_{\XA})$-generic.
\end{lem}\begin{proof}
Let \( x \in \XA \) satisfy the hypothesis of the lemma. By Theorem~\ref{muJM}, we have
\[
\lim_{T \to \infty} \frac{1}{m_{\R^{k+r-1}}(J^T)} \sum_{t \in N(x, T, \BA)} \delta_{a_t x} = \mu_{\BA},
\]
which implies that for any open subset \( E \subset \BA \), 
\begin{equation} \label{eq:CSG}
\liminf_{T \to \infty} \frac{1}{m_{\R^{k+r-1}}(J^T)} \# N(x, T, E) \geq \mu_{\BA}(E).
\end{equation} 
\medskip

To deduce that \( x \) is \( (a_t, \mu_{\XA}) \)\nobreakdash–generic, we will show that for any open subset \( U \subset \XA \),
\begin{equation} \label{eq:flow-lower}
\liminf_{T \to \infty} \frac{1}{m_{\R^{k+r-1}}(J^T)} \int_{J^T} \delta_{a_t x}(U) \, dt \geq \mu_{\XA}(U).
\end{equation}
\medskip

\textbf{Step 1: Reduction to open sets in the saturated cross-section.}
By ergodicity of the \( a_t \)\nobreakdash–action on \( \XA \), and using Result~\ref{res 2:measure}, Result~\ref{res 5}, and the fact that \( \mu_{\seda}(\BA) > 0 \) (Lemma~\ref{lem: BA is JM}), the set
\[
(\BA \cap \ueda)^{\R^{k+r-1}} = \left\{ a_s y : y \in \BA \cap \ueda,\, s \in \R^{k+r-1} \right\}
\]
has full \( \mu_{\XA} \)\nobreakdash–measure. Hence, it suffices to verify~\eqref{eq:flow-lower} for every open set \( U \subset (\BA \cap \ueda)^{\R^{k+r-1}} \). \medskip

\textbf{Step 2: A basis of open sets.}
By Result~\ref{res 5} and fact that $\BA \cap \sseda$ is open in $\seda$ (proof of Lemma \ref{lem: BA is JM}), we have $\BA \cap \ueda$ is open in $\seda$. Therefore again by Result~\ref{res 5}, a basis for the topology on \( (\BA \cap \ueda)^{\R^{k+r-1}} \) is given by sets of the form
\[
a_s(E^I),
\]
where \( s \in \R^{k+r-1} \), \( I \subset [-\alpha, \alpha]^{k+r-1} \) is open, and \( E \subset \BA \cap \ueda \cap (\seda)_{\ge \alpha} \) is open in \( \seda \), for some $\alpha>0$. Moreover again by Result~\ref{res 5}, each such \( E \) is open in \( \BA \) and hence satisfies~\eqref{eq:CSG}.

\medskip

\textbf{Step 3: Verifying~\eqref{eq:flow-lower} on basis elements.}
Fix \( s, \alpha, E, I \) as above, and set \( U = a_s(E^I) \). Then,
\[
\int_{J^T} \delta_{a_t x}(U)\, dt
= \int_{J^T} \ind_{E^I}(a_{t-s} x)\, dt
= \int_{J^T - s} \ind_{E^I}(a_t x)\, dt.
\]
Since $$ \mathrm{vol}(J^T \setminus (J^T - s)) + \mathrm{vol}( (J^T - s) \setminus J^T )  = o(T^{k+r-1}), $$ it follows that
\[
\int_{J^T} \delta_{a_t x}(U)\, dt
= \int_{J^T} \ind_{E^I}(a_t x)\, dt + o(T^{k+r-1})
= m_{\R^{k+r-1}}(I) \cdot \# N(x, T, E) + o(T^{k+r-1}).
\]
Dividing both sides by \( m_{\R^{k+r-1}}(J^T) \), taking the \( \liminf \), and invoking~\eqref{eq:CSG}, we obtain
\[
\liminf_{T \to \infty} \frac{1}{m_{\R^{k+r-1}}(J^T)} \int_{J^T} \delta_{a_t x}(U)\, dt 
\ge m_{\R^{k+r-1}}(I) \cdot \mu_{\BA}(E)
= \mu_{\XA}(a_s(E^I)),
\]
where the final equality uses Result~\ref{res 2:measure} and \( a_t \)\nobreakdash–invariance of \( \mu_{\XA} \).

\medskip

Since any open \( U \subset (\BA \cap \ueda)^{\R^{k+r-1}} \) is a countable disjoint union of such basis elements, the lower bound~\eqref{eq:flow-lower} holds for all such \( U \). Hence, \( x \) is \( (a_t, \mu_{\XA}) \)\nobreakdash–generic, as required.
\end{proof}


\section{Final Proofs}
\label{sec: Final proofs}

Before proceeding with the final proofs, we record the following observation. 

It is sufficient to prove Theorems \ref{main thm best}, \ref{Main thm time visits}, and Corollaries \ref{thm:intro 1}, \ref{cor: determinant} for the case $j = 1$. Specifically, for Theorem \ref{main thm best}, consider the map
\[
    \phi_{j1} \colon \ZZ_1 \longrightarrow \ZZ_j
\]
defined by
\[
    \phi_{j1}(\Lambda, x, \gamma, v) = \left( A^j(x, \gamma) A^1(x, \gamma)^{-1} \Lambda,\, x,\, \gamma,\, v \right),
\]
where $\Lambda \in \E_d^1$, $x \in \Sphere^{m_1} \times \cdots \times \Sphere^{n_r}$, $\gamma \in [0, \varepsilon]$, and $v \in \hZp^d$. Here, $A^j$ and $A^1$ are as defined in Section~\ref{sec: The Relative Size}.

Observe that this map transforms the measure $\tmu^d$ to the measure $\tmu^j$ and satisfies
\[
    \phi_{j1} \circ \Theta_1 = \Theta_j.
\]
Therefore, taking $\tnu^j = (\phi_{j1})_* \tnu^1$ completes the reduction of Theorem~\ref{main thm best} to the case $j=1$.

A similar argument applies to Theorem~\ref{Main thm time visits} and Corollaries~\ref{thm:intro 1} and~\ref{cor: determinant}. Hence, in the upcoming proofs, we will restrict ourselves to the case $j = 1$.

\begin{proof}[Proof of Theorem \ref{main thm best}]
    \noindent\textbf{Claim.} For Lebesgue–almost every $\theta$, the best approximates $(p_l, q_l)$ appearing in the sum on the left-hand side of \eqref{eq: main thm general 2} that do not satisfy \eqref{eq: con 1}, \eqref{eq: con 2}, or \eqref{eq: con 3} contribute negligibly. To see this, we proceed in three steps:

\medskip

\noindent\textbf{1. Condition \eqref{eq: con 1} holds almost surely.}  
The set of $\theta \in \Mat$ for which \eqref{eq: con 1} fails is contained in a countable union of rational hyperplanes. Since each such hyperplane has Lebesgue measure zero, it follows that condition \eqref{eq: con 1} holds for Lebesgue–almost every $\theta$.

\medskip

\noindent\textbf{2. Negligible contribution from approximates failing \eqref{eq: con 2}.}  
Suppose $(p,q)$ is a best approximation of $\theta$ satisfying \eqref{eq: con 1} but violating \eqref{eq: con 2}. Let $I$ denote the set of indices $j$ for which \eqref{eq: con 2} fails. Without loss of generality, assume $I = \{s, s+1, \dots, r\}$. Write $q' = (q_1, \dots, q_{s-1})$, where $q_j = \rho_j'(q)$, and let $\theta'$ be the $m \times (n_1 + \cdots + n_{s-1})$ matrix obtained by removing the last $n_s + \cdots + n_r$ columns of $\theta$.

By the definition of best approximation, the lattice
\[
\begin{pmatrix}
I_m & \theta' \\ & I_{n_1 + \cdots + n_{s-1}}
\end{pmatrix}
\Z^{m + n_1 + \cdots + n_{s-1}}
\]
contains no nonzero point in the region
\[
B_{\|\rho_1(p + \theta' q')\|}^{m_1} \times \cdots \times B_{\|\rho_k(p + \theta' q')\|}^{m_k}
\times B_{\|q_1\|}^{n_1} \times \cdots \times B_{\|q_{s-1}\|}^{n_{s-1}}.
\]
By Minkowski’s second theorem, this implies
\begin{equation}
\label{123}
\left( \prod_{i=1}^k \|\rho_i(p + \theta' q')\|^{m_i} \right)
\left( \prod_{j=1}^{s-1} \|q_j\|^{n_j} \right) \leq c,
\end{equation}
for some constant $c > 0$ depending only on the choice of norms.

Again, from the definition of best approximation and the definition of $I$, we have
\begin{align}
\label{124}
\|\rho_i(p + \theta' q')\| &\leq 1 \quad \text{for all } 1 \leq i \leq k, \\
\label{125}
\|q_j\| &\neq 0 \quad \text{for all } 1 \leq j < s.
\end{align}
But the number of such approximates $(p, q')$ satisfying \eqref{123}--\eqref{125} is of order $o(T^{k + r - 1})$ for Lebesgue–almost every $\theta$ (see for e.g. \cite[Thm.~1.3]{aggarwalghosh2024joint}). Hence, these may be ignored in the asymptotic count in \eqref{eq: main thm general 2}.

\medskip

\noindent\textbf{3. Negligible contribution from approximates failing \eqref{eq: con 3}.}  
Write
\[
\theta =
\begin{pmatrix}
\theta_1 \\ \theta_2
\end{pmatrix}, \qquad
\theta_1 \in \M_{m_1 \times n}(\R), \quad
\theta_2 \in \M_{(m - m_1) \times n}(\R).
\]
Suppose $(p, q)$ is a best approximation satisfying \eqref{eq: con 1} and \eqref{eq: con 2}, but violating \eqref{eq: con 3}. Write $p = (p', p'')$, where $p' = \rho_1(p)$ and $p''$ denotes the remaining components. Then, from the assumptions, we have
\begin{align*}
\|(p'' + \theta_2 q)_i\| &\leq 1 \quad \text{for all } 2 \leq i \leq k, \\
0< \left( \prod_{i=2}^k \|(p'' + \theta_2 q)_i\|^{m_i} \right)
\left( \prod_{j=1}^r \|\rho_j'(q)\| \right) &< 1,
\end{align*}
where $(p'' + \theta_2 q)_i$ denotes the projection onto the $\R^{m_i}$ component of $\R^{m - m_1} = \R^{m_2} \times \cdots \times \R^{m_k}$.

Again the number of such approximates is of order $o(T^{k + r - 1})$ for Lebesgue–almost every $\theta$ (see for e.g. \cite[Thm.~1.3]{aggarwalghosh2024joint}). Hence, these may also be ignored in \eqref{eq: main thm general 2}.

\medskip

\noindent Fix such a $\theta$ as in the claim above—namely, one for which it suffices to consider only the best approximates $(p_l, q_l)$ satisfying \eqref{eq: con 1}, \eqref{eq: con 2}, and \eqref{eq: con 3}. Moreover, assume that $\theta$ is of generic type with respect to the decomposition $m = m_1 + \cdots + m_k$, $n = n_1 + \cdots + n_r$; this again holds for Lebesgue–almost every $\theta$ by Corollary~\ref{lem: Birkhoff genericity}. We now claim that \eqref{eq: main thm general 2} holds for such $\theta$.

    To see this, note that by Lemmas \ref{lem: Generic iff adele generic }, and \ref{lem:BG implies CG}, along with Theorem \ref{muJM}, we obtain
    \begin{align}
        \label{eq:waste w 3}
        \lim_{T \rightarrow \infty} \frac{1}{m_{\R^{k+r-1}}(J^T)} \sum_{t \in N(\Tilde\Lambda_\theta, T, \BA)} \delta_{a_t \Tilde\Lambda_\theta} = \mu_{\BA},
    \end{align}
    with respect to the weak-* topology. This implies that the number of best approximates $(p, q)$ for which the corresponding $t$ satisfies $a_t \Tilde\Lambda_\theta \in \BA \setminus \ssed$ is also of order $o(T^{k + r - 1})$, and hence can also be ignored in \eqref{eq: main thm general 2}.

    Now, push forward the measures on both sides of \eqref{eq:waste w 3} under the maps $\Tilde{\psi}$ and $\Tilde{\psi}^-$, and using Lemma \ref{lem: Cross-section correspondence} along with equation \eqref{eq: measure of J T}, we obtain
    \[
        \lim_{T \rightarrow \infty} \frac{1}{T^{k + r - 1}} \sum_{\|q\| \leq e^T} \delta_{\Theta_1(\theta, p, q)} = \frac{c_{k + r - 1}(n_1, \ldots, n_r)}{m_1 \cdots n_{r-1} \cdot (k + r - 1)!} \left( \Tilde{\psi}_* \mu_{\BA} + \Tilde{\psi}^-_* \mu_{\BA} \right).
    \]
    This proves \eqref{eq: main thm general 2} for 
    \[
        \tnu^1 = \frac{\zeta(d)}{m_1 \cdots n_{r - 1}} \left( \Tilde{\psi}_* \mu_{\BA} + \Tilde{\psi}^-_* \mu_{\BA} \right).
    \]

    Finally, we verify the stated properties of $\tnu^1$. The first property follows from Result \ref{res 7} and the fact that $\mu_{\BA}$ is the restriction of $\mu_{\seda}$ to $\BA$. The second property follows from the $K_f$-invariance of $\BA$ (by definition), $\mu_{\seda}$ (by Result \ref{res 2:measure}), and the maps $\Tilde{\psi}, \Tilde{\psi}^-$ (by Result \ref{res 7}). The third property of $\tnu^1$ follows from the fact that both $\mu_{\seda}$ and $\BA$ are invariant under the action of the group
    \begin{align}
    \label{eq: def compact group}
        \left\{ 
        \begin{pmatrix}
            O_1 \\ & \ddots \\ && O_{k + r} 
        \end{pmatrix} \;\middle|\; 
        \begin{aligned}
            &O_i \in \Ortho(m_i) \text{ for } 1 \leq i \leq k, \\
            &O_{k + j} \in \Ortho(n_j) \text{ for } 1 \leq j \leq r
        \end{aligned} 
        \right\},
    \end{align}
    and from the fact that the composition of the maps $\Tilde{\psi}, \Tilde{\psi}^-$ with the natural projection $\ZZ_1 \rightarrow \Sphere^{m_1} \times \cdots \times \Sphere^{m_k} \times \Sphere^{n_1} \times \cdots \times \Sphere^{n_r} \times \R$ is also invariant under the action of the compact group in \eqref{eq: def compact group}. This completes the proof.
\end{proof}

\begin{proof}[Proof of Theorem \ref{Main thm time visits}]
Fix a $\tnu^{1}$-JM subset $A$ of $\ZZ_1$ satisfying \eqref{eq: condition Main thm time visits}. Fix $\theta \in \Mat$ of generic type. Then, by Lemma \ref{lem:BG implies CG}, we have
\begin{align}
    \label{eq:waste w 3 1}
    \lim_{T \rightarrow \infty} \frac{N(\Tilde\Lambda_\theta, T, E)}{T} = \mu_{\BA}(E),
\end{align}
for every $\mu_{\BA}$-JM subset $E \subset \BA$.

Let $(p_l, q_l) \in \Z^m \times \Z^n$ denote the sequence of $\e$-approximations to $\theta$, ordered by increasing $\|q_l\|$, such that $\Theta_1(\theta, p_l, q_l) \in A$. Let $t_l = \log \|q_l\|$, and define $\Y_T(\theta, A) = \{t_l : t_l \leq T\}$.

By Theorem \ref{muJM}, it suffices to show that for all $j \in \N$, there exists a probability measure $\mu^{A,j}$ on $(\ZZ_1 \times \R)^j$ such that
\begin{align}
    \label{eq: main thm time visits waste}
    \lim_{T \rightarrow \infty} \frac{1}{T} \sum_{\|q_l\| \leq e^T} \delta_{\left( \Theta_1(\theta, p_{l+i}, q_{l+i}), \log \left( \frac{\|q_{l+i}\|}{\|q_{l+i-1}\|} \right) \right)_{i=0}^{j-1}} = \mu^{A,j}.
\end{align}
The existence of a measure $\mu^A$ with the desired properties then follows from the Kolmogorov extension theorem applied to the family $(\mu^{A,j})_{j \in \N}$.

To prove \eqref{eq: main thm time visits waste}, let $\Tilde{A} = \Tilde{\psi}^{-1}(A)$. Then, using Theorem \ref{muJM}, Lemma \ref{JMintersect}, and equation \eqref{eq:waste w 3 1}, we have
\begin{align}
    \label{eq:waste w 3 2}
    \lim_{T \rightarrow \infty} \frac{1}{T} \sum_{t \in N(\Tilde\Lambda_\theta, T, \Tilde{A})} \delta_{a_t \Tilde\Lambda_\theta} = \mu_{\BA}|_{\Tilde{A}}.
\end{align}

Arguing as in the proof of Lemma \ref{lem: Cross-section correspondence}, we see that for all $T$, 
\begin{equation*}
    N(\Tilde\Lambda_\theta, T, \Tilde{A}) \subset \Y_T(\theta, A), \quad
    \Y_T(\theta, A) \setminus N(\Tilde\Lambda_\theta, T, \Tilde{A}) \subset N(\Tilde\Lambda_\theta, T, \BA \cap \sseda).
\end{equation*}
Moreover, using \eqref{eq:waste w 3 1}, we get
\[
\#N(\Tilde\Lambda_\theta, T, \BA \cap \sseda) = o(T).
\]
Therefore, we may restrict the sum in \eqref{eq: main thm time visits waste} to those $(p_l, q_l)$ for which
\begin{align}
    \label{eq: abced 1}
    \{t_l, \ldots, t_{l+j-1}\} \cap N(\Tilde\Lambda_\theta, T, \BA \cap \sseda) = \emptyset.
\end{align}

For each such $(p_l, q_l)$, we have
\begin{align}
    \label{eq:temp11}
    \tau_A\left( T_A^s(a_{t_l} \Tilde\Lambda_\theta) \right) = \log \|q_{l+s+1}\| - \log \|q_{l+s}\|,
\end{align}
for all $0 \leq s \leq j-1$. The maps $\tau_{\Tilde{A}}, T_{\Tilde{A}}^1, \ldots, T_{\Tilde{A}}^{j-1}$ (Result \ref{res 9}) and $\Tilde{\psi}, \Tilde{\psi}^{-1}$ (Result \ref{res 7}) are continuous $\mu_{\sseda}|_{\Tilde{A}}$-almost everywhere. Therefore, pushing forward both sides of \eqref{eq:waste w 3 2} under the map
\[
\left( (\Tilde{\psi}, \tau_{\Tilde{A}}), \ldots, (\Tilde{\psi}, \tau_{\Tilde{A}}) \circ T_{\Tilde{A}}^{j-1} \right),
\]
and using \eqref{eq:temp11} and Lemma \ref{lem: Cross-section correspondence}, we obtain the desired convergence in \eqref{eq: main thm time visits waste}. This completes the first part of the proof.

For the second part, note that the pushforward measure $\mu^{A*}$ in the theorem equals $\mu^{A,1}$ as defined above. The first property of $\mu^{A,1}$ then follows directly. For the second property, observe that the expectation of the pushforward of $\mu^{A*}$ under the projection map $\ZZ_1 \times \R \rightarrow \R$ equals
\[
\int_{\Tilde{A}} \tau_A(x) \, d\mu_{\sseda}|_{\Tilde{A}},
\]
which equals $1$ by the Kac formula \eqref{Kac formula}. Hence, the theorem follows.
\end{proof}

\begin{proof}[Proof of Corollary \ref{thm:intro 1}]
Note that using Remark \ref{rem: Generic has full measure}, it is enough to show that the limits in \eqref{eq: thm:intro 1 1}, \eqref{eq: thm:intro 1 2}, and \eqref{eq: thm:intro 1 3} exist for every $\theta$ of generic type, and are independent of $\theta$.

\smallskip

Equation \eqref{eq: thm:intro 1 1} follows from Theorem \ref{Main thm time visits} in the same way that Corollary \ref{thm: cor 3 to main thm} follows from Theorem \ref{main thm best}.

\smallskip

For equation \eqref{eq: thm:intro 1 3}, fix $s \in \Z_{\geq 0}$. If $s = 0$, then \eqref{eq: thm:intro 1 3} amounts to the study of the distribution of the sequence $\{\disp(\theta, p_l, q_l) : l \in \N\}$, which follows directly from Theorem \ref{Main thm time visits}. Hence, we may assume that $s > 0$. Consider the continuous map
\[
F: (\ZZ_1 \times \R)^s \rightarrow \R, \quad (\Lambda_i, w_i, \delta_i, v_i, \gamma_i )_{i=1}^s \longmapsto e^{\gamma_1 + \cdots + \gamma_s} \delta_i,
\]
where $\Lambda_i \in \E_d^1$, $w_i \in \Sphere^{m_1} \times \cdots \times \Sphere^{n_r}$, $\delta_i \in [0,\epsilon]$, $v_i \in \hZp^d$, and $\gamma_i \in \R$ for all $1 \leq i \leq s$. Then, using Theorem \ref{muJM}, the equation \eqref{eq: thm:intro 1 3} follows by pushing forward both sides of equation \eqref{eq: main thm time visits waste} by $F$.

\smallskip

To prove \eqref{eq: thm:intro 1 2}, note that the arguments used so far do not directly yield the result. However, the strategy developed above—the use of cross-sections—can be adapted to prove it.

More precisely, we need to redefine the cross-section to suit the quantity of interest. Let
\[
\Le = \{(x, y) \in \Sphere^m \times \R^n : \|y\|^n \leq \epsilon\},
\]
and consider the new cross-section $\seda = \pi^{-1}(\X(\Le))$. With this choice, the analogues of the results in Sections \ref{sec: Results borrowed}, \ref{sec: Properties of the Cross-section}, and \ref{sec: Cross-section Genericity} still hold.

The main difference lies in the analogue of Lemma \ref{lem: Cross-section correspondence}. In this setting, we obtain a correspondence between best approximates (considered up to a sign) with $\|p + \theta q\| \geq e^{-T}$ and points in the set
\[
N(\tilde{\Lambda}_\theta, T, \BA) = \{t \in [0, T] : a_t \tilde{\Lambda}_\theta \in \BA\}.
\]
This correspondence leads to the identity
\[
\lim_{l \to \infty} \frac{1}{l} \log \|p_l + \theta q_l\| = \left( \lim_{T \to \infty} \frac{N(\tilde{\Lambda}_\theta, T, \Tilde{A}')}{-T} \right)^{-1},
\]
for suitable $ \Tilde{A}'$ depending on $A$, and since the limit on the right exists for every generic $\theta$, it follows that \eqref{eq: thm:intro 1 2} holds.

\smallskip

To prove \eqref{eq: thm:intro 1 4}, note that, since the limits in \eqref{eq: thm:intro 1 1} and \eqref{eq: thm:intro 1 2} exist, it suffices to show that
\begin{align}
\label{eq: disp waste}
   n \lim_{j \rightarrow \infty} \frac{\log \|q_{l_j}\|}{l_j}  + m \lim_{j \rightarrow \infty} \frac{\log\|p_{l_j} +\theta q_{l_j}\|}{l_j}= \lim_{j \rightarrow \infty} \frac{\log \disp(\theta,p_{l_j}, q_{l_j})}{ l_j}  =0,
\end{align}
for some subsequence $(l_j)$. The equation \eqref{eq: disp waste} follows if we can show that there exist constants $0 < \alpha < \beta < \infty$ such that $\disp(\theta, p_l, q_l) \in [\alpha, \beta]$ for infinitely many $l$.

This in turn follows from Theorem \ref{main thm best}, which implies that the error terms $\disp(\theta, p_l, q_l) $ equidistribute with respect to a measure absolutely continuous with respect to Lebesgue measure. Hence, we can choose $0 < \alpha < \beta < \infty$ such that $[\alpha, \beta]$ has positive measure under the limiting distribution, proving \eqref{eq: thm:intro 1 4}.

\smallskip

This completes the proof of the corollary.
\end{proof}

\begin{proof}[Proof of Corollary \ref{cor: determinant}]
Let us define a function $F: (\ZZ_1 \times \R)^{d} \rightarrow \R$ by
\[
    (\Lambda_i, x_i, y_i, \delta_i, v_i, \gamma_i )_{i=1}^d \longmapsto \det \begin{pmatrix}
        \delta_1^{1/m} x_1 & e^{-\tfrac{n}{m} \gamma_1} \delta_2^{1/m} x_2 & \cdots & e^{- \tfrac{n}{m} \left(\gamma_1 + \cdots+ \gamma_{d-1} \right)} \delta_d^{1/m} x_d \\
        y_1 & e^{\gamma_1}y_2 & \cdots & e^{\gamma_1 + \cdots+ \gamma_{d-1}}y_d
    \end{pmatrix}
\]
where $\Lambda_i \in \E_d^1$, $x_i \in \Sphere^{m}$, $y_i \in \Sphere^n$, $\delta_i \in [0,\varepsilon]$, $v_i \in \hZp^d$, and $\gamma_i \in \R$ for all $1 \leq i \leq d$.

Then for all $l$, observe that
\begin{align*}
    \det \begin{pmatrix}
        p_{l} & \cdots & p_{l+m-1} \\
        q_l & \cdots & q_{l+m-1}
    \end{pmatrix}
    &= \det \begin{pmatrix}
        p_{l} + \theta q_l & \cdots & p_{l+m-1} + \theta q_{l+m-1} \\
        q_l & \cdots & q_{l+m-1}
    \end{pmatrix} \\
    &= \det \begin{pmatrix}
        \|q_l\|^{n/m} p_{l} + \theta q_l & \cdots & \|q_l\|^{n/m} p_{l+m-1} + \theta q_{l+m-1} \\
        \|q_l\|^{-1} q_l & \cdots & \|q_l\|^{-1} q_{l+m-1}
    \end{pmatrix} \\
    &= F\left( \left(\Theta_1(\theta,p_{l+j}, q_{l+j}), \log\frac{\|q_{l+j+1}\|}{\|q_{l+j}\|} \right)_{j=0}^{d-1} \right)
\end{align*}

Therefore, using Theorem \ref{muJM}, equation \eqref{eq: cor: determinant} follows by pushing forward both sides of \eqref{eq: main thm time visits waste} for $s = d$, and taking $\mu^{A,\Z} = F_*(\mu^{A,j})$. This completes the proof.
\end{proof}


\appendix

\section{Comparison and Heuristic}
\label{appendix:best-approximation-remarks}

This brief section addresses two common queries raised by readers. The first concerns the distinction between Definition~\ref{def:best approx k=m, n=1} and Definition~\ref{def: best k=r=1} for the $\ell^\infty$-norm. The second concerns the heuristic behind Cheung’s conjecture—specifically, why the expression $(\log q_\ell)^m/\ell$ appears in one setting~\eqref{eq:Cheung Conjecture}, but only $\log q_\ell/\ell$ in the other~\eqref{eq: CC19 1}. For ease of explanation, throughout this section, we assume that $m=2$ and $n=1$.

\subsection*{Comparison of Definitions}

We first explain the difference between Definition~\ref{def:best approx k=m, n=1} and Definition~\ref{def: best k=r=1} in the case of the $\ell^\infty$-norm. To see the distinction, note that a vector $(p,q)$ is a best approximation of $\theta \in \mathbb{R}^2$ in the sense of Definition~\ref{def: best k=r=1} (Cheung--Chevallier) if, for all integers $q' < q$ and all $p' \in \mathbb{Z}^2$, we have
\begin{align}
    \label{eq: xa 1}
    \max\{ |p_1 + q \theta_1| , |p_2 + q \theta_2| \} < \max \{ |p_1'+ q'\theta_1|, |p_2'+ q'\theta_2| \}.
\end{align}
In contrast, $(p,q)$ is a best approximation of $\theta$ in the sense of Definition~\ref{def:best approx k=m, n=1} if for all $q' < q$ and all $p' \in \mathbb{Z}^2$, we have
\begin{align}
\label{eq: xa 2}
    \text{either} \quad |p_1 + q \theta_1| <  |p_1'+ q'\theta_1| \quad \text{or} \quad |p_2 + q \theta_2| <  |p_2'+ q'\theta_2|,
\end{align}
and
\begin{align}
    \label{eq: xa 3}
    \max\{|p_1+ q\theta_1|, |p_2+ q\theta_2|\} \leq \tfrac{1}{2}.
\end{align}

It is evident that equation~\eqref{eq: xa 1} implies both \eqref{eq: xa 2} and \eqref{eq: xa 3}. Thus, any best approximation in the sense of Definition~\ref{def: best k=r=1} is also a best approximation in the sense of Definition~\ref{def:best approx k=m, n=1}. However, the converse does not hold.

Indeed, \eqref{eq: xa 2} only requires that for each $q' < q$, the point $q\theta + p$ is closer to \emph{at least one} coordinate axis than $q'\theta + p'$. This axis may vary with $q'$. That is, $p+ q\theta$ may be closer to the $x$-axis than some earlier approximation, and closer to the $y$-axis than another. In contrast, \eqref{eq: xa 1} demands that $p+ q\theta$ be uniformly closer to the origin in the $\ell^\infty$-norm than every earlier approximation. So \eqref{eq: xa 2} allows $(p,q)$ to be considered a best approximation even if it is not strictly better in all directions at once.

To illustrate this distinction, consider $\theta = (\tfrac{1}{5}, \tfrac{1}{7})$ and let $p = (-1, -1)$, $q = 5$. Then
\begin{align*}
    |p_1 + q \theta_1| = |-1 + 5 \cdot \tfrac{1}{5}| = 0,
\end{align*}
so for all $q' < 5$ and all $p' \in \mathbb{Z}^2$, we have
\begin{align*}
    |p_1 + q \theta_1| = 0 < |p_1' + q' \theta_1|,
\end{align*}
showing that $(p,q)$ satisfies Definition~\ref{def:best approx k=m, n=1}. However, it does not satisfy Definition~\ref{def: best k=r=1}, since
\begin{align*}
    \max\left\{ |-1 + 5 \cdot \tfrac{1}{5}| , |-1 + 5 \cdot \tfrac{1}{7}| \right\} 
    = \max\left\{ 0, \tfrac{2}{7} \right\} = \tfrac{2}{7} 
    > \max\left\{ \tfrac{1}{5}, \tfrac{1}{7} \right\} = \tfrac{1}{5}.
\end{align*}

This example demonstrates that the condition in Definition~\ref{def:best approx k=m, n=1} is strictly weaker than that in Definition~\ref{def: best k=r=1}.

\subsection*{Geometric Interpretation}
A second way to visualize the difference is via the geometric interpretation discussed in Remark~\ref{rem: difference}. To assist the reader, we include a top-down view of the corresponding regions in the figure below.

\begin{figure}[H]
    \centering
    \begin{minipage}{0.45\textwidth}
        \centering
        \includegraphics[width=0.9\textwidth]{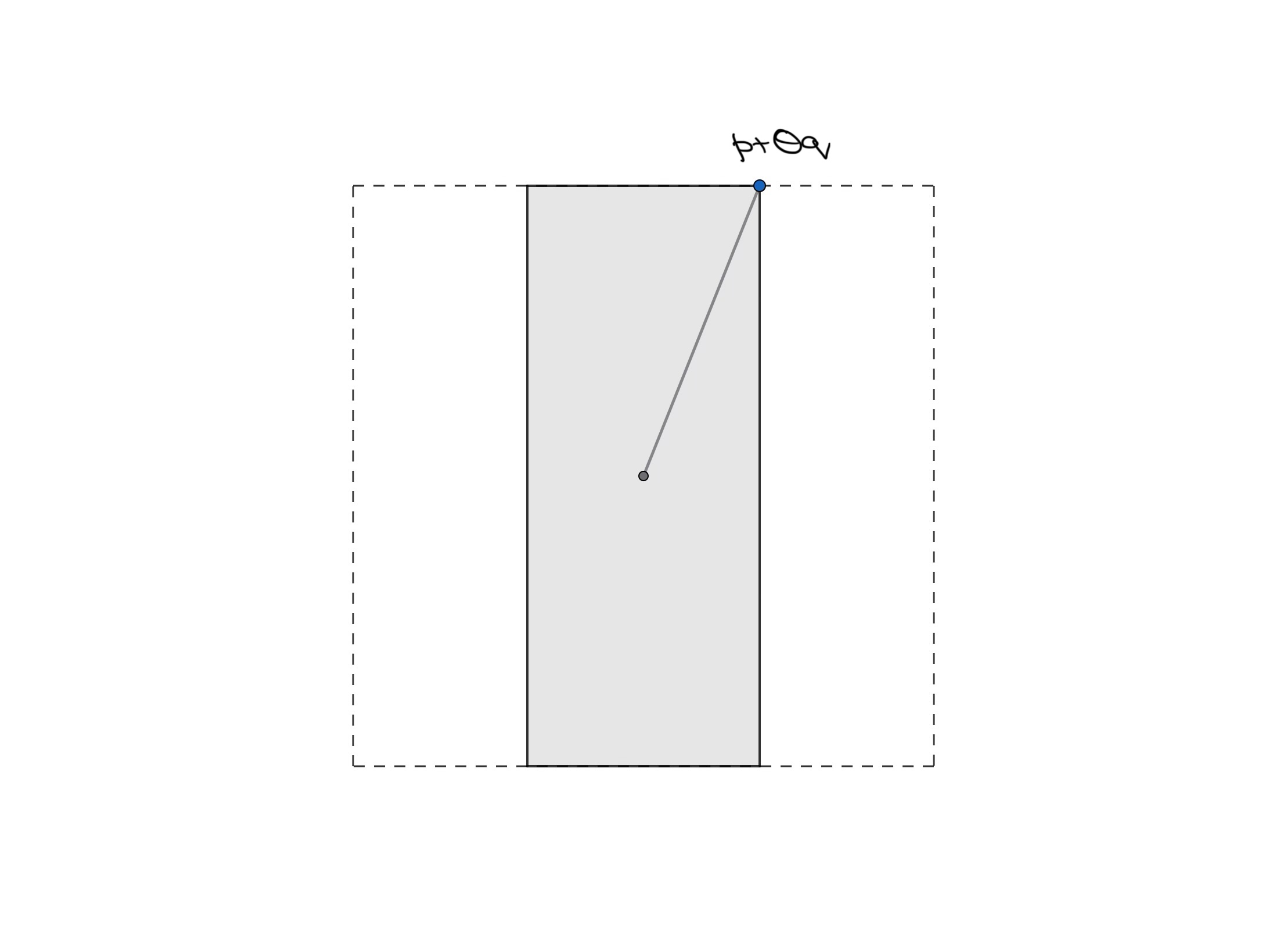}
        \textit{Definition \ref{def:best approx k=m, n=1}}
    \end{minipage}
    \hfill
    \begin{minipage}{0.45\textwidth}
        \centering
        \includegraphics[width=0.9\textwidth]{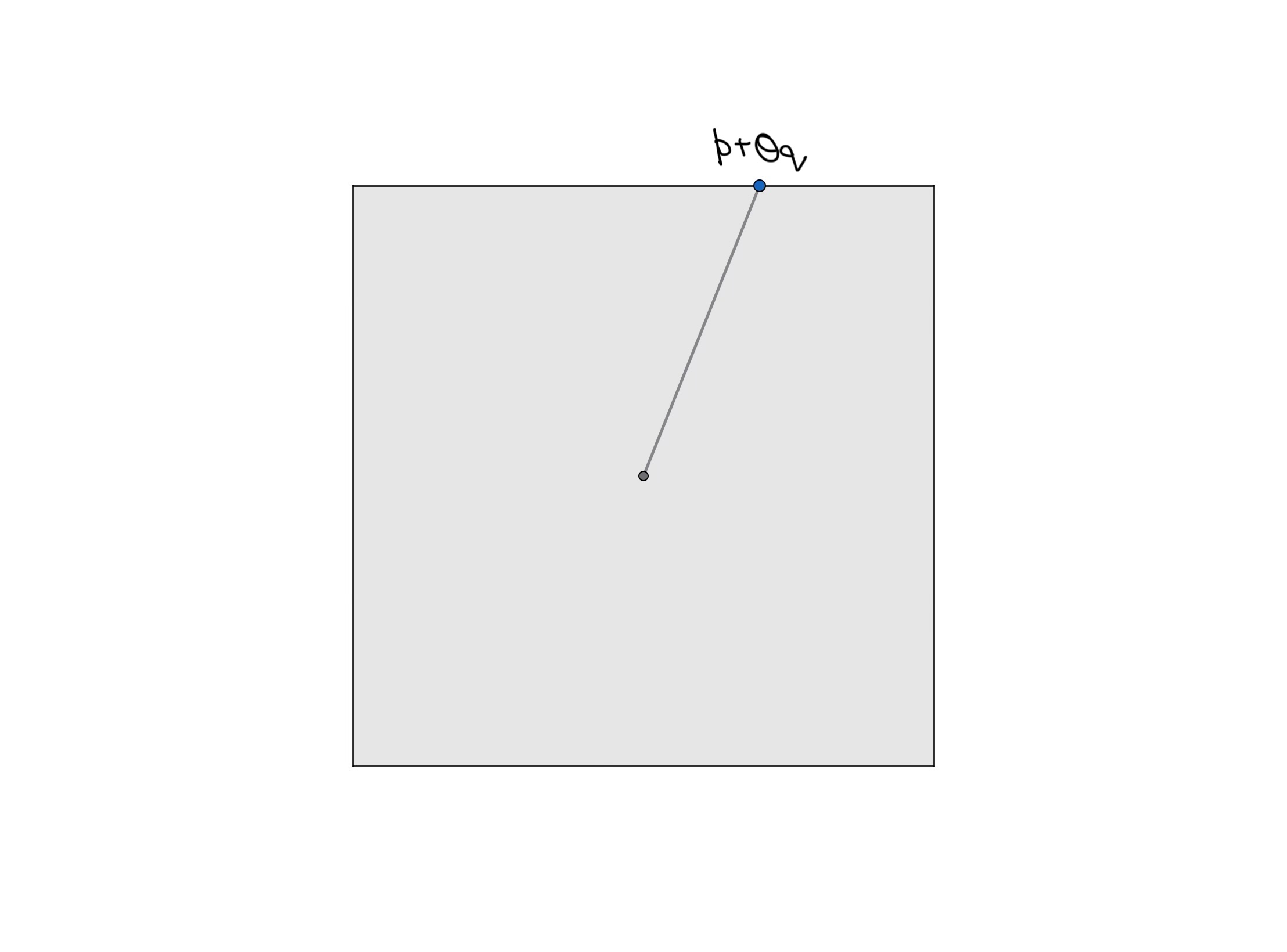}
        \textit{Definition \ref{def: best k=r=1}}
    \end{minipage}
    \caption{Visual comparison between the top view of cylinder (Cheung--Chevallier) and cuboid (Cheung's conjecture) conditions}
    \label{fig:geometric_interpretation_3}
\end{figure}

Since the rectangle with $p + q \theta$ as one of its vertices (used in the Cheung's conjecture setting) is always contained within the square centered at the same point (used in the Cheung--Chevallier setting), this again confirms that the condition in Definition~\ref{def: best k=r=1} is stronger.

\subsection*{Heuristic for Growth Rates in Cheung’s Conjecture}

The asymmetry between Definitions~\ref{def:best approx k=m, n=1} and~\ref{def: best k=r=1} has important implications for the growth of denominators in sequences of best approximations. Since the condition in Definition~\ref{def:best approx k=m, n=1} is weaker, it is typically easier to find the next best approximant $q_{\ell+1}$ once $q_1, \ldots, q_\ell$ are known. In contrast, the stricter condition in Definition~\ref{def: best k=r=1} makes such approximants rarer—and therefore, $q_{\ell+1}$ tends to be significantly larger.

This rough intuition accounts for the difference in growth rates:
\begin{align*}
    \log q_\ell &\sim \ell^{1/2} \quad \text{under Definition~\ref{def:best approx k=m, n=1},} \\
    \log q_\ell &\sim \ell \quad \text{under Definition~\ref{def: best k=r=1}.}
\end{align*}
This, in turn, explains why expressions like $(\log q_\ell)^m / \ell$ and $\log q_\ell / \ell$ naturally arise in the respective formulations of Cheung’s conjecture.

\vspace{1em}
We now provide a more detailed geometric explanation. Suppose that $(p_1,q_1), \ldots, (p_4,q_4)$ are best approximants of $\theta \in \R^2$ in the sense of Definition~\ref{def: best k=r=1} for the $\ell^\infty$ norm. To find the next approximant $(p,q)$, one must ensure that:
\begin{align*}
    \|p_i + \theta q_i\|_{\ell^\infty} > \|p + \theta q\|_{\ell^\infty} \quad \text{for each } i = 1,\dots,4.
\end{align*}
Geometrically, these inequalities define nested exclusion regions in the torus $\T^2 \simeq (-1/2, 1/2]^2$. The point $q\theta \mod 1$ must lie in the intersection of these allowed regions—illustrated in the series of diagrams below.

\begin{figure}[H]
    \centering
    \includegraphics[width=0.3\textwidth]{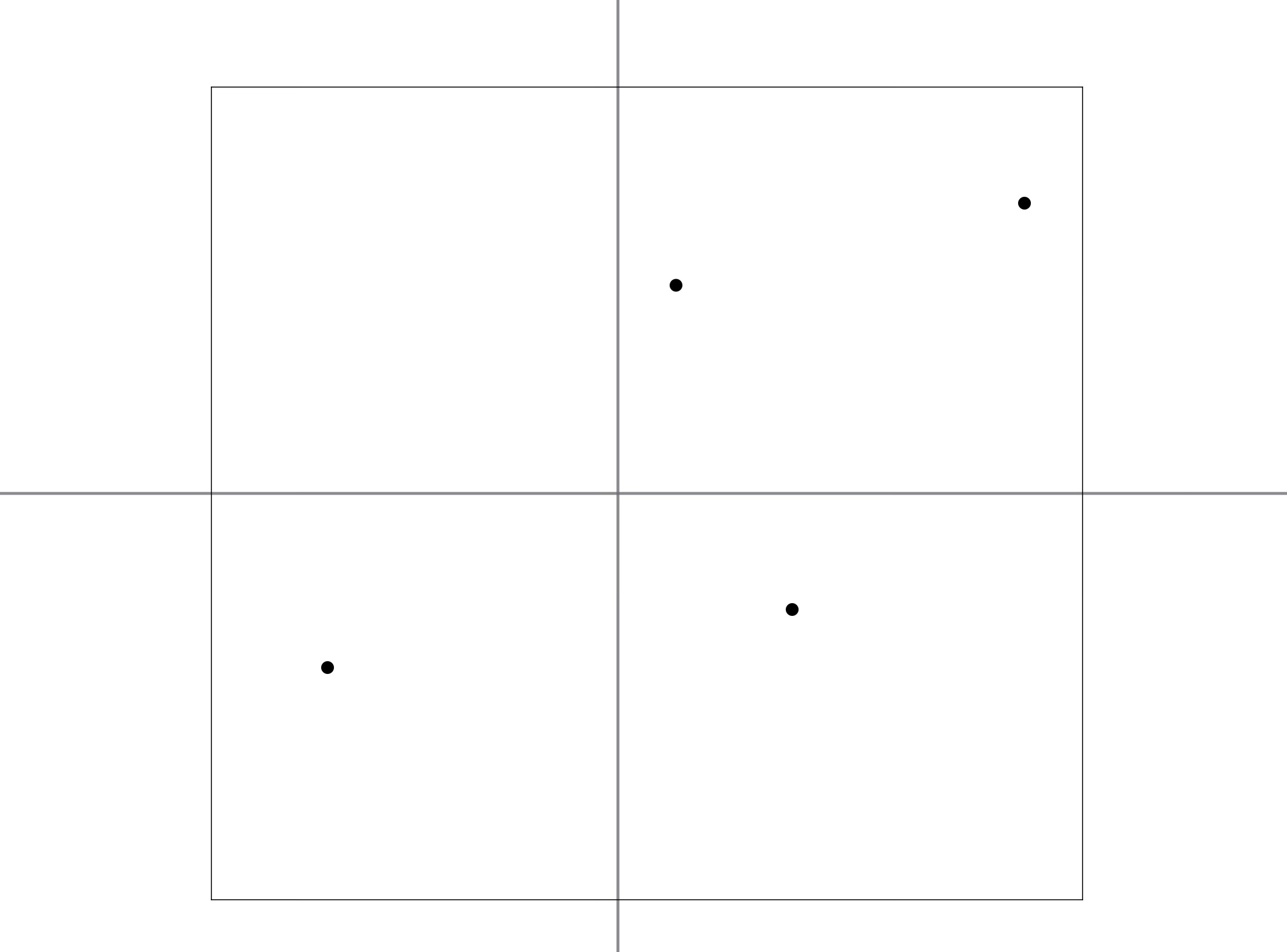}
    \caption{Initial four best approximants: $q_i \theta$ shown in the torus.}
\end{figure}

\begin{figure}[H]
    \centering
    \includegraphics[width=0.3\textwidth]{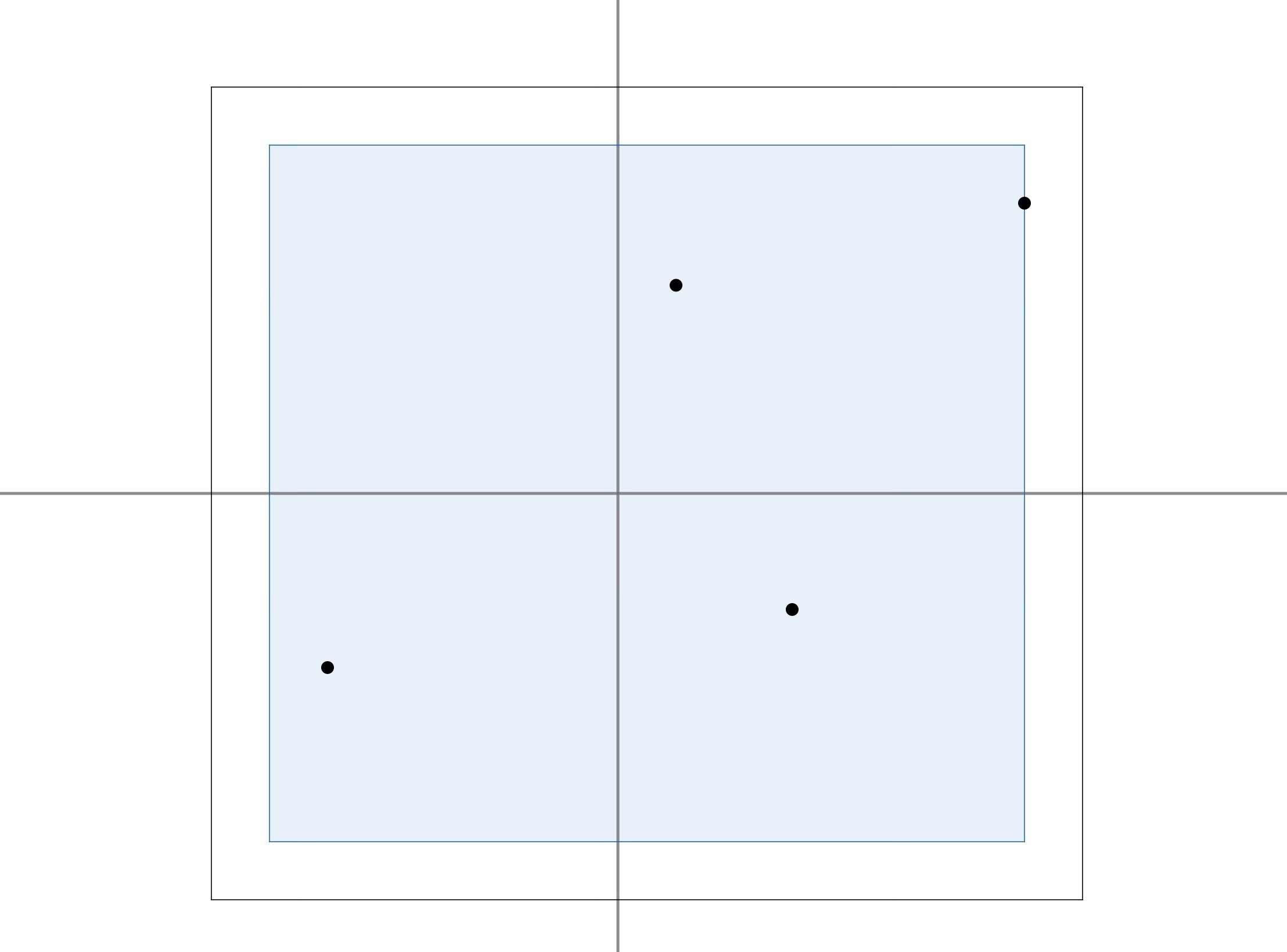}
    \caption{Region (blue) allowed by the first inequality.}
\end{figure}

\begin{figure}[H]
    \centering
    \includegraphics[width=0.3\textwidth]{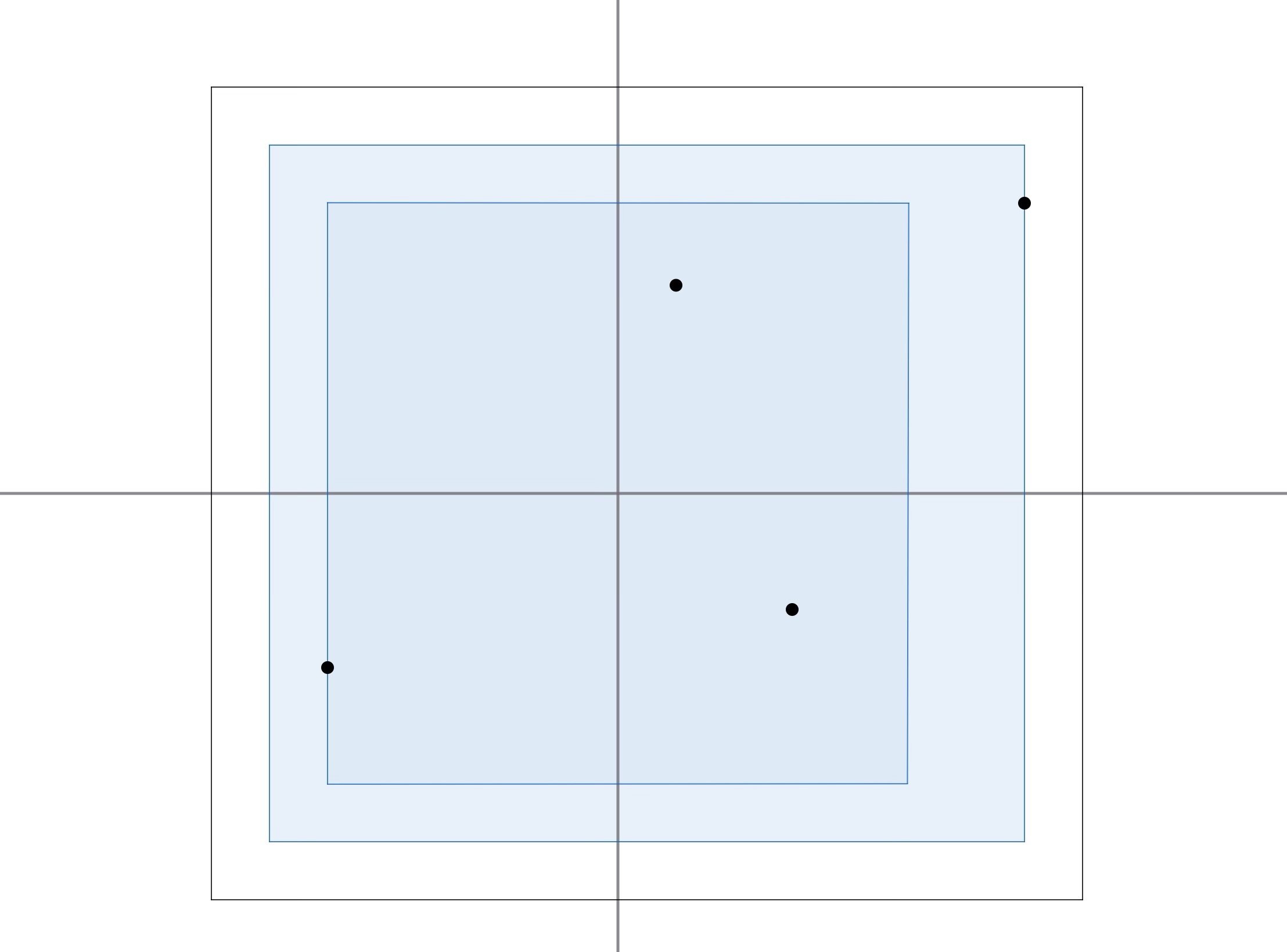}
    \caption{Refined region (blue) from the second inequality.}
\end{figure}

\begin{figure}[H]
    \centering
    \includegraphics[width=0.3\textwidth]{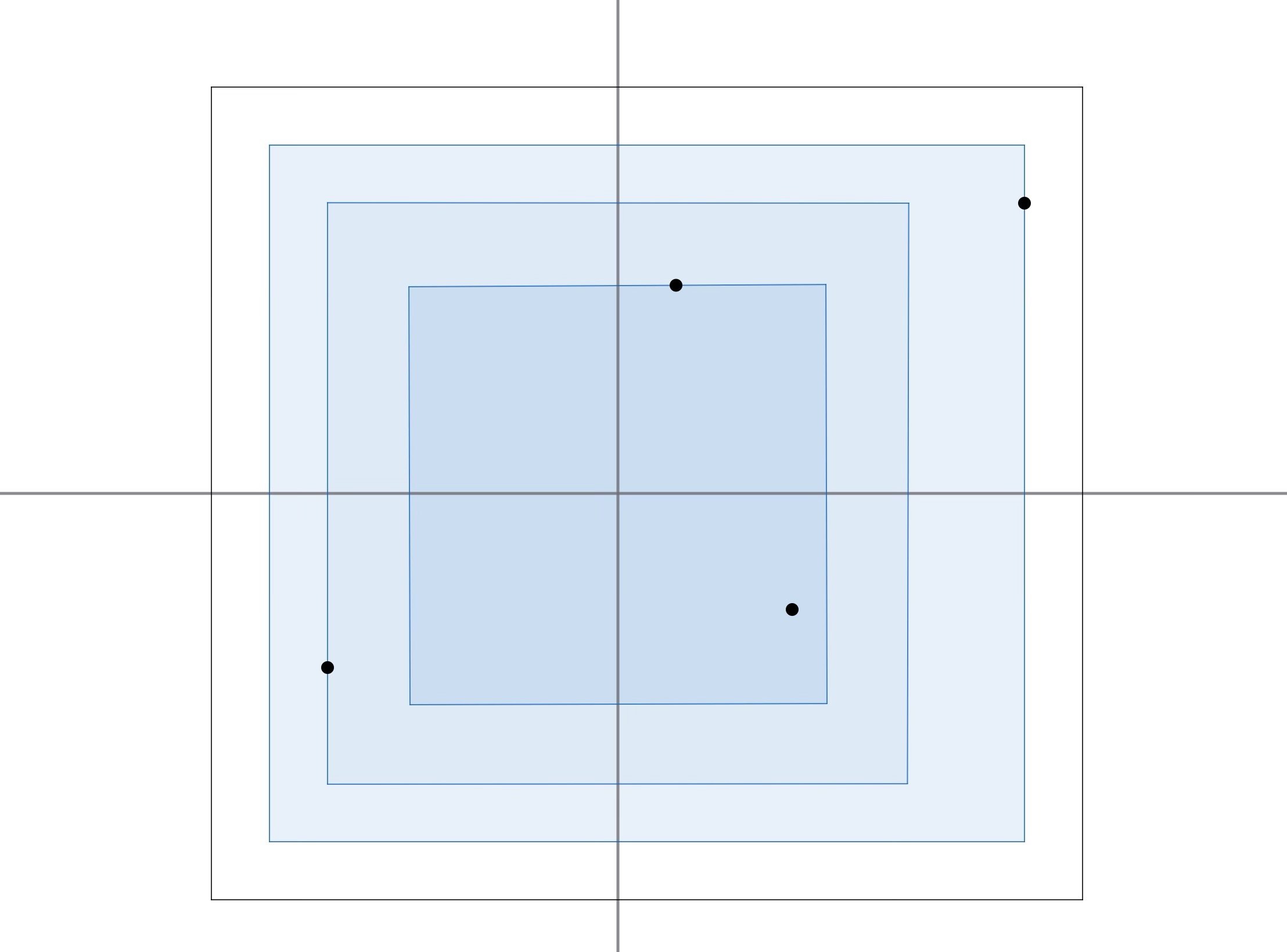}
    \caption{Further refinement after three conditions.}
\end{figure}

\begin{figure}[H]
    \centering
    \includegraphics[width=0.3\textwidth]{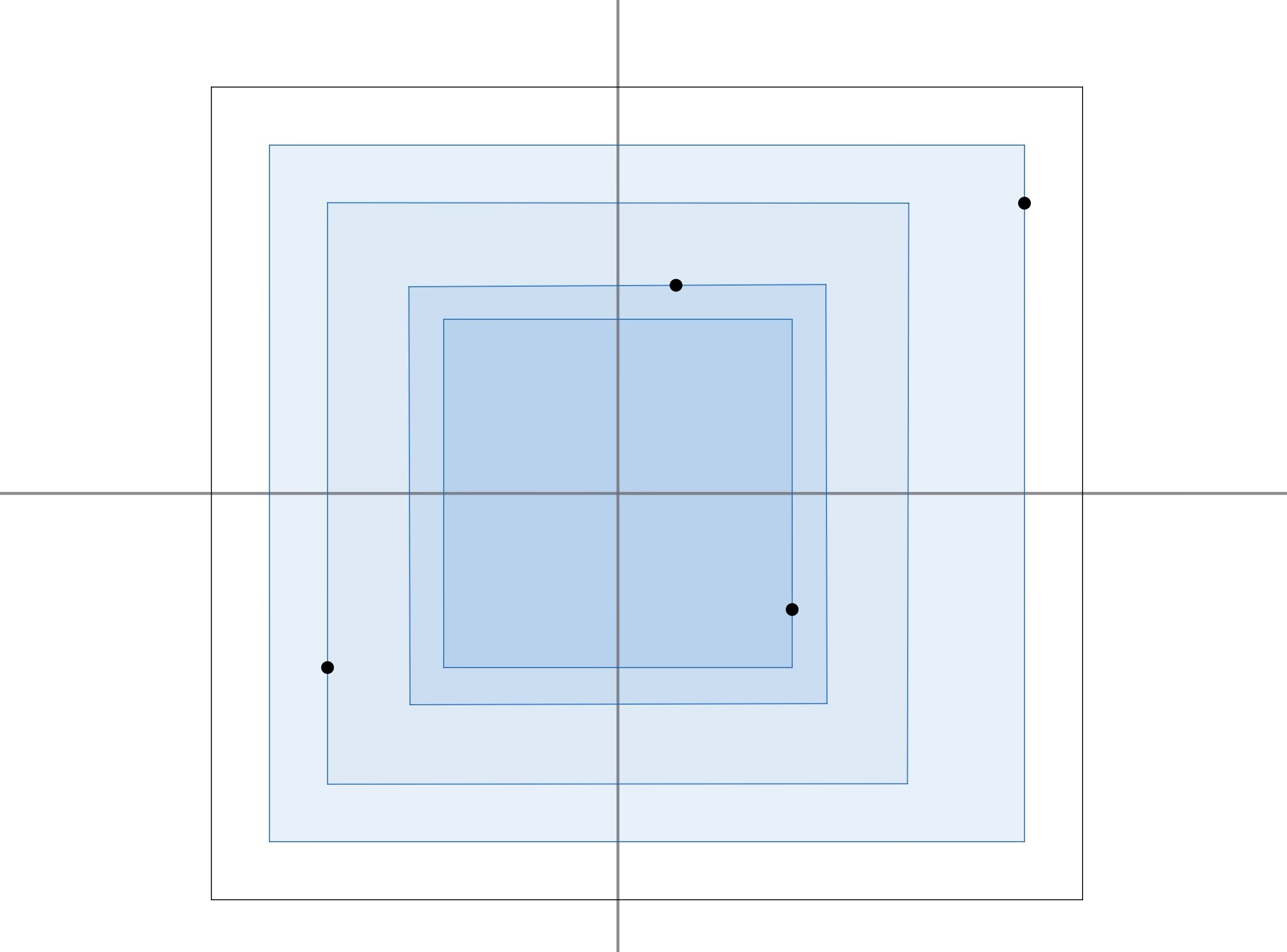}
    \caption{Final feasible region (blue) after all four inequalities.}
\end{figure}

In this setting, finding $(p, q)$ is akin to solving a simultaneous Diophantine inequality, i.e., locating points of the orbit $\{ q\theta + \Z^2 : q \in \N \}  \subset \T^2$ in progressively shrinking rectangles. By Schmidt's counting results, the number of best approximants with $\|q\| \leq e^T$ therefore grows like $T$. Thus, the discussion in Section~\ref{subsec: Generalized Levy-Khintchine Theorem} leads to the heuristic $q_\ell \sim \ell$. This explains the linear growth in $\log q_\ell$.

\vspace{1em}
Now consider the same task under Definition~\ref{def:best approx k=m, n=1}, where the previous approximants are $(p_1,q_1), \ldots, (p_5,q_5)$. The condition for a new approximant $(p,q)$ becomes: for each $i$, the point $q\theta \mod 1$ must avoid a certain blue region in the torus. 

These blue regions correspond to points that would violate the condition
\[
|(p_i)_j + q_i\theta_j| > |x_j| \quad \text{for some } j,
\]
or in other words satisfy
\[
|(p_i)_j + q_i\theta_j| \leq  |x_j| \quad \text{for both } j.
\]
That is, if $q\theta$ enters the blue zone, then there exists a smaller denominator $q_i$ giving a better approximation in both coordinate, making $(p, q)$ inadmissible. These exclusion zones accumulate with each new approximant.

The diagrams below visualize this behavior:

\begin{figure}[H]
    \centering
    \includegraphics[width=0.3\textwidth]{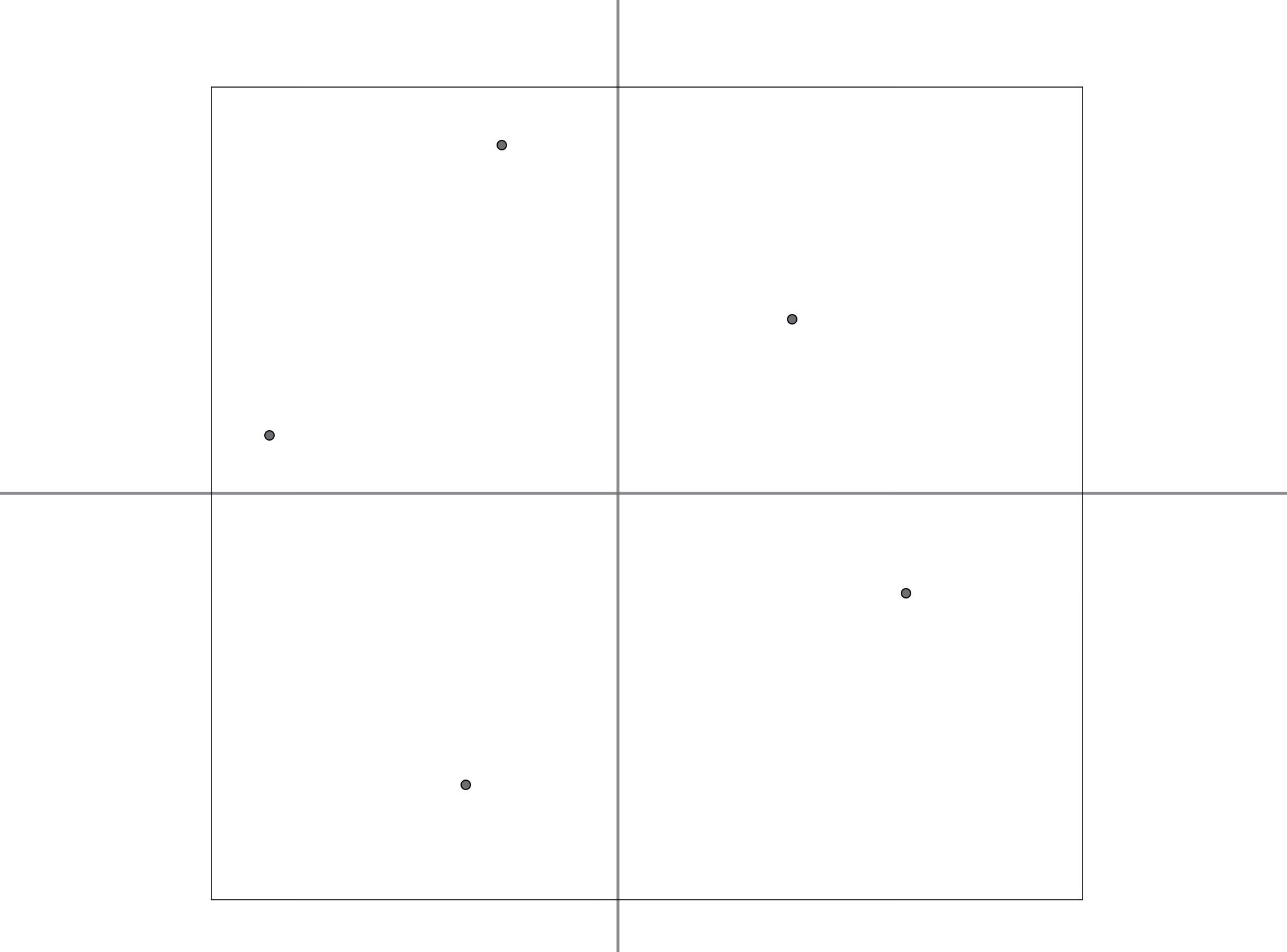}
    \caption{Initial five best approximants: $q_i \theta$ shown.}
\end{figure}

\begin{figure}[H]
    \centering
    \includegraphics[width=0.3\textwidth]{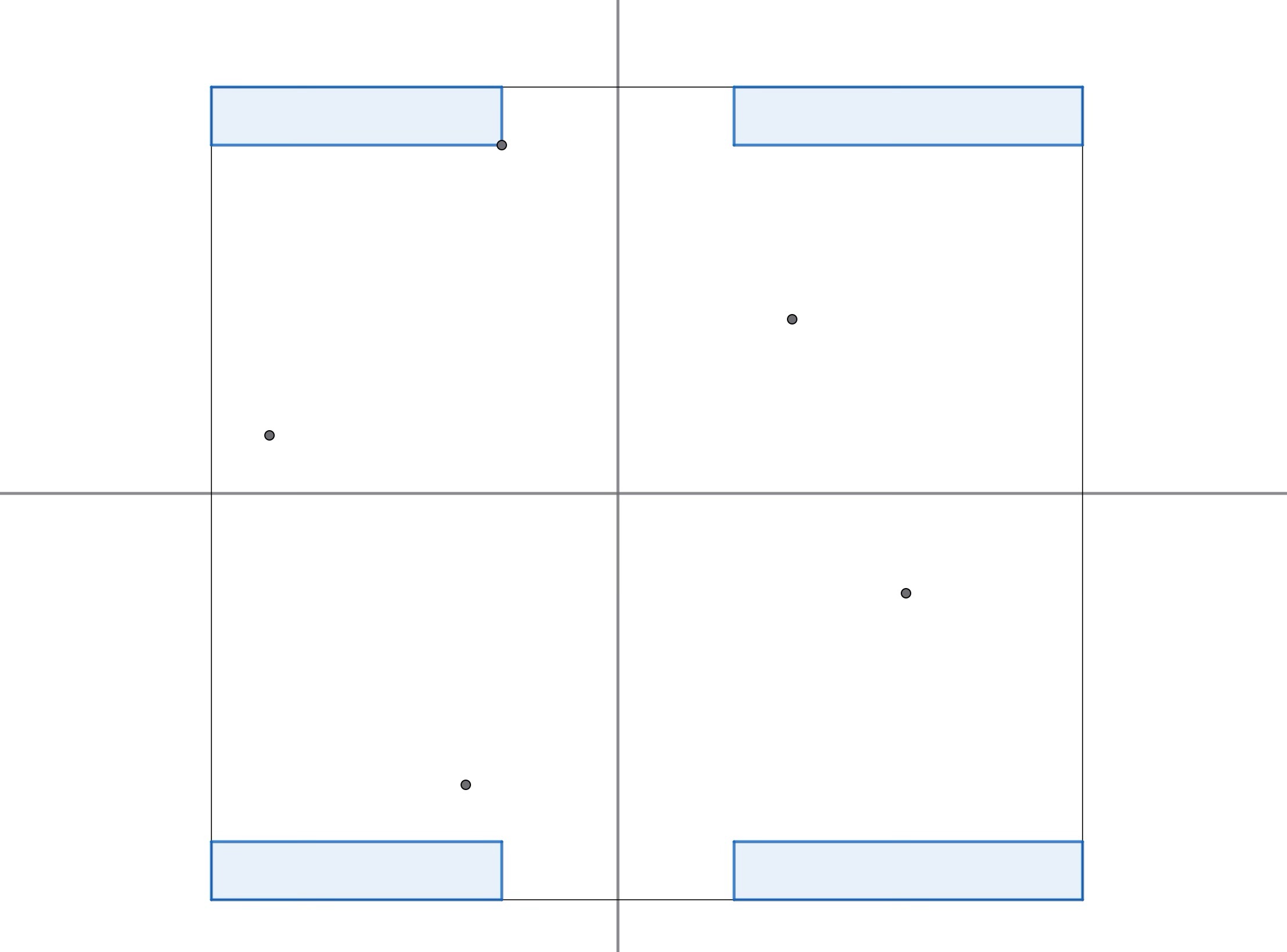}
    \caption{Forbidden region for $(p_1, q_1)$.}
\end{figure}

\begin{figure}[H]
    \centering
    \includegraphics[width=0.3\textwidth]{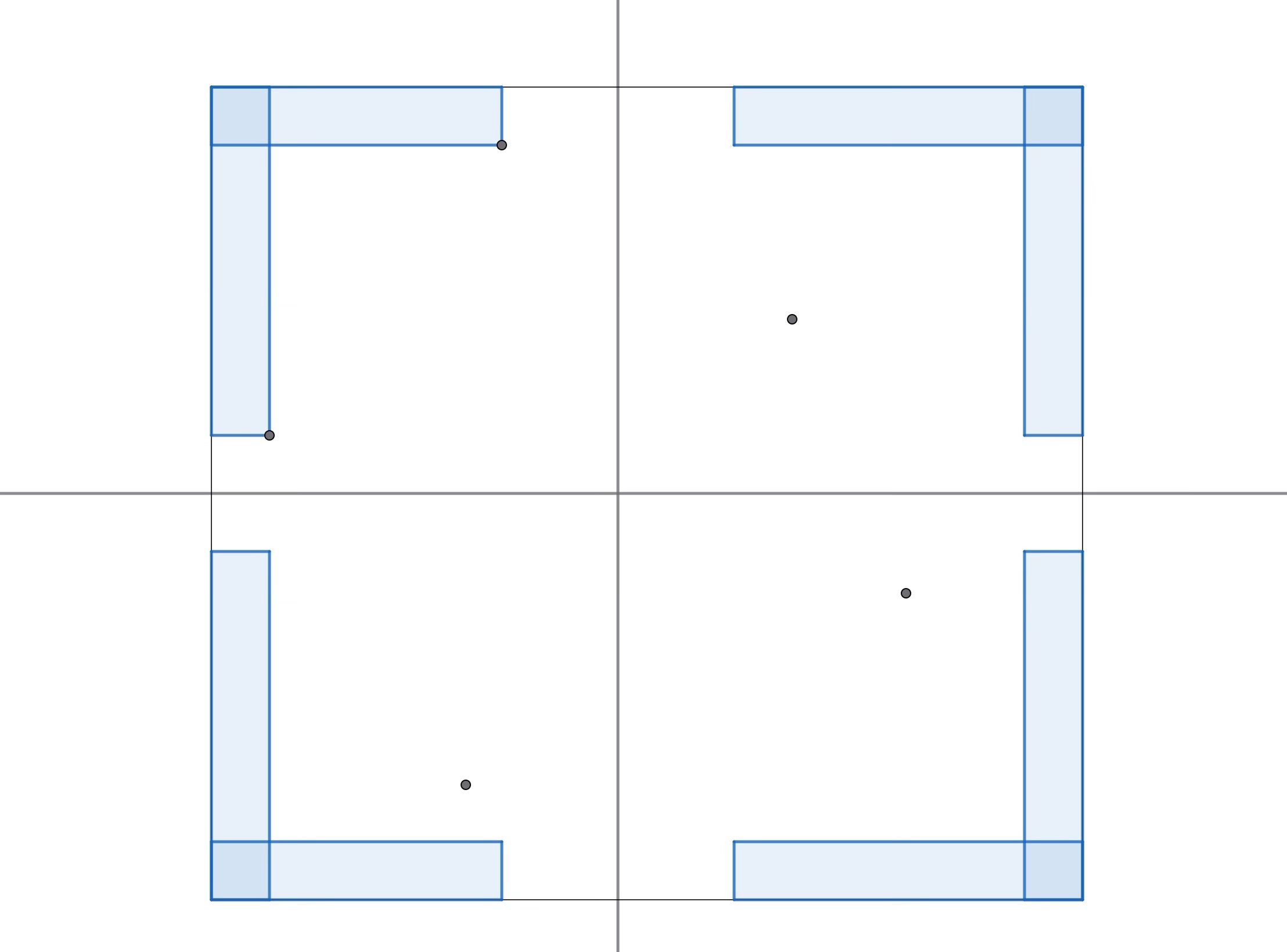}
    \caption{Forbidden region for $(p_2, q_2)$.}
\end{figure}

\begin{figure}[H]
    \centering
    \includegraphics[width=0.3\textwidth]{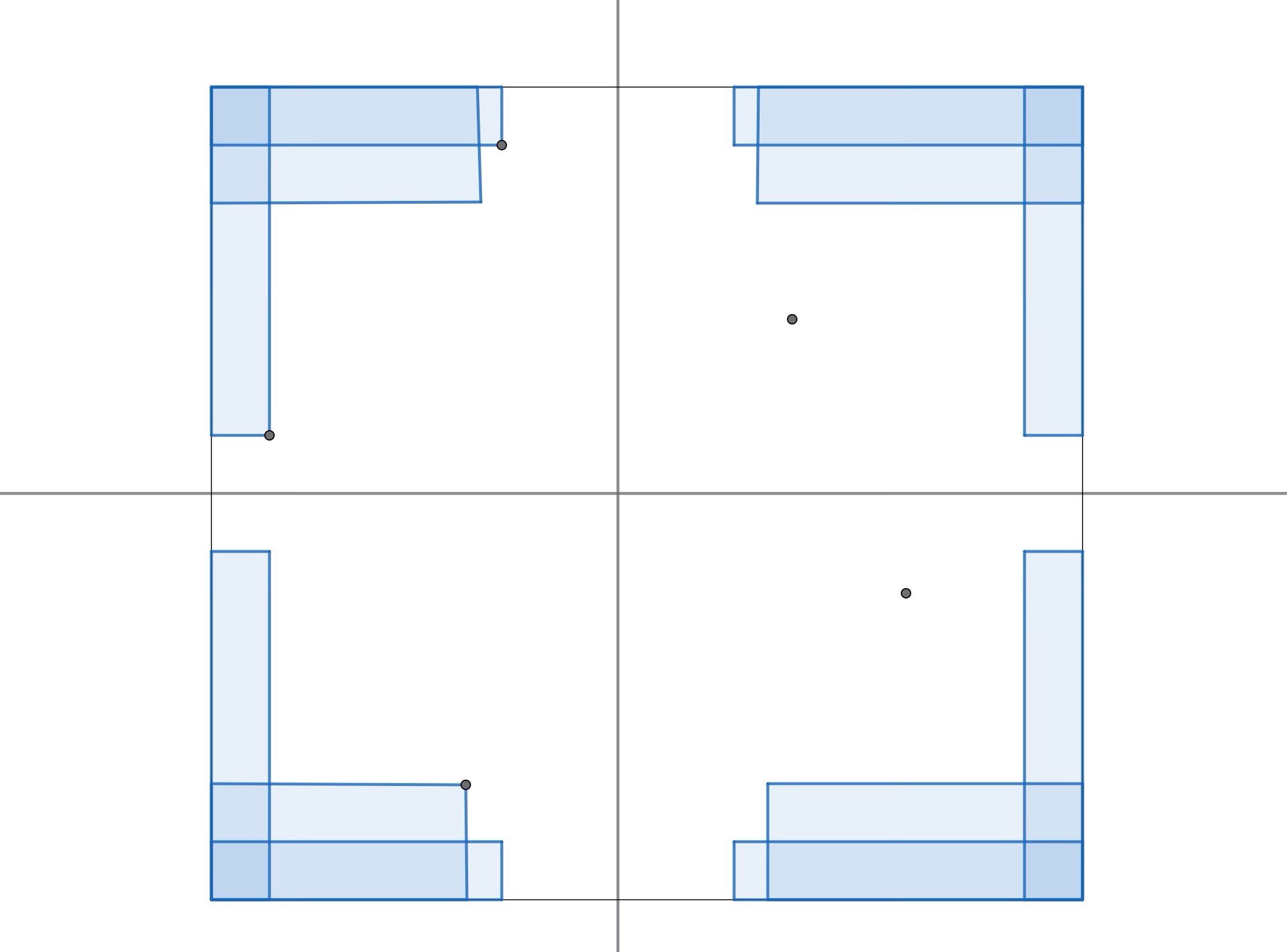}
    \caption{Continuing exclusion: $(p_3, q_3)$.}
\end{figure}

\begin{figure}[H]
    \centering
    \includegraphics[width=0.3\textwidth]{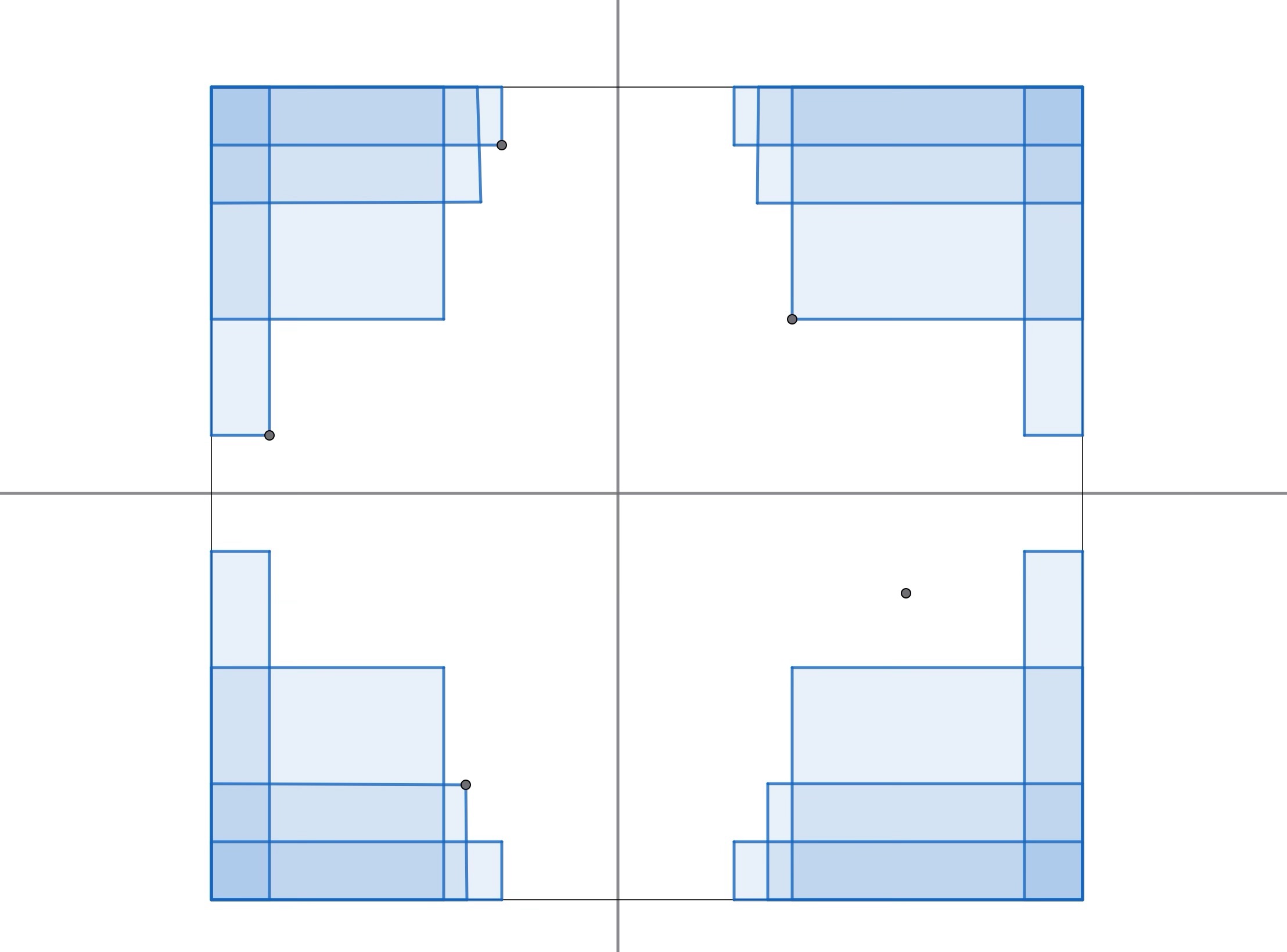}
    \caption{For $(p_4, q_4)$.}
\end{figure}

\begin{figure}[H]
    \centering
    \includegraphics[width=0.3\textwidth]{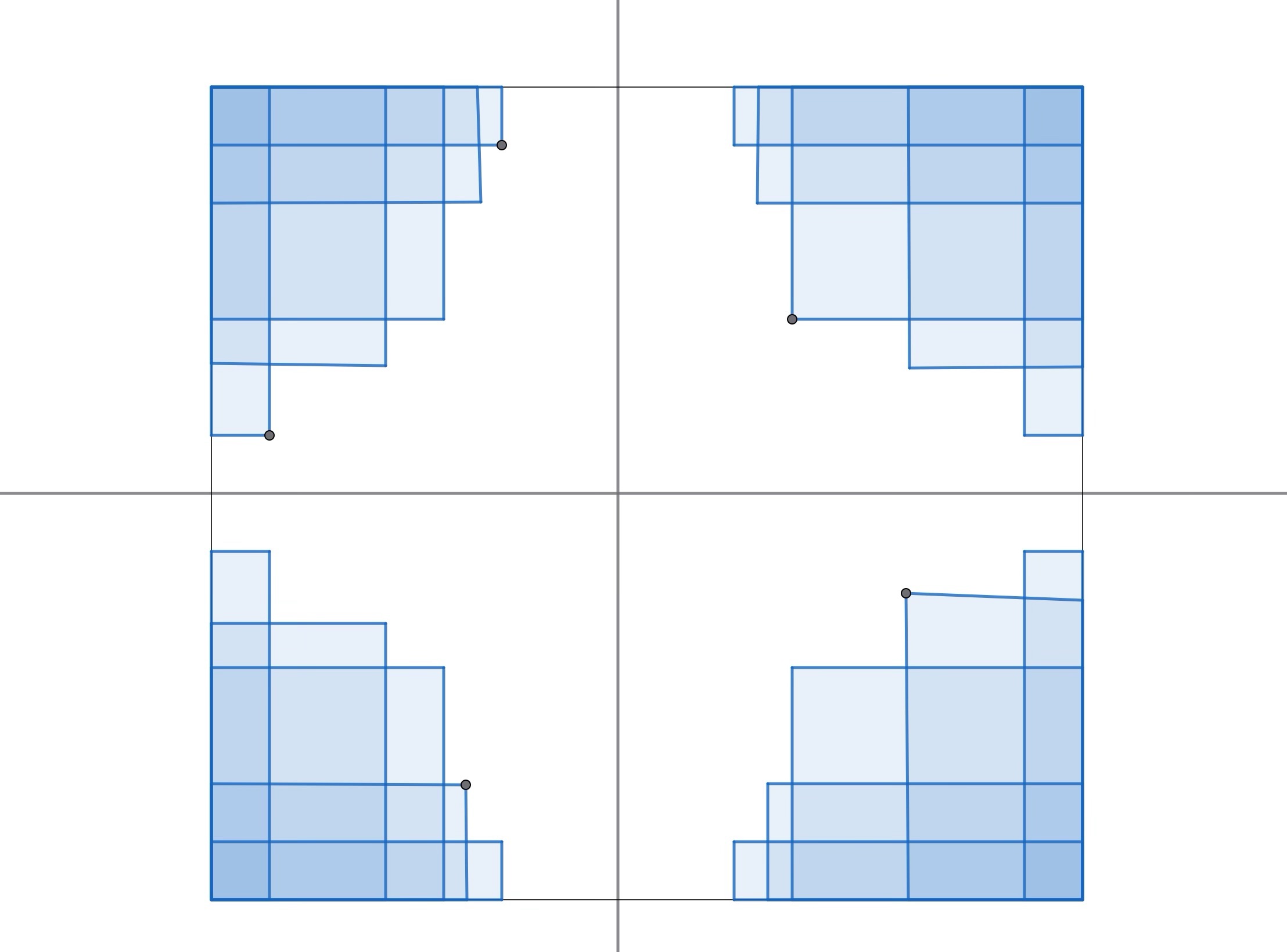}
    \caption{For $(p_5, q_5)$.}
\end{figure}

With many such conditions, the admissible region becomes hyperbolic in shape—roughly described by
\[
|x||y| \leq \varepsilon.
\]
This corresponds to solving a multiplicative Diophantine inequality:
\[
q \cdot |q\theta + p_1| \cdot |q\theta + p_2| < \varepsilon.
\]

\begin{figure}[H]
    \centering
    \includegraphics[width=0.3\textwidth]{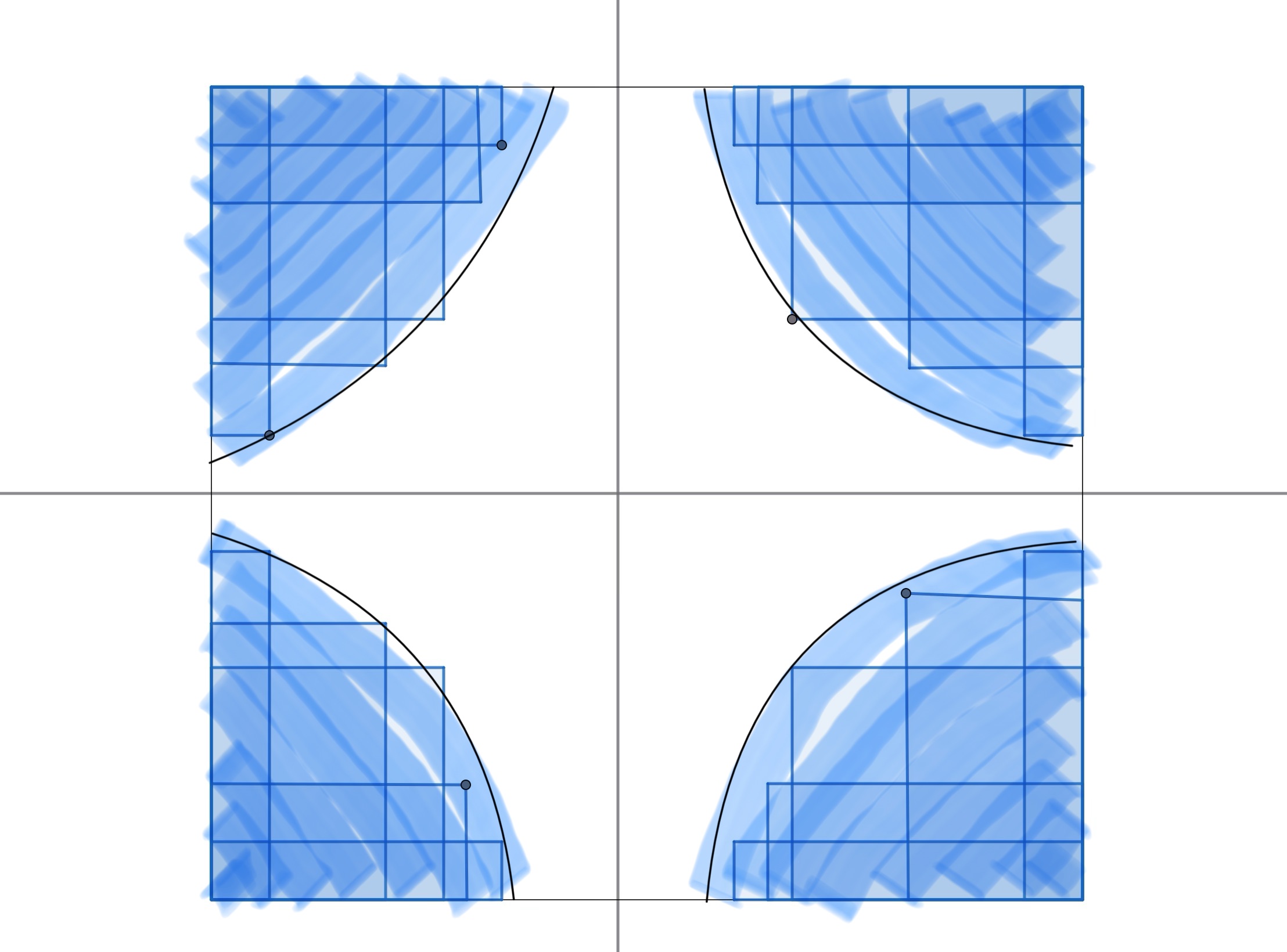}
    \caption{Limiting shape: a hyperbola-like region.}
\end{figure}

Since multiplicative approximations are more permissive, we expect to find solutions faster: if simultaneous approximations appear at scale $q \sim T$, multiplicative ones should emerge at scale $q \sim T^{1/2}$. Consequently, $\log q_\ell$ grows like $\ell^{1/2}$ in this regime.

\vspace{1em}
In summary, this explains—at least heuristically—why:
- Definition~\ref{def: best k=r=1} leads to linear growth: $\log q_\ell \sim \ell$,
- Definition~\ref{def:best approx k=m, n=1} leads to sublinear growth: $\log q_\ell \sim \ell^{1/2}$, and hence, why the respective forms of L\'{e}vy-Khintchine theorem involve $(\log q_\ell)^m/\ell$ and $\log q_\ell/\ell$. Though non-rigorous, these geometric and counting-based insights help illuminate the origin of the different growth behaviors.

\bibliography{Biblio}
\end{document}